\documentclass[11pt, reqno]{amsart}
\usepackage{textcmds} 
\usepackage{amsmath, amssymb, amsfonts, amstext, verbatim, amsthm, mathrsfs}
\usepackage{microtype}
\usepackage[all]{xy}
\usepackage[modulo]{lineno}
\usepackage{aliascnt}
\usepackage{enumitem}
\usepackage{xspace}
\usepackage{amsfonts}
\usepackage{amssymb}
\usepackage[centertags]{amsmath}
\usepackage{amsthm}
\usepackage{rotating}
\usepackage[margin=3.25cm]{geometry}
\usepackage{dsfont}
\usepackage{bm}
\usepackage{subfigure}
\usepackage{amsmath}
\usepackage{array}
\usepackage[all]{xy}
\usepackage{euscript}
\usepackage[T1]{fontenc}
\usepackage{mathbbol}
\usepackage{nccmath}
\usepackage{marginnote}
\usepackage{stmaryrd}
\usepackage{youngtab}
\usepackage{tikz}
\usepackage{blkarray}
\usepackage{chngcntr}

\usepackage{tikz}

\usepackage{graphics,graphicx}  

\let\oldtocsection=\tocsection
\let\oldtocsubsection=\tocsubsection

\renewcommand{\tocsection}[2]{\hspace{0em}\oldtocsection{#1}{#2}}
\renewcommand{\tocsubsection}[2]{\hspace{1em}\oldtocsubsection{#1}{#2}}

\usepackage[colorlinks=true,linkcolor=black,citecolor=blue,urlcolor=blue,citebordercolor={0 0 1},urlbordercolor={0 0 1},linkbordercolor={0 0 1}]{hyperref} 

\tikzset{node distance=3cm, auto}

\makeatletter
\def\@secnumfont{\bfseries}
\def\section{\@startsection{section}{1}%
  \z@{.7\linespacing\@plus\linespacing}{.5\linespacing}%
  {\normalfont\Large\bfseries}}
\def\subsection{\@startsection{subsection}{2}%
  \z@{.75\linespacing\@plus.7\linespacing}{-.5em}%
  {\normalfont\large\bfseries}}
  \def\subsubsection{\@startsection{subsubsection}{3}%
  \z@{.75\linespacing\@plus.7\linespacing}{-.5em}%
  {\normalfont\bfseries}}
\makeatother

\newtheorem{thm}{Theorem}[subsection]
\newtheorem{lemma}[thm]{Lemma}
\newtheorem{prop}[thm]{Proposition}
\newtheorem{cor}[thm]{Corollary}

\newtheorem{definition}[thm]{Definition}

\theoremstyle{remark}
\numberwithin{equation}{subsection} 

\newtheoremstyle{customremark}
{3pt}
{3pt}
{}
{}
{\bfseries}
{.}
{.5em}
{}
\theoremstyle{customremark}
\newtheorem{rmk_no_diamond}[thm]{Remark}
\newenvironment{rmk}{\begin{rmk_no_diamond} } {\hfill$\er$ \end{rmk_no_diamond}}
\newtheorem{example_no_diamond}[thm]{Example}
\newenvironment{example}{\begin{example_no_diamond} } {\hfill$\er$ \end{example_no_diamond}}

\newcommand{\ccirc}{{\circledcirc}}
\newcommand{\notccirc}{{\,\,{|\mskip-9.75mu\circledcirc}\,}}

\newcommand{\calX}{\mathcal{X}}

\newenvironment{itemlist}
   { \begin{list} {$\bullet$}
         { \setlength{\topsep}{.5ex}  \setlength{\itemsep}{.5ex} \setlength{\leftmargin}{2.5ex} } }
   { \end{list} }

\newcommand{\NI}{{\noindent}}

\newcommand{\Om}{{\Omega}}

\newcommand{\de}{{\delta}}

\newcommand{\ga}{{\gamma}}

\newcommand{\Ga}{{\Gamma}}
\newcommand{\io}{{\iota}}

\newcommand{\la}{{\lambda}}

\newcommand{\less} {{\smallsetminus}}
\newcommand{\p}{{\partial}}
\newcommand{\MS}{{\medskip}}

\newcommand{\er}{{\Diamond}}

\newcommand{\Z}{\mathbb{Z}}

\newcommand{\R}{\mathbb{R}}
\newcommand{\Q}{\mathbb{Q}}
\newcommand{\C}{\mathbb{C}}

\newcommand{\K}{\mathbb{K}}

\newcommand{\eps}{\varepsilon}

\newcommand{\calA}{\mathcal{A}}
\newcommand{\calL}{\mathcal{L}}
\newcommand{\calW}{\mathcal{W}}

\renewcommand{\bar}{\mathcal{B}}

\newcommand{\A}{\mathcal{A}}
\newcommand{\bdy}{\partial}

\newcommand{\Li}{\mathcal{L}_\infty}

\newcommand{\calM}{\mathcal{M}}

\newcommand{\wh}{\widehat}
\newcommand{\wt}{\widetilde}

\newcommand{\ovl}{\overline}
\newcommand{\ovll}[1]{\overline{\overline{#1}}}
\newcommand{\op}[1]{{\operatorname{#1}}}
\newcommand{\e}{\eps}

\newcommand{\std}{{\op{std}}}

\newcommand{\chlin}{\op{CH}_{\op{lin}}}

\newcommand{\lin}{{\op{lin}}}

\newcommand{\cz}{{\op{CZ}}}
\newcommand{\ind}{\op{ind}}

\newcommand{\Op}{\mathcal{O}p}
\newcommand{\nil}{\varnothing}
\newcommand{\sss}{\vspace{2.5 mm}}

\newcommand{\auglin}{\e_\lin}

\newcommand{\gapac}{\mathfrak{g}}

\newcommand{\bb}{\frak{b}}

\renewcommand{\lll}{\Langle}
\newcommand{\rrr}{\Rangle}

\newcommand{\sk}{{\op{sk}}}
\newcommand{\T}{\mathcal{T}}

\newcommand{\calC}{\mathcal{C}}
\newcommand{\calJ}{\mathcal{J}}

\newcommand{\gt}{\tilde{\gapac}}

\newcommand{\hooksymp}{\overset{s}\hookrightarrow}
\newcommand{\wind}{\op{wind}}

\newcommand{\formal}{\EuScript{F}}
\newcommand{\Jadm}{\mathcal{J}}

\newcommand{\fr}{\op{FR}}
\newcommand{\tX}{\wt{X}}
\newcommand{\tcalA}{\wt{\calA}}
\newcommand{\smx}{\oset[-0.5ex]{\frown}{\times}}
\newcommand{\tauhut}{{\tau_{\op{Hut}}}}

\newcommand{\Int}{\op{Int}\,}
\newcommand{\mb}{\op{MB}}

\newcommand{\calG}{\mathcal{G}}

\newcommand{\wmin}{w_{\op{min}}}

\newcommand{\delbar}{{\ovl{\partial}}}
\newcommand{\tu}{{\wt{u}}}
\newcommand{\tnabla}{{\wt{\nabla}}}
\newcommand{\tD}{{\wt{D}}}
\newcommand{\ssst}{\scriptscriptstyle}
\newcommand{\ver}{{\op{ver}}}
\newcommand{\hor}{{\op{hor}}}
\newcommand{\calB}{{\mathcal{B}}}
\newcommand{\calE}{{\mathcal{E}}}
\newcommand{\calT}{{\mathcal{T}}}
\renewcommand{\hom}{\op{Hom}}

\makeatletter
\newcommand{\dashover}[2][\mathop]{#1{\mathpalette\df@over{{\dashfill}{#2}}}}
\newcommand{\fillover}[2][\mathop]{#1{\mathpalette\df@over{{\solidfill}{#2}}}}
\newcommand{\df@over}[2]{\df@@over#1#2}
\newcommand\df@@over[3]{%
  \vbox{
    \offinterlineskip
    \ialign{##\cr
      #2{#1}\cr
      \noalign{\kern1pt}
      $\m@th#1#3$\cr
    }
  }%
}
\newcommand{\dashfill}[1]{%
  \kern-.5pt
  \xleaders\hbox{\kern.5pt\vrule height.4pt width \dash@width{#1}\kern.5pt}\hfill
  \kern-.5pt
}
\newcommand{\dash@width}[1]{%
  \ifx#1\displaystyle
    2pt
  \else
    \ifx#1\textstyle
      1.5pt
    \else
      \ifx#1\scriptstyle
        1.25pt
      \else
        \ifx#1\scriptscriptstyle
          1pt
        \fi
      \fi
    \fi
  \fi
}
\newcommand{\solidfill}[1]{\leaders\hrule\hfill}

\newcommand{\oset}[3][0ex]{%
  \mathrel{\mathop{#3}\limits^{
    \vbox to#1{\kern-2\ex@
    \hbox{$\scriptstyle#2$}\vss}}}}

\makeatother

\date{\today}

\begin{document}

\title{Symplectic capacities, unperturbed curves, and convex toric domains}
\date{\today}

\begin{abstract}
We use explicit pseudoholomorphic curve techniques (without virtual perturbations) to define a sequence of symplectic capacities analogous to those defined recently by the second named author using symplectic field theory. We then compute these capacities for all four-dimensional convex toric domains. This gives various new obstructions to stabilized symplectic embedding problems which are sometimes sharp.
\end{abstract}

\author{Dusa McDuff}
\author{Kyler Siegel}\thanks{K.S. is partially supported by NSF grant DMS-2105578 and a visiting membership at the Institute for Advanced Study.}

\maketitle

\tableofcontents

\pagestyle{plain}

\section{Introduction}

\subsection{Overview}

Symplectic capacities have long played an important role in symplectic geometry, providing a systematic tool for studying nonsqueezing phenomena.
Let us mention here just two prominent sequences of symplectic capacities: the Ekeland--Hofer capacities \cite{EH1,EH2} and the embedded contact homology (ECH) capacities \cite{Hutchings_quantitative_ECH}. The former are defined in any dimension and they provide obstructions which can be viewed as refinements of Gromov's celebrated nonsqueezing theorem \cite{gromov1985pseudo}. The latter are defined only in dimension four, but they often give very strong obstructions, e.g. they give sharp obstructions for symplectic embeddings between four-dimensional ellipsoids.

Higher dimensional symplectic embeddings remain rather poorly understood, but there has been considerable recent interest \cite{HK,CGH,Ghost,Mint,HSC,chscI,irvine2019stabilized} in so-called ``stabilized symplectic embedding problems'', in which one studies symplectic embeddings of the form $X \times \C^N \hooksymp X' \times \C^N$ for four-dimensional Liouville domains $X,X'$ and $N \in \Z_{\geq 1}$.
In order to systematize and generalize these results, the second named author introduced in \cite{HSC} a sequence of symplectic capacities $\gapac_1,\gapac_2,\gapac_3,\dots$ which are ``stable'' in the sense that $\gapac_k(X \times \C^N) = \gapac_k(X)$ for any Liouville domain $X$ and $k,N \in \Z_{\geq 1}$.
These capacities are defined using symplectic field theory (SFT), more specifically the (chain level) filtered $\Li$ structure on linearized contact homology, and their definition also involves curves satisfying local tangency constraints.
As a proof of concept, \cite{HSC} shows that these capacities perform quite well in toy problems, for instance they recover the sharp obstructions from \cite{Mint} and they often outperform the Ekeland--Hofer capacities.
In fact, the capacities $\gapac_1,\gapac_2,\gapac_3,\dots$ are a specialization of a more general family of capacities $\{\gapac_{\bb}\}$ which are expected to give sharp obstructions to the stabilized ellipsoid embedding problem.

However, two broad questions naturally become apparent:
\begin{enumerate}
\item What is the role of symplectic field theory? Namely, it is known that SFT typically requires virtually perturbing moduli spaces of pseudoholomorphic curves, and yet ultimately all of the data of $\gapac_k(X)$ should be carried by honest pseudoholomorphic curves in $\wh{X}$ and $\R \times \bdy X$, so does one really need the full SFT package?\footnote{As outlined in \cite[\S1]{HSC}, we also expect an alternative definition of $\gapac_k$ using ($S^1$-equivariant) Floer theory instead of symplectic field theory. Since this involves Hamiltonian perturbations and many associated choices, it is also quite difficult to compute directly from the definition.}
\item How does one actually compute $\gapac_1,\gapac_2,\gapac_3,\dots$ for Liouville domains of interest? Note that even computing $\gapac_k$ for a four-dimensional ellipsoid is a nontrivial problem.
\end{enumerate}
\NI Note that these questions are coupled, since a concrete answer to (1) could open up new direct avenues for computations as in (2).

The primary purpose of this paper is to address both of these questions. In short:
\begin{enumerate}
\item We give an ersatz definition of $\gapac_k$, denoted by $\gt_k$, which is simple and explicit and does not require any virtual perturbations.
\item We compute (or at least reduce to elementary combinatorics) $\gt_k$ for all four-dimensional convex toric domains. This gives a large family of examples which includes ellipsoids and polydisks as special cases.
\end{enumerate}
Combining these, one can directly extract many new symplectic embedding obstructions. As an illustration, the recent work \cite{cristofaro2021higher} applies our computations for ellipsoids and polydisks in order to obstruct various stabilized symplectic embeddings between these. Remarkably, these obstructions are often sharp, at least when certain aspect ratios are integral;  see Example~\ref{ex:concrete} and Remark~\ref{rmk:poly_formula}.

\subsection{Statement of main results}

We now describe our results in more detail. 
In \S\ref{sec:capacity}, we define the capacity $\gt_k(M)$ for all symplectic manifolds $M$ and $k \in \Z_{\geq 1}$.
Roughly, if $X$ is a Liouville domain with nondegenerate contact boundary, then $\gt_k(X)$ is the maximum over all suitable almost complex structures $J$ of the minimum energy of any asymptotically cylindrical rational $J$-holomorphic curve in $\wh{X}$ which satisfies a local tangency constraint $\lll \T^{(k)}p\rrr$.
The latter means that the curve has contact order $k$ (or equivalently tangency order $k-1$) to a chosen local divisor $D$ defined near a point $p \in X$.
Note that we do not require the curves entering into the definition of $\gt_k(X)$ to be regular or even index zero.
This definition of $\gt_k(X)$ is extended to $\gt_k(M)$ for $M$ an arbitrary symplectic manifold by taking a supremum over all Liouville domains which symplectically embed into $M$.\footnote{After a first draft of this paper was completed, the authors learned from G. Mikhalkin about independent work defining a similar capacity directly for all symplectic manifolds using an even broader class of almost complex structures and pseudoholomorphic curves. It seems likely that these two definitions are equivalent, but they may have slightly different realms of utility.}

\begin{rmk} 
In the special case of the first capacity $\gt_1$, our definition essentially coincides with Gromov's original definition of  ``symplectic width'' 
via a maxi-min procedure - see \cite[\S4.1]{soft_and_hard}.
\end{rmk}

The following summarizes some of the key properties of $\gt_k$: 
{}
\begin{thm}\label{thm:gt}
For each $k \in \Z_{\geq 1}$, $\gt_k$ is independent of the choice of local divisor and is a symplectomorphism invariant. It satisfies the following properties:
\begin{itemize}
\item
{\bf Scaling:} it scales like area, i.e. $\gt_k(M,\mu \omega) = \mu\gt_k(M,\omega)$ for any symplectic manifold $(M,\omega)$ and $\mu \in \R_{> 0}$.
\item
{\bf Nondecreasing:} we have $\gt_1(M) \leq \gt_2(M) \leq \gt_3(M) \leq \cdots$ for any symplectic manifold $M$.

\item
{\bf Subadditivity:} we have $\gt_{i+j}(M) \leq \gt_i(M) + \gt_j(M)$ for any $i,j \in \Z_{\geq 1}$.

\item 
{\bf Symplectic embedding monotonicity:} It is monotone under equidimensional symplectic embeddings, i.e. $M \hooksymp M'$ implies $\gt_k(M) \leq \gt_k(M')$ for any symplectic manifolds $M,M'$.
\item 
{\bf Closed curve upper bound:} If $(M,\omega)$ is a closed semipositive symplectic manifold satisfying $N_{M,A}\lll \T^{(k)}p\rrr \neq 0$ for some $A \in H_2(M)$, 
then we have $\gt_k(M) \leq [\omega] \cdot A$.
\item
{\bf Stabilization:} For any Liouville domain $X$ we have $\gt_k(X \times B^{2}(c)) = \gt_k(X)$ for any $c \geq \gt_k(X)$, provided that the hypotheses of Proposition~\ref{prop:stab_ub} are satisfied (this holds e.g. for $X$ any four-dimensional convex toric domain).
\end{itemize}
\end{thm}
\NI In the penultimate point, $N_{M,A}\lll \T^{(k)}p\rrr$ denotes the Gromov--Witten type invariant which counts closed rational pseudoholomorphic curves in $M$ in homology class $A$ satisfying the local tangency constraint $\lll \T^{(k)}p\rrr$, as defined in \cite{McDuffSiegel_counting}.
Also, $B^{2}(c)$ denotes the closed two-ball of area $c$ (i.e. radius $\sqrt{c/\pi}$), equipped with its standard symplectic form. 
For more detailed explanations and proofs, see \S\ref{sec:prelim} and \S\ref{sec:capacity}.

\begin{rmk}[Stabilization hypotheses]
The hypotheses of Proposition~\ref{prop:stab_ub} roughly amount to the assumption that $\gt_k(X)$ is represented by a moduli space of curves which is sufficiently robust that it cannot degenerate in generic $1$-parameter families.
When this holds, we can iteratively stabilize to obtain $\gt_k(X \times B^2(c) \times \cdots \times B^2(c)) = \gt_k(X)$ for $c \geq \gt_k(X)$, and in particular we have $\gt_k(X \times \C^N) = \gt_k(X)$ for $N \in \Z_{\geq 1}$.
Compared with $\gapac_k$, the extra hypotheses in the stabilization property is one place where we ``pay the price'' for such a simple definition of $\gt_k$, although we do not know whether the extra hypotheses is truly essential.
\end{rmk}

\begin{rmk}[Relationship with $\gapac_k$]
As we explain in \S\ref{subsec:comp_with_SFT}, we must have $\gt_k(X) = \gapac_k(X)$ whenever $X$ is a Liouville domain satisfying the hypotheses of Proposition~\ref{prop:stab_ub}.
In particular, this is the case for all four-dimensional convex toric domains, and we are not aware of any examples with $\gt_k(X) \neq \gapac_k(X)$.
\end{rmk}

\begin{rmk}[Relationship with Gutt--Hutchings capacities]\label{rmk:reln_GH}
In \S\ref{subsec:basic_props}, we define (following \cite{HSC}) a refined family of capacities $\gt_k^{\leq \ell}$ for $k,\ell \in \Z_{\geq 1}$, using the same prescription as for $\gt_k$ except that we now only allow curves having at most $\ell$ positive ends.
Note that the case $\ell = \infty$ recovers $\gt_k = \gt_k^{\leq \infty}$.
The capacities $\{\gt_k^{\leq \ell}\}$ satisfy most of the properties in Theorem~\ref{thm:gt}, except that the closed curve upper bound no longer holds, and monotonicity for $\gt_k^{\leq \ell}$ only holds for \textbf{generalized Liouville embeddings}, i.e. smooth embeddings $\iota: (X,\la) \hookrightarrow (X',\la')$ of equidimensional Liouville domains such that the closed $1$-form $(\iota^*(\la') - \la)|_{\bdy X}$ is exact (c.f. \cite[\S1.4]{Gutt-Hu}).
In \S\ref{subsec:GH}, we show that, at least for four-dimensional convex toric domains, the $\ell = 1$ specialization $\gt_k^{\leq 1}$ coincides with the $k$th Gutt--Hutchings capacity $c_k^{\op{GH}}$ \cite{Gutt-Hu}. The latter is in turn known to agree with the $k$th Ekeland--Hofer capacity  $c_k^{\op{EH}}$ in all examples where both are computed, e.g. ellipsoids and polydisks.
\end{rmk}

\begin{rmk} [Nondecreasing property]
Curiously, for the analogous SFT capacities the nondecreasing property $\gapac_1 \leq \gapac_2 \leq \gapac_3 \leq \cdots$ is not at all obvious from the definition. 
\end{rmk}
\begin{rmk}[Generalizations]
The approach taken in this paper to define $\{\gt_k\}$ naturally generalizes to define various other families of capacities, e.g. by replacing the local tangency constraint $\lll \T^{(k)}p\rrr$ with $k$ generic point constraints, and/or by allowing curves of higher genus.
In this spirit, the very recent preprint \cite{hutchings2022} adapts our approach to define (without relying on Seiberg--Witten theory) a sequence of four-dimensional capacities which agree in many cases with the ECH capacities.
\end{rmk}

\sss

With the capacities $\gt_1,\gt_2,\gt_3,\dots$ at hand, we turn to computations.
Given a compact convex domain $\Omega \subset \R^n$, 
put $X_{\Omega} := \mu^{-1}(\Omega)$, where $\mu: \C^n \rightarrow \R^n_{\geq 0}$ is given by
$$
\mu(z_1,\dots,z_n) = (\pi |z_1|^2,\dots,\pi|z_n|^2).
$$
Define $||-||_{\Omega}^*: \R^n \rightarrow \R$ by $||\vec{v}||_{\Omega}^* := \max\limits_{\vec{w} \in \Omega}\langle \vec{v},\vec{w}\rangle$, where $\langle -,-\rangle$ denotes the standard dot product. 
Note that if $\bdy\Omega$ is smooth, then the maximizer $\vec{w}$ lies in $\bdy \Omega$ and is such that the hyperplane through $\vec{w}$ normal to $\vec{v}$ is tangent to $\bdy\Omega$. 
If $\Omega$ contains the origin in its interior, then $||-||_{\Omega}^*$ is a (non-symmetric) norm, dual to the norm having $\Omega$ as its unit ball. Otherwise, $||-||_{\Omega}^*$ is not generally nondegenerate or even nonnegative, although it is still convenient to treat it like a norm.
Recall that $X_{\Omega}$ is a ``convex toric domain'' if the symmetrization of $\Omega$ about the axes is itself convex (see \S\ref{subsec:fr} for more details).

\begin{thm}\label{thm:main_comp}
Let $X_\Omega$ be a four-dimensional convex toric domain.
For $k \in \Z_{\geq 1}$, we have
\begin{align}\label{eq:gt_orig}
\gt_k(X_\Omega) = \min\sum_{s=1}^q ||(i_s,j_s)||_{\Omega}^*,
\end{align}
where the minimization is over all $(i_1,j_1),\dots,(i_q,j_q) \in \Z_{\geq 0}^2 \setminus \{(0,0)\}$
such that
\begin{itemize}
\item
$\sum_{s=1}^q(i_s + j_s) + q-1 = k$
\item
if $q \geq 2$, then $(i_1,\dots,i_q) \neq (0,\dots,0)$ and $(j_1,\dots,j_q) \neq (0,\dots,0)$. 
\end{itemize}
\end{thm}

Using results from \S\ref{sec:rounding}, we have the following appealing reformulation, which we prove at the end of \S\ref{subsec:min_words}. 
If $P \subset \R^2$ is a convex lattice polygon, i.e. a convex polygon such that each vertex lies at an integer lattice point, let $\ell_{\Omega}(\bdy P)$ denote the length of its boundary as measured by $||-||_{\Omega}^*$, and let $|\bdy P \cap \Z^2|$ denote the number of lattice points along the boundary.
Here we allow the degenerate case where $P$ is a line segment, in which case by definition $\bdy P = P$.
Note that $\ell_{\Omega}(\bdy P)$ is unaffected if we translate $\Omega$ so that it contains the origin in its interior, after which $||-||_{\Omega}^*$ becomes nondegenerate.

\begin{cor}\label{cor:latt_form}
For $X_{\Omega}$ a four-dimensional convex toric domain and $k \in \Z_{\geq 1}$, we have:
\begin{align}\label{eq:latt_form}
\gt_k(X_{\Omega}) = \min\left\{\ell_{\Omega}(\bdy P)  \;\middle\vert\;\begin{array}{l} P \subset \R^2 \text{ convex lattice polygon},\\ |\bdy P \cap \Z^2| = k+1\end{array}\right\}.
\end{align}
\end{cor}

\begin{rmk}

\NI (i)  
The $k$th ECH capacity $c_k^{\op{ECH}}(X_{\Omega})$ is given by the exact same formula except that we replace $|\bdy P \cap \Z^2|$ with $|P \cap \Z^2|$, i.e. the number of lattice points in both the interior and boundary of $P$ (see \cite{Hutchings_quantitative_ECH}).
Under the correspondence between lattice polygons and generators, $|\bdy P \cap \Z^2|$ corresponds to the (half) Fredholm index, whereas $|P \cap \Z^2|$ corresponds to the (half) ECH index. 
It is interesting to ask whether Corollary~\ref{cor:latt_form} holds for more general domains $\Omega \subset \R^2$. One can also ask about extensions to higher dimensions, with lattice polygons in $\R^2$ replaced by lattice paths in $\R^n$.

\MS

\NI (ii)
Note that Corollary~\ref{cor:latt_form} involves arbitrary lattice points, whereas Theorem~\ref{thm:main_comp} involves only nonnegative ones.
Conceptually this mirrors the fact that $X_\Omega$ has the same values for $\gt_k$ as its associated ``free toric domain'' $\mathbb{T}^2 \times \Omega$, thanks to the ``Traynor trick'' (see e.g. \cite{landry2015symplectic}).
\MS

\NI (iii)
Closely related formulas appear in the recent work \cite{chaidez2021lattice}. In particular, \cite[Cor. 1]{chaidez2021lattice} computes $\gapac_k(X_\Omega)$ under the additional assumption that the lengths of $\Omega$ along the $x$ and $y$ axes agree, which holds e.g. if $X$ is the round ball $B^4(c)$ or the cube $B^2(c) \times B^2(c)$.
Whereas our upper bounds come from curves constructed via the ECH cobordism map and iterated obstruction bundle gluing (see \S\ref{sec:constructing_curves}), the upper bounds in \cite{chaidez2021lattice} come from cocharacter curves in (possibly singular) closed toric surfaces. 
\MS

\NI (iv)
The work \cite{chscI} offers another combinatorial computation of $\gapac_k(X_\Omega)$ for any four-dimensional convex toric domain $X_\Omega$, and in fact it also computes the full family of capacities $\{\gapac_\bb(X_\Omega)\}$. However, since that framework involves a nontrivial recursive algorithm, it is not clear how to use it to extract the above formulas.
\end{rmk}

\subsection{Examples and applications}

 In \S\ref{subsec:min_words} we significantly simplify the combinatorial optimization problem involved in Theorem~\ref{thm:main_comp} by showing that there are only a few possibilities for the minimizers. 
 Indeed, Corollary~\ref{cor:weak_perm} implies the following simplification of Theorem~\ref{thm:main_comp}:
 
 \begin{cor}\label{cor:main_thm_4_cases}
Let $X_\Om$ be a four-dimensional convex toric domain as in Theorem~\ref{thm:main_comp}, and assume that $\Om$ has sides of length $a,b$ along the $x$ and $y$ axes respectively, with $a \geq b$.
Then there is a minimizer $(i_1,j_1),\dots,(i_q,j_q)$ taking one of the following forms:
 \begin{enumerate}
 \item[{\rm (i)}]
 $(0,1)^{\times i} \times (1,1)^{\times j}$ for $i \geq 0$, $j \geq 1$
 \item[{\rm (ii)}]
 $(0,1)^{\times i} \times (1,s)$ for $i \geq 0$ and $s \geq 2$
 \item[{\rm (iii)}]
 $(0,1)^{\times i} \times (1,0)$ for $i \geq 1$
 \item[{\rm (iv)}]
 $(0,s)$ for $s \geq 1$.
 \end{enumerate}
 \end{cor}
 \NI This formulation is particularly useful for extracting closed-form expressions for $\gt_k$ in various families of examples, as in the following results.
 
Let $$
E(a_1,a_2) := \{(z_1,z_2) \in \C^2\;|\; \tfrac{1}{a_1}\pi|z_1|^2 + \tfrac{1}{a_2}\pi|z_2|^2 \leq 1\}
$$
 denote the ellipsoid with area factors $a_1,a_2$.
Up to scaling and symplectomorphism, we can assume that $a_2 = 1$ and $a_1 \geq 1$.
 \begin{thm}\label{thm:ell_comp}
 \hfill
 \begin{enumerate}
 \item[{\rm (i)}] For $1 \leq a \leq 3/2$, we have 
\begin{align}
\gt_k(E(a,1)) = 
\begin{cases}
1 + ia&\text{ for } k = 1+3i\text{ with } i \geq 0\\
a+ia&\text{ for } k = 2+3i\text{ with } i \geq 0\\
2 + ia& \text{ for } k = 3+3i\text{ with } i \geq 0.
\end{cases} \label{eq:ellip_comp3}
\end{align}

\item[{\rm (ii)}]  For $a > 3/2$, we have
\begin{align}
\gt_k(E(a,1)) = 
\begin{cases}
k&\text{ for } 1 \leq k \leq \lfloor a \rfloor\\
a + i&\text{ for } k = \lceil a \rceil + 2i \text{ with } i \geq 0\\
\lceil a \rceil + i&\text{ for } k = \lceil a \rceil + 2i+1 \text{ with } i \geq 0.
\end{cases} \label{eq:ellip_comp2}
\end{align}
\end{enumerate}
\end{thm}

\begin{example}\label{ex:concrete}

We illustrate Theorem~\ref{thm:ell_comp} with a simple embedding example which is a special case of \cite[Thm. 1.1]{cristofaro2021higher}.
The first few $\gt_k$ capacities are:
\[
\begin{tabular}{l|l|l|l|l|l|l|l|l|l|l|l|l}
$k$ & 1 & 2 & 3 & 4 & 5 & 6 & 7 & 8 & 9 & 10 & 11 & 12 \\ \hline
$\gt_k(E(1,7))$ & 1 & 2 & 3 & 4 & 5 & 6 & 7 & 7 & 8 & 8 & 9 & 9 \\ \hline
$\gt_k(E(1,2))$ & 1 & 2 & 2 & 3 & 3 & 4 & 4 & 5 & 5 & 6 & 6 & 7
\end{tabular}
\]
This gives a lower bound for stabilized embeddings $E(1,7) \times \C^N \hooksymp \mu \cdot E(1,2) \times \C^N$ (with $N \geq 1$) of $\mu \geq 7/4$.
By \cite[Cor. 3.4]{cristofaro2021higher} this is optimal, i.e. there exists a stabilized symplectic embedding realizing this lower bound.
In particular, this outperforms the Gutt--Hutchings (or Ekeland--Hofer) capacities, the first few of which are: 
\[
\begin{tabular}{l|l|l|l|l|l|l|l|l|l|l|l|l}
$k$ & 1 & 2 & 3 & 4 & 5 & 6 & 7 & 8 & 9 & 10 & 11 & 12 \\ \hline
$c_k^{\op{GH}}(E(1,7))$ & 1 & 2 & 3 & 4 & 5 & 6 & 7 & 7 & 8 & 9 & 10 & 11\\ \hline
$c_k^{\op{GH}}(E(1,2))$ & 1 & 2 & 2 & 3 & 4 & 4 & 5 & 6 & 6 & 7 & 8 & 8
\end{tabular},
\]
and in fact the best bound obtained by the full infinite sequence is $\mu \geq 3/2$.
By contrast, the ECH capacities give a stronger lower bound, which evidently cannot stabilize. Indeed, we have:
\[
\begin{tabular}{l|l|l|l|l|l|l|l|l|l|l|l|l}
$k$ & 1 & 2 & 3 & 4 & 5 & 6 & 7 & 8 & 9 & 10 & 11 & 12 \\ \hline
$c_k^{\op{ECH}}(E(1,7))$ & 1 & 2 & 3 & 4 & 5 & 6 & 7 & 7 & 8 & 8 & 9 & 9\\ \hline
$c_k^{\op{ECH}}(E(1,2))$ & 1 & 2 & 2 & 3 & 3 & 4 & 4 & 4 & 5 & 5 & 5 & 6
\end{tabular},
\]
giving the lower bound $\mu \geq 9/5 > 7/4$ for the unstabilized problem $E(1,7) \hooksymp \mu \cdot E(1,2)$.
Note that the volume bound is $\mu \geq \sqrt{7/2} \approx 1.87 > 9/5$, and this is necessarily recovered by the full sequence of ECH capacities since these are known to give sharp obstructions for four-dimensional ellipsoid embeddings (and also their asymptotics recover the volume). 
\end{example}

Now let $P(a_1,a_2) := B^2(a_1) \times B^2(a_2)$ denote the polydisk with area factors $a_1,a_2$.
Again, without loss of generality we can assume $a_2 = 1$ and $a_1 \geq 1$.
\begin{thm}\label{thm:poly_comp}
For $k \in \Z_{\geq 1}$ and $a\geq 1$ we have
\begin{align}\label{eq:poly_formula}
\gt_k(P(a,1)) = \min(k,a+\lceil \tfrac{k-1}{2} \rceil).
\end{align}
\end{thm}

\begin{rmk}[Sharp obstructions] \label{rmk:poly_formula}
    Example~\ref{ex:concrete} generalizes as follows.
By complementing Theorem~\ref{thm:ell_comp} with explicit embedding constructions, 
 \cite[Thm. 1.1]{cristofaro2021higher} shows that the capacities $\{\gt_k\}$ are sharp for embeddings of the form
 $E(a,1) \times \C^N \hooksymp \mu \cdot E(b,1) \times \C^N$ with $a \geq b+1 \geq 3$ integers of the opposite parity, and $\mu \in \R_{> 0}, N \in \Z_{\geq 1}$, and such an embedding exists if and only if $\mu \geq \tfrac{2a}{a+b-1}$.
 Similarly, \cite[Thm. 1.3]{cristofaro2021higher} shows that the capacities $\{\gt_k\}$ are sharp for embeddings of the form $$
 E(a,1) \times \C^N \hooksymp \mu \cdot P(b,1) \times \C^N
 $$
  with $b \in \R_{\geq 1}$ (not necessarily an integer), $a \geq 2b-1$ any odd integer, and $\mu \in \R_{> 0}, N \in \Z_{\geq 1}$, and such an embedding exists if and only if $\mu \geq \tfrac{2a}{a+2b-1}$.

For embeddings of the form $E(a,1)\times \C^N \hooksymp \mu \cdot B^4(b) \times \C^N$ with $N \in \Z_{\geq 1}$, it was observed in \cite{HSC} that the capacities $\{\gapac_k\}$ (and hence also $\{\gt_k\}$ by the results of this paper) give sharp obstructions when $a \in 3\Z_{\geq 1} - 1$.
On the other hand, for all other $a \in \R_{> 1}$ we do not expect optimal obstructions from the capacities $\{\gapac_k\}$, but rather from the full family $\{\gapac_\bb\}$ (see the discussion at the end of \cite[\S6.3]{HSC}). It is natural to ask whether a ``naive'' analogue $\{\wt{\gapac}_\bb\}$ could be defined and computed in the spirit of this paper.
\end{rmk}

\begin{rmk}
The formulas \eqref{eq:poly_formula} also appeared for $\gapac_k$ in \cite{HSC} in the case of odd $k$, based on a slightly different computational framework.
 \end{rmk}

\sss

Next, we consider a more complicated family of examples.
Given $c \geq 1$ and $(a,b) \in \R_{> 0}^2$, let $Q(a,b,c) \subset \R_{\geq 0}^2$ denote the quadrilateral with vertices $(0,0),(0,1),(c,0),(a,b)$.
We note that $X_{Q(a,b,c)} \subset \C^2$ is a convex toric domain if and only if we have $a \leq c$, $b \leq 1$, and $a+bc \geq c$. 
The next result gives the formula for $\gt_k$ when $\max(a+b,c)\leq 2$;  the case $\max(a+b,c)>2$ is similar to case (ii) below (see Remark~\ref{rmk:abc}).

\begin{thm}\label{thm:quad}
Let $X: = X_{Q(a,b,c)}$ be a convex toric domain for some $c \geq 1$ and $(a,b) \in \R_{> 0}^2$, and put $M := \max(a+b,c)$.

\begin{enumerate}
\item[{\rm (i)}] For $M \leq3/2$, we have:  

\begin{align}\label{eq:quad_1}
\begin{split}
&\gt_1(X) = 1, \quad   \gt_2(X) = M, \quad 
\gt_3(X) = \min\bigl(\max(2,a+2b),\ 1+c\bigr),\\ 
& \gt_4 = 1 + M, \quad \gt_5 = 2M, \quad \gt_6 = 2+M,\\
& \gt_{k+3}(X) = \gt_k(X) + M, \quad k\ge 4.
\end{split}
\end{align}

\item[{\rm (ii)}]  For $3/2\leq M\leq 2$, then $\gt_k(X)$ is as above for $k\le 4$, and
\begin{align}\label{eq:quad_2}
\begin{split}
&\gt_5(X) = \min\bigl(\max(3,1+ a+2b),\ 2 M, \ 2+c\bigr),\\
& \gt_{k+2}(X) = 1 + \gt_{k}(X) , \quad k\ge 4.
\end{split}
\end{align}

\end{enumerate}
\end{thm}


\sss

For our last family of examples, take $p \in \R_{\geq 1} \cup \{\infty\}$, and consider the $L^p$ norm $||-||_p$ defined by $||(x,y)||_{p} := 
(x^p + y^p)^{1/p}$,
and put $$\Omega_p := \{(x,y) \in \R^2_{\geq 0}\;|\; ||(x,y)||_p \leq 1\}.$$
Note that $\Omega_1$ is the right triangle with vertices $(0,0),(1,0),(0,1)$ and $\Omega_{\infty}$ is the square with vertices $(0,0),(1,0),(0,1),(1,1)$, i.e. the corresponding family of convex toric domains $\{X_{\Omega_p}\}$ interpolates between the round ball and the cube.
Also, note that for $(x,y) \in \R_{\geq 0}^2$, we have $$||(x,y)||_{\Omega_p}^* = ||(x,y)||_q,$$ where $q \in \R_{\geq 1} \cup \{\infty\}$ is such that $\tfrac{1}{p} + \tfrac{1}{q} = 1$.
\begin{thm}\label{thm:Lp}
\hfill

\begin{enumerate}
\item[{\rm (i)}] For $p \leq \frac{\ln(2)}{\ln(4/3)}$ we have 
\begin{align}\label{eq:Lp_1}
\gt_k(X_{\Omega_p}) = 
\begin{cases}
1 + i\sqrt[q]{2}&\text{ for } k = 1+3i\text{ with } i \geq 0\\
(i+1)\sqrt[q]{2}&\text{ for } k = 2+3i\text{ with } i \geq 0\\
2 + i\sqrt[q]{2}& \text{ for } k = 3+3i\text{ with } i \geq 0.
\end{cases} 
\end{align}
\item[{\rm (ii)}] For $p > \frac{\ln(2)}{\ln(4/3)}$ we have
\begin{align}\label{eq:Lp_2}
\gt_k(X_{\Omega_p}) = 
\begin{cases}
1 + i &\text{ for } k = 1+2i\text{ with } i \geq 0\\
\sqrt[q]{2} + i&\text{ for } k = 2+ 2i\text{ with } i \geq 0.\\
\end{cases} 
\end{align}

\end{enumerate}
\end{thm}

\begin{rmk}
Incidentally, note that we have $\gt_k(X_{\Om_p}) = \gt_k(E(1,\sqrt[q]{2}))$. 
Moreover, one can show using Corollary~\ref{cor:main_thm_4_cases} that 
the capacities $\gt_k(X_\Om)$ of any four-dimensional convex toric domain 
normalized as in Corollary~\ref{cor:main_thm_4_cases}
are 
eventually either  $2$-periodic or $3$-periodic in $k$, depending on which of 
$3 ||(0,1)||_\Om^*,\ 2 ||(1,1)||_\Om^*$ 
is 
smaller.  Intuitively, domains which are ``rounder'' have $3$-periodic capacities while domains which are ``skinnier'' have $2$-periodic ones.
\end{rmk}

\begin{example}
For concreteness let us flesh out a simple implication of Theorem~\ref{thm:Lp} for e.g. the symplectic embedding problem $E(1,5) \times \C^N \hooksymp \mu \cdot X_{\Om_2} \times \C^N$ with $N \in \Z_{\geq 0}$.
Using \cite[Thm. 1.6]{Gutt-Hu} (see also \cite{kerman2021symplectic}), it is easy to check that we have
\begin{align*}
c_k^{\op{GH}}(X_{\Om_2}) = 
\begin{cases}
  \sqrt{\tfrac{1}{2}k^2} & \text{k even}\\
    \sqrt{\tfrac{1}{2}(k^2+1)} & \text{k odd},
\end{cases}
\end{align*}
i.e. 
\[
\begin{tabular}{l|l|l|l|l|l|l|l|l|l|l|l|l}
$k$ & $1$ & $2$ & $3$ & $4$ & $5$ & $6$ & $7$ & $8$ & $9$ & $10$ & $11$ & $12$ \\ \hline
$c_k^{\op{GH}}(E(1,5))$ & $1$ & $2$ & $3$ & $4$ & $5$ & $5$ & $6$ & $7$ & $8$ & $9$ & $10$ & $10$ \\ \hline
$c_k^{\op{GH}}(X_{\Om_2})$ & $1$ & $\sqrt{2}$ & $\sqrt{5}$ & $2\sqrt{2}$ & $\sqrt{13}$ & $3\sqrt{2}$ & $5$ & $4\sqrt{2}$ & $\sqrt{41}$ & $5\sqrt{2}$ & $\sqrt{61}$ & $6\sqrt{2}$
\end{tabular},
\]
and the capacities $\{c_k^\op{GH}\}$ give the lower bound 
$\mu \geq 2/\sqrt{2} \approx 1.414$.
Meanwhile, we have:
\[
\begin{tabular}{l|l|l|l|l|l|l|l|l|l|l|l}
$k$ & $1$ & $2$ & $3$ & $4$ & $5$ & $6$ & $7$ & $8$ & $9$ & $10$ & $11$ \\ \hline
$\gt_k(E(1,5))$ & $1$ & $2$ & $3$ & $4$ & $5$ & $5$ & $6$ & $6$ & $7$ & $7$ & $8$ \\ \hline
$\gt_k(X_{\Om_2})$ & $1$ & $\sqrt{2}$ & $2$ & $1+\sqrt{2}$ & $2\sqrt{2}$ & $2+\sqrt{2}$ & $1+2\sqrt{2}$ & $3\sqrt{2}$ & $2+2\sqrt{2}$ & $1+3\sqrt{2}$ & $4\sqrt{2}$ 
\end{tabular}
\]
and the $\{\gt_k\}$ capacities give the lower bound $\mu \geq 5/(2\sqrt{2}) \approx 1.768$.
\end{example}

\sss

We end this introduction with a brief outline of the proof of Theorem~\ref{thm:main_comp}, deferring the reader to the body of the paper for the details.
Firstly, as in \cite{chscI,chscII}, we ``fully round'' our convex toric domain.
This is a small perturbation and so leaves $\gt_k$ essentially unaffected, while it standardizes the Reeb dynamics on the boundary. Next, we obtain a lower bound on $\gt_k$ by mostly action and index considerations, with the second condition in Theorem~\ref{thm:main_comp} coming from the relative adjunction formula and writhe bounds. To obtain a corresponding upper bound, we first study the combinatorial optimization problem in Theorem~\ref{thm:main_comp} more carefully and arrive at the simplifications described in \S\ref{subsec:min_words}. We then inductively construct a curve for each minimizer. The base cases are cylinders or pairs of pants which we produce using the ECH cobordism map, while the inductive step is based on an iterated application of obstruction bundle gluing based on the work of Hutchings--Taubes.

\section{Preliminaries on pseudoholomorphic curves}\label{sec:prelim}

The main purpose of this section is to briefly recall some requisite background on pseudoholomorphic curves and to establish notation, conventions, and terminology for the rest of the paper. In \S\ref{subsec:punc_curves} we discuss moduli spaces of punctured pseudoholomorphic curves in symplectic cobordisms. In \S\ref{subsec:loc_tang_and_E_sk} we recall the notion of local tangency constraints and the equivalence with skinny ellipsoidal constraints as in \cite{McDuffSiegel_counting}.
In \S\ref{subsec:formal_curves} we introduce the notion of formal curves, which provides a convenient language and bookkeeping tool in SFT compactness arguments.
Lastly, in \S\ref{subsec:pert_inv} we discuss the extent to which our moduli spaces persist in $1$-parameter families, and we introduce the notion of ``formal perturbation invariance'' which will be particularly relevant for us.
 
\subsection{Asymptotically cylindrical curves and their moduli}\label{subsec:punc_curves}

Our exposition in this subsection will be somewhat brief; we refer the reader to e.g. \cite{wendl_SFT_notes,BEHWZ} for more details.

\subsubsection{Symplectic and contact manifolds}
Recall that a {\bf Liouville cobordism} $(X,\la)$ is a compact manifold-with-boundary $X$, equipped with a one-form $\la$ whose exterior derivative $\omega := d\la$ is symplectic, and whose restriction to $\bdy X$ is a contact form. 
We have a natural decomposition $\bdy X = \bdy^+ X \sqcup \bdy^-X$, where $\la|_{\bdy^+X}$ is a positive contact form and $\la|_{\bdy^-X}$ is a negative contact form.
When no confusion should arise, we will typically suppress $\la$ from the notation and denote such a Liouville cobordism simply by $X$ (a similar convention will apply to most other mathematical objects). 
We view $\bdy^+X$ and $\bdy^-X$ as strict (i.e. equipped with a preferred contact form) contact manifolds.

Quite often we will have $\bdy^-X = \nil$, in which case $X$ is a {\bf Liouville domain}. 
We say that a Liouville domain $X$ has {\bf nondegenerate contact boundary} if the contact form $\alpha := \la|_{\bdy X}$ has nondegenerate Reeb orbits. 
The {\bf action} of a Reeb orbit $\gamma$ in $\bdy X$ is its period, i.e. the integral $\calA_{\bdy X}(\gamma) := \int_{\gamma}\alpha$, assuming $\gamma$ is parametrized so that its velocity always agrees with the Reeb vector field $R_\alpha$ on $\bdy X$.

More generally, a {\bf compact symplectic cobordism} is a compact manifold-with-boundary $X$ equipped with a symplectic form $\omega$ and a primitive one-form $\la$ defined on $\Op(\bdy X)$ whose restriction to $\bdy X$ is a contact form.
As before we have a natural decomposition $\bdy X = \bdy^+X \sqcup \bdy^- X$. We will refer to the case $\bdy^- X = \nil$ as a {\bf symplectic filling} and the case $\bdy^+ X = \nil$ as a {\bf symplectic cap}.
Note that the case with $\bdy X = \nil$ is simply a closed symplectic manifold.

\MS

\NI {\it Convention}: if $X$ and $X'$ are Liouville domains
and $\iota: X \hooksymp X'$ is a symplectic embedding, we will by slight abuse of notation write $X' \setminus X$ to denote the compact symplectic cobordism $X' \setminus \Int \iota(X)$, after attaching a small collar $[0,\delta) \times \bdy X'$ to $X'$ if necessary (i.e. if $\iota (X) \cap \bdy X' \neq \nil$).
\MS

\subsubsection{Admissible almost complex structures}\label{sss:admJ}
Let $Y$ be a strict contact manifold with contact form $\alpha$.
Recall that the {\bf symplectization} of $Y$ is the symplectic manifold $\R_r \times Y$, with symplectic form given by $d(e^r\alpha)$.
We denote by $\Jadm(Y)$ the space of {\bf admissible almost complex structures} on the symplectization $\R \times Y$.
That is, $J_Y \in \Jadm(Y)$ is a compatible almost complex structure on $\R \times Y$ which is $r$-translation invariant, maps $\bdy_r$ to the Reeb vector field $R_\alpha$, and restricts to a compatible almost complex structure on each contact hyperplane.

Given a compact symplectic cobordism $X$ with $Y^{\pm} := \bdy^{\pm}X$, its {\bf symplectic completion} $\wh{X}$ is given attaching a positive half-symplectization $\R_{\geq 0} \times Y^+$ to its positive boundary and a negative half-symplectization $\R_{\leq 0} \times Y^-$ to its negative boundary.
There is a natural symplectic form on $\wh{X}$ which extends that of $X$ and looks like the restriction of the symplectic form on a symplectization on the cylindrical ends.
We denote by $\Jadm(X)$ the space of {\bf admissible almost complex structures} on the symplectic completion of $X$. That is, $J_X \in \Jadm(X)$ is a compatible almost complex structure on $\wh{X}$ which is symplectization-admissible on the cylindrical ends, i.e. 
we have $J_X|_{\R_{\geq 0} \times Y^+} = J_{Y^+}|_{\R_{\geq 0} \times Y^+}$ for some $J_{Y^+} \in \Jadm(Y^+)$ and $J_X|_{\R_{\leq 0} \times Y^-} = J_{Y^-}|_{\R_{\leq 0} \times Y^-}$ for some $J_{Y^-} \in \Jadm(Y^-)$.
In particular, $J_X$ is translation-invariant on each cylindrical end.
Given fixed $J_{Y^+} \in \Jadm(Y^+)$ and $J_{Y^-} \in \Jadm(Y^-)$ as above, we denote by $$
\Jadm^{J_{Y^+}}_{J_{Y^-}}(X) \subset \Jadm(X)
$$
 the subspace consisting of almost complex structures $J$ which satisfy
$J |_{\R_{\geq 0} \times Y^+} = J_{Y^+}|_{\R_{\geq 0} \times Y^+}$ and 
$J |_{\R_{\leq 0} \times Y^-} = J_{Y^-}|_{\R_{\leq 0} \times Y^-}$.
By slight abuse of notation, for $J \in \Jadm^{J_{Y^+}}_{J_{Y^-}}$ we also use the notation $J|_{Y^\pm} := J_{Y^\pm}$ to denote the ``restriction'' of $J$ to $Y^{\pm}$.

\subsubsection{Moduli spaces of pseudoholomorphic curves}

Let $X$ be a compact symplectic cobordism, and consider $J \in \Jadm(X)$.
A {\bf $J$-holomorphic curve} $C$ in $\wh{X}$ consists of a Riemann surface $\Sigma$, with almost complex structure $j$, and a map $u: \Sigma \rightarrow \wh{X}$ satisfying $du \circ j = J \circ du$.
We will often refer to $C$ as a ``pseudoholomorphic curve'' (or simply ``curve'') if $J$ is implicit or unspecified.
Such a curve $C$ is {\bf asymptotically cylindrical} if $\Sigma$ is a closed Riemann surface minus a finite set of puncture points, such that $u$ is positively or negatively asymptotic to a Reeb orbit in the ideal boundary at each puncture (see e.g. \cite[\S 6.4]{wendl_SFT_notes} for a more precise formulation).
All pseudoholomorphic curves considered in this paper will be asymptotically cylindrical in either the symplectic completion of a compact symplectic cobordisms (closed symplectic manifolds being a special case), or in the symplectization of a contact manifold. Strictly speaking the latter is a special case of the former, but it is helpful to distinguish between these two cases since in the latter case we work with almost complex structures having an additional translation symmetry.
\MS

\NI \textit{Convention:} All pseudoholomorphic curves in this paper are asymptotically cylindrical, and for brevity we often refer to curves in $\wh{X}$ as simply ``curves in $X$'', with the process of symplectically completing tacitly understood.

\MS 

Consider tuples of nondegenerate Reeb orbits $\Gamma^+ = (\gamma_1^+,\dots,\gamma_a^+)$ in $\bdy^+X$ and $\Gamma^- = (\gamma_1^-,\dots,\gamma_b^-)$ in $\bdy^-X$.
Given $J \in \Jadm(X)$, we denote by $\calM^J_X(\Gamma^+;\Gamma^-)$ the moduli space of asymptotically cylindrical rational $J$-holomorphic curves in $\wh{X}$ with positive asymptotics $\Gamma^+$ and negative asymptotics $\Gamma^-$, equipped with the Gromov topology. 
Here the conformal structure on the domain varies over the moduli space of genus zero Riemann surfaces with $a$ (resp. $b$) ordered positive (resp. negative) punctures.
If $\bdy^-X = \nil$ we write $\calM^J_X(\Gamma^+)$ as a shorthand for $\calM^J_X(\Gamma^+;\nil)$, and similarly in the case $\bdy^+ X = \nil$ we write $\calM^J_X(\Gamma^-)$ in place of $\calM^J_X(\nil;\Gamma^-)$.
We will sometimes suppress $J$ from the notation and write simply $\calM(\Gamma^+;\Gamma^-)$ if the almost complex structure is implicit or unspecified.
\MS

\NI
\textit{Convention: by default all curves in this paper have genus zero unless otherwise stated.}
\MS

Similarly, given a strict contact manifold $Y$, $J \in \Jadm(Y)$, and Reeb orbits $\Gamma^+ = (\gamma_1^+,\dots,\gamma_a^+)$ and $\Gamma^- = (\gamma_1^-,\dots,\gamma_b^-)$ in $Y$, we denote by $\calM^J_Y(\Gamma^+;\Gamma^-)$ the moduli space of asymptotically cylindrical curves in $\R \times Y$ with positive asymptotics $\Gamma^+$ and negative asymptotics $\Gamma^-$. There is a natural $\R$-action on $\calM^J_Y(\Gamma^+;\Gamma^-)$ induced by translations in the first factor of $\R \times Y$, and this is free away from the {\bf trivial cylinders}, i.e. cylinders of the form $\R \times \gamma$ with $\gamma$ a Reeb orbit in $Y$.
We denote the quotient by $\calM^J_Y(\Gamma^+;\Gamma^-)/\R$.

We will consider moduli spaces associated to $1$-parameter families of almost complex structures. For instance, given a $1$-parameter family $\{J_t\}_{t \in [0,1]}$ in $\Jadm(X)$, we denote by $\calM_{X}^{\{J_t\}}(\Gamma^+;\Gamma^-)$ the corresponding {\bf parametrized moduli space} consisting of pairs $(t,C)$ with $t \in [0,1]$ and $C \in \calM_{X}^{J_t}(\Gamma^+;\Gamma^-)$. 

We will assume throughout that suitable choices have been made so that every regular moduli space of curves is oriented. In particular, any curve $C$ which is regular and isolated in $\calM_X(\Gamma^+;\Gamma^-)$ or $\calM_Y(\Gamma^+;\Gamma^-)/\R$ has an associated sign $\eps(C) \in \{-1,1\}$.
We briefly recall the procedure for orienting moduli spaces in \S\ref{subsec:aut_trans}.

\subsubsection{SFT compactifications}\label{subsubsec:SFT_cpct}

The above moduli spaces admit SFT compactifications as in \cite{BEHWZ}, which we denote by replacing $\calM$ with $\ovl{\calM}$. 
For example, let $X$ be a compact symplectic cobordism with $J_{\pm} \in \Jadm(\bdy^{\pm}X)$ and $J_X \in \Jadm^{J_{+}}_{J_{-}}(X)$.
Elements of $\ovl{\calM}_{X}^{J_X}(\Gamma^+;\Gamma^-)$ are {\bf stable pseudoholomorphic buildings} in $\wh{X}$,
which consist of the following data:
\begin{itemize}
\item
some number (possibly zero) of $J_{+}$-holomorphic levels in the symplectization $\R \times \bdy^+X$
\item 
a ``main'' $J_X$-holomorphic level in $\wh{X}$
\item some number (possibly zero) of $J_{-}$-holomorphic levels in the symplectization $\R \times \bdy^- X$
\end{itemize}
such that for each pair of adjacent levels the positive asymptotic Reeb orbits of the lower level are paired with the negative asymptotic Reeb orbits of the upper level. 
The symplectization levels are always defined modulo target translations.
Note that each level consists of one or more connected components, each of which is a nodal punctured Riemann surface. 
The stability condition states that each 
component of the domain on which the map is constant
must have negative Euler characteristic after removing all special points; also there are no symplectization levels consisting entirely of trivial cylinders.
See \cite{BEHWZ} for details.

We will use the following language in this paper.  (Note that the slightly different notion of {\em matched component} employed in \cite{McDuffSiegel_counting} serves a similar purpose.)

\begin{definition}\label{def:curvecomp}
We  say that a (rational) curve in a given level is ``connected'' if its domain is connected but possibly nodal, ``smooth'' if its domain is without nodes, and ``irreducible'' if it is both connected and smooth. 
By {\bf curve component} we mean a (rational) curve which is irreducible. 
\end{definition}
\NI Note that each level of a pseudoholomorphic building can be decomposed into its constituent (irreducible) components.

We will also frequently make use of {\bf neck stretching}. 
If $X^+$ and $X^-$ are compact symplectic cobordisms with a common contact boundary $\bdy^-X^+ = \bdy^+X^- =Y$, we denote the glued compact symplectic cobordism by $X^- \ccirc X^+$. 
Given almost complex structures $J_Y \in \Jadm(Y)$, $J_{X^+} \in \Jadm_{J_Y}(X^+)$ and $J_{X^-} \in \Jadm^{J_Y}(X^-)$, we can consider the corresponding neck-stretching family of almost complex structures 
$J_t \in \Jadm(X)$, $t \in [0,1)$.
The limit $t \rightarrow 1$ corresponds to the broken cobordism which we denote by $X^- \notccirc X^+$. 
The compactification $\ovl{\calM}_X^{\{J_t\}}(\Gamma^+;\Gamma^-)$ consists of pairs $(t,C)$ for $t \in [0,1)$ and $C \in \ovl{\calM}^{J_t}_X(\Gamma^+;\Gamma^-)$, as well as limiting configurations for $t=1$, which are pseudoholomorphic buildings with:
\begin{itemize}
\item some number (possibly zero) of $J_{\bdy^+ X^+}$-holomorphic levels in the symplectization $\R \times \bdy^+X^+$
\item a $J_{X^+}$-holomorphic level in $\wh{X^+}$
\item some number (possibly zero) of $J_Y$-holomorphic levels in the symplectization $\R \times Y$
\item a $J_{X^-}$-holomorphic level in $\wh{X^-}$
\item some number (possibly zero) of $J_{\bdy^-X^-}$-holomorphic levels in the symplectization $\R \times \bdy^-X^-$,
\end{itemize}
subject to suitable matching and stability conditions. Here we put
$J_{\bdy^+ X^+} := J_{X^+}|_{\bdy^+ X^+}$ and $J_{\bdy^- X^-} := J_{X^-}|_{\bdy^- X^-}$.

\subsubsection{Homology classes and energy}

Given a compact symplectic cobordism $X$ and Reeb orbits $\Gamma^+ = (\gamma^+_1,\dots,\gamma^+_a)$ in $\bdy^+X$ and $\Gamma^- = (\gamma_1^-,\dots,\gamma_b^-)$ in $\bdy^-X$,
we let $H_2(X,\Gamma^+\cup\Gamma^-)$ denote the group of potential homology classes of curves in $\calM_X(\Gamma^+;\Gamma^-)$. Namely, $H_2(X,\Gamma^+\cup\Gamma^-)$ is the abelian group freely generated by $2$-chains $\Sigma$ in $X$ with $\bdy \Sigma = \sum_{i=1}^a \gamma_i^+ - \sum_{j=1}^b\gamma_j^-$, modulo boundaries of $3$-chains in $X$ (see also \cite[\S6.4]{wendl_SFT_notes} for a slightly more homological perspective).
Given $A \in H_2(X,\Gamma^+\cup\Gamma^-)$, we denote by $\calM_{X,A}(\Gamma^+;\Gamma^-) \subset \calM_{X}(\Gamma^+;\Gamma^-)$ the subspace of curves lying in homology class $A$.

Similarly, given a strict contact manifold $Y$ and Reeb orbits $\Gamma^+ = (\gamma^+_1,\dots,\gamma^+_a)$ and $\Gamma^- = (\gamma_1^-,\dots,\gamma_b^-)$ in $Y$, we denote by $H_2(Y,\Gamma^+\cup\Gamma^-)$ the homology group of $2$-chains $\Sigma$ in $Y$ with $\bdy \Sigma = \sum_{i=1}^a \gamma_i^+ - \sum_{j=1}^b\gamma_j^-$, modulo boundaries of $3$ chains in $Y$. 
Given $A \in H_2(Y,\Gamma^+\cup\Gamma^-)$, 
we denote by $\calM_{Y,A}(\Gamma^+;\Gamma^-) \subset \calM_{Y}(\Gamma^+;\Gamma^-)$ the subspace of curves in $\R\times Y$ lying in homology class $A$.

There are also natural subspaces $\ovl{\calM}_{X,A}(\Gamma^+;\Gamma^-) \subset \ovl{\calM}_{X}(\Gamma^+;\Gamma^-)$ and $\ovl{\calM}_{Y,A}(\Gamma^+;\Gamma^-) \subset \calM_{Y}(\Gamma^+;\Gamma^-)$ and so on.
These are defined by required the total homology class of a building, which is defined in a natural way by concatenating the levels, to be $A$.

\sss

If $(Y,\alpha)$ is strict contact manifold, we define the {\bf energy} of a curve $C \in \calM_{Y,A}(\Gamma^+;\Gamma^-)$ to be
$E_Y(C) := \int_C d\alpha.$\footnote{Note that this is called the {\bf $\omega$-energy in \cite{BEHWZ}}, their full energy having this as one of its two summands.} 
By Stokes' theorem, we have
\begin{align*}
E_Y(C) = \sum_{i=1}^a \calA_{Y}(\gamma^+_i) - \sum_{j=1}^b \calA_Y(\gamma^-_j).
\end{align*}
Note that this depends only the homology class $A \in H_2(Y,\Gamma^+\cup\Gamma^-)$, so we can also put $E_Y(C) = E_Y(A) := \int_A d\alpha$.
Similarly, if $X$ is a compact symplectic cobordism with symplectic form $\omega$ and locally defined Liouville one-form $\la$, 
the energy $E_X(C)$ of a curve $C \in \calM_{Y,A}^J(\Gamma^+;\Gamma^-)$ is defined to be the integral over $C$ of the piecewise smooth two-form which agrees with $\omega$ on $X$ and with $d\la$ on the cylindrical ends $\wh{X} \less X$. 
If $X$ is further a Liouville cobordism (i.e. $\la$ is globally defined), then Stokes' theorem gives
\begin{align*}
E_X(C) = \sum_{i=1}^a \calA_{\bdy^+X}(\gamma^+_i) - \sum_{j=1}^b \calA_{\bdy^-X}(\gamma^-_j).
\end{align*}
This again depends only on $A \in H_2(X,\Gamma^+\cup\Gamma^-)$, and we have
$E_X(C) = E_X(A) := \int_A \omega$.

\subsection{Local tangency and skinny ellipsoidal constraints}\label{subsec:loc_tang_and_E_sk}

Let $X$ be a compact symplectic cobordism.
Recall that the {\bf local tangency constraint} $\lll \T^{(m)}p\rrr$ with $m \in \Z_{\geq 1}$ is imposed by choosing a point $p \in \Int X$ and a smooth symplectic divisor $D \subset \Op(p)$ and considering curves with an additional marked point required to pass through $p$ with contact order (at least) $m$ to $D$ (see e.g. \cite{CM1,CM2,McDuffSiegel_counting}). 
We will also denote this constraint by $\lll \T^{m-1}p\rrr$, with $m-1$ representing the tangency order (in particular $\lll p\rrr$ corresponds simply to a marked point passing through $p$).

Let $\Jadm(X;D) \subset \Jadm(X)$ denote the space of admissible almost complex structures on $\wh{X}$ which are integrable near $p$ and preserve the germ of $D$ near $p$. 
Given tuples of Reeb orbits $\Gamma^+$ and $\Gamma^-$ in $\bdy^+X$ and $\bdy^-X$ respectively and $J \in \Jadm(X;D)$, we define the moduli space $\calM_{X}^J(\Gamma^+;\Gamma^-)\lll \T^{(m)}p\rrr$ as before, but now the local tangency constraint $\lll \T^{(m)}p\rrr$ is imposed on each curve.

Some care is needed when compactifying $\calM_{X}^J(\Gamma^+;\Gamma^-)\lll \T^{(m)}p\rrr$, due to the possibility of a ghost (i.e. constant) component inheriting the marked point. Indeed, strictly speaking a constant component is tangent to $D$ to infinite order, and hence ghost configurations always appear with much higher than expected dimension. 
To get around this, first note, as in the proof of Proposition 2.2.2 in \cite{McDuffSiegel_counting}, that there is a natural inclusion
$$\calM_{X}^J(\Gamma^+;\Gamma^-)\lll \T^{(m)}p\rrr \subset \ovl{\calM}_{X}^J(\Gamma^+;\Gamma^-)\lll p\rrr,$$
where the codomain is the usual SFT compactification of $\calM_{X}^J(\Gamma^+;\Gamma^-)\lll p\rrr$ by stable pseudoholomorphic buildings.
Let $\ovl{\calM}_{X}^J(\Gamma^+;\Gamma^-)\lll T^{(m)}p\rrr$
denote the closure of $\calM_{X}^J(\Gamma^+;\Gamma^-)\lll \T^{(m)}p\rrr$
in this compact ambient space.
To understand what this amounts to, consider a pseudoholomorphic building $C$ in 
$\ovl{\calM}_{X}^J(\Gamma^+;\Gamma^-)\lll T^{(m)}p\rrr$ such that the marked point $z_0$ mapping to $p$ lies on a ghost component $C_0$. 
Let $N_1,\dots,N_a$ denote those nodes connecting a nonconstant component of $C$ to $C_0$, or more generally connecting a nonconstant component of $C$ to some ghost component which is nodally connected through ghost components to $C_0$.
Let $z_1,\dots,z_a$ denote the corresponding special points in the domain of $C$ which are ``near $z_0$'', i.e.
participate in the nodes $N_1,\dots,N_a$ and lie on nonconstant components of $C$.
Let $C_1,\dots,C_a$ denote the respective nonconstant components of $C$ on which $z_1,\dots,z_a$ lie.
According to \cite[Lem. 7.2]{CM1}, in this situation the marked points 
$z_1,\dots,z_a$ satisfy local tangency constraints $\lll \T^{(m_1)}p\rrr,\dots,\lll \T^{(m_a)}p\rrr$ respectively such that we have
$$ m_1 + \dots m_a \geq m.$$
In this way, elements of $\ovl{\calM}_{X}^J(\Gamma^+;\Gamma^-)\lll \T^{(m)}p\rrr$ ``remember'' the constraint $\lll\T^{(m)}p\rrr$. 

We will also need to consider a potentially larger compactification of $\calM_{X}^J(\Gamma^+;\Gamma^-)\lll T^{(m)}p\rrr$ which allows all ghost configurations as described above, even if they do not arise as a limit of smooth curves:
\begin{definition}\label{def:doublebar_cpctification}
Let $\ovll{\calM}_X^J(\Gamma^+;\Gamma^-)\lll \T^{(m)}p\rrr$ denote the subset of $\ovl{\calM}_X^J(\Gamma^+;\Gamma^-)\lll p\rrr$ given by the union of $\calM_X^J\lll \T^{(m)}p\rrr$ with the set of all buildings $C$ such that the marked point $z_0$ mapping to $p$ lies on a ghost component and the special points $z_1,\dots,z_a$ near $z_0$ (as above) satisfy respective constraints $\lll \T^{(m_1)}p\rrr,\dots,\lll \T^{(m_a)}p\rrr$ such that $m_1 + \dots + m_a \geq m$. See Figure~\ref{fig:ghost}.
\end{definition}
\begin{figure}
  \includegraphics[scale=1.5]{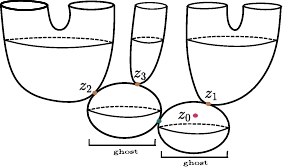}
  \caption{A configuration which could potentially arise in $\ovl{\calM}_X^J(\Gamma^+;\Gamma^-)\lll \T^{(m)}p\rrr$. 
Here the marked point $z_0$ mapping to $p$ lies on a ghost component, and $z_1,z_2,z_3$ are the special points near $z_0$ lying on nonconstant components. These satisfy respective constraints $\lll \T^{(m_1)}p\rrr, \lll \T^{(m_2)}p\rrr,\lll \T^{(m_3)}p\rrr$ such that $m_1+m_2+m_3 \geq m$.
Such a configuration is also included in $\ovll{\calM}_X^J(\Gamma^+;\Gamma^-)\lll \T^{(m)}p\rrr$ even if it does not arise as a limit of curves in $\calM_X^J(\Gamma^+;\Gamma^-)\lll \T^{(m)}p\rrr$.
}
  \label{fig:ghost}
\end{figure}
\begin{rmk}
It is worth emphasizing that the extra buildings $C$ involving ghost components which appear in Definition~\ref{def:doublebar_cpctification} have virtual codimension at least two (c.f. the proof of \cite[Prop. 2.2.2]{McDuffSiegel_counting}), and hence are not expected to appear whenever sufficient transversality holds.
This is essentially why such configurations do not contribute to the local tangency constraint counts $N_{M,A}\lll \T^{(m)}p\rrr$ defined in \cite{McDuffSiegel_counting} for semipositive closed symplectic manifolds $M$. 
\end{rmk}

\sss

For $m \in \Z_{\geq 1}$, let $\lll (m) \rrr_E$ denote the {\bf skinny ellipsoidal constraint} of order $m$, defined as follows.
Let $E_\sk$ denote a {\bf skinny ellipsoid}, i.e. a symplectic ellipsoid whose first area factor is sufficiently small compared to the others. 
After possibly shrinking (i.e.  
replacing 
 $E_\sk$ by $\eps E_\sk$ for $0<\eps <<1$) we can assume that $E_\sk$ symplectically embeds into $X$ in an essentially unique way, and we typically denote this embedding by an inclusion $E_\sk \subset X$. 
Let $\eta_m$ denote the $m$-fold cover of the simple Reeb orbit of least action in $\bdy E_\sk$. 
For curves in $X$, the constraint $\lll (m) \rrr$ is imposed by replacing $\wh{X}$ with $\wh{X \setminus E_\sk}$, and considering curves with one additional negative puncture which is asymptotic to $\eta_m$. 
We define the moduli space $\calM_X^J(\Gamma^+;\Gamma^-)\lll {(m)}\rrr_E$ 
by analogy with $\calM_X^J(\Gamma^+;\Gamma^-)\lll \T^{(m)}p\rrr$, replacing the local tangency constraint $\lll \T^{(m)}p\rrr$ with the skinny ellipsoidal constraint $\lll (m)\rrr_E$.
Note that both of these moduli spaces have the same index, namely 
\begin{align}\label{eq:fred_ind}
\ind = (n-3)(2-a-b) + 2c_1^\tau(A)  + \sum_{i=1}^a \cz_\tau(\gamma^+_i) - \sum_{j=1}^b \cz_\tau(\gamma^-_j) -2n - 2m + 4,
\end{align}
where $2n = \dim_\R(X)$.
Here $\tau$ is a choice of trivialization (up to homotopy) of the symplectic vector bundle over each Reeb orbit, $c_1^\tau(A)$ is the corresponding relative first Chern class evaluated on $A$, and $\cz_\tau$ denotes the Conley--Zehnder index measured with respect to $\tau$.
Recall that the index does not depend on the choice of $\tau$, even though the individual terms do.

If $M$ is a closed symplectic four-manifold with homology class $A \in H_2(M)$, \cite[\S4.1]{McDuffSiegel_counting} establishes an equivalence of signed counts
$$ \#\calM_{M,A} \lll \T^{(m)}p\rrr = \#\calM_{M,A}\lll(m)\rrr_E.$$
The basic idea is to place the tangency constraint in $E_\sk$ and stretch the neck along $\bdy E_\sk$, and then to argue that only degenerations of the expected type can arise.
Although \cite{McDuffSiegel_counting} only proves this in dimension four in order to invoke an argument which sidesteps any technicalities about gluing curves with tangency constraints, this is expected to hold for closed manifolds of all dimensions.  In
 the context of a symplectic cobordism $X$, it is not quite reasonable to expect in general an equality of signed curve counts
\begin{align*}
\#\calM_{X,A}^J(\Gamma^+;\Gamma^-)\lll {(m)}\rrr_E = \#\calM_{X,A}^J(\Gamma^+;\Gamma^-)\lll \T^{(m)}p\rrr.
\end{align*}
Indeed, these counts might not be particularly robust, e.g. they could depend on $J$ and the embedding $E_\sk \hooksymp X$.
However, the argument in \cite[Thm.4.1.1]{McDuffSiegel_counting} does extend to this setting to prove:

\begin{prop}\label{prop:T=E} If $\dim X = 4$,
we have $$
\#\calM_{X,A}^J(\Gamma^+;\Gamma^-)\lll {(m)}\rrr_E = \#\calM_{X,A}^J(\Gamma^+;\Gamma^-)\lll \T^{(m)}p\rrr,
$$
 provided that the following conditions hold:
\begin{enumerate}
\item[{\rm (i)}] the moduli space $\#\calM_{X,A}^J(\Gamma^+;\Gamma^-)\lll \T^{(m)}p\rrr$ is formally perturbation invariant (see \S\ref{subsec:pert_inv} below) 
\item[{\rm (ii)}] each Reeb orbit in $\Gamma^+ \cup \Gamma^-$ is nondegenerate and either elliptic or negative hyperbolic.
\end{enumerate}
\end{prop} 
\NI Indeed, the first condition  guarantees that curve counts remain constant
over a generic $1$-parameter family of almost complex structures (c.f. Proposition~\ref{prop:count_J_indep}), and the second condition ensures that the relevant curves count positively (c.f. Theorem~\ref{subsec:aut_trans} and Remark~\ref{rmk:aut_trans}~(ii)).

\subsection{Formal curves}\label{subsec:formal_curves}

In this subsection we introduce the notion of a ``formal curve'', which is a convenient device for storing combinatorial curve data, but without requiring that this data be represented by any actual solution to the pseudoholomorphic curve equation. 
We also define ``formal buildings'', which are analogous to pseudoholomorphic buildings but with each pseudoholomorphic curve component replaced by a formal curve component.
This will allow us to discuss ``formal perturbation invariance'' of moduli spaces in the next subsection.

\subsubsection{Formal curve components}

To begin, we define:
\begin{definition}
A {\bf formal curve component} $C$ in a compact symplectic cobordism $(X,\omega)$ is a triple $(\Gamma^+,\Gamma^-,A)$, where
\begin{itemize}
\item
$\Gamma^+ = (\gamma^+_1,\dots,\gamma^+_a)$ is a tuple of Reeb orbits in $\bdy^+X$ 
\item $\Gamma^- = (\gamma^-_1,\dots,\gamma^-_b)$ is a tuple of Reeb orbits in $\bdy^-X$
\item $A \in H_2(X,\Gamma^+\cup \Gamma^-)$ is a homology class
\item we require the energy $E_X(C) := E_X(A) = \int_A \omega$ to be nonnegative. 
\end{itemize}

Similarly, a formal curve component $C$ in a strict contact manifold $(Y,\alpha)$ 
is a triple $(\Gamma^+,\Gamma^-,A)$, where $\Gamma^+,\Gamma^-$ are tuples of Reeb orbits in $Y$ and $A \in H_2(Y,\Gamma^+\cup\Gamma^-)$ is a homology class, and we require the energy
$E_Y(C) := E(A) = \int_A d\alpha$ to be nonnegative.
\end{definition}

We view $C$ as representing a hypothetical genus zero\footnote{One could of course extend the definition to allow for higher genus curves, but we will not need this.} irreducible asymptotically cylindrical curve in $\wh{X}$ or $\R \times Y$.
Note that a formal curve component also has a well-defined index $\ind(C)$, defined by the same formula ~\eqref{eq:fred_ind}.
We will say that a formal curve component in $Y$ is a ``trivial cylinder'' (or just ``trivial'') if $a = b = 1$ and $\Gamma^+ = \Gamma^- = (\gamma)$ for some Reeb orbit $\gamma$ in $Y$. 
A formal curve component $C$ is ``closed'' if $\Gamma^+ = \Gamma^- = \nil$, and it is moreover ``constant'' if 
$E_X(C) = 0$.

It will also be convenient to speak about formal curve components in $X$ carrying a constraint $\lll \T^{(m)}p\rrr$ for some $m \in \Z_{\geq 1}$. 
Here the constraint $\lll \T^{(m)}p\rrr$ is an extra piece of formal data which has the effect of 
decreasing the index by $2n-4 + 2m$ (here $2n = \dim_{\R}(X)$).

Given a formal curve $C = (\Gamma^+,\Gamma^-,A)$ in $X$ and $J_X \in \Jadm(X)$, we introduce the shorthand notation $\calM_X^{J_X}(C) := \calM_{X,A}^{J_X}(\Gamma^+;\Gamma^-)$ 
 for the corresponding space of $J_X$-holomorphic curves representing $C$.
 As before, we will often omit the almost complex structure from the notation.
 Similarly, if $C = (\Gamma^+,\Gamma^-,A)$ is a formal curve in $Y$ and $J_Y \in \Jadm(Y)$, we put $\calM_Y^{J_Y}(C) := \calM_{Y,A}^J(\Gamma^+;\Gamma^-)$.
This shorthand also applies when $C$ carries a local tangency constraint, which is then implicit e.g. in the notation 
$\calM_X(C)$.

\subsubsection{Formal nodal curves and buildings}
We now extend the above definition in order to model elements of the SFT compactification.
Firstly, a \textit{connected formal nodal curve} $C$ in $X$ or $Y$ is roughly the same as a pseudoholomorphic nodal curve, but with each pseudoholomorphic curve component replaced by a formal curve component. 
More precisely:

\begin{definition}\label{def:formcurve}
A {\bf connected formal nodal curve} $C$ in $X$ (resp. $Y$) consists of:
\begin{itemize}
\item
a tree $T$
\item for each vertex $v$ of $T$, a formal curve component $C_v$ in $X$ 
(resp. $Y$).
\end{itemize}

More generally, we drop the ``connected'' condition by allowing $T$ to be a forest (i.e. disjoint union of trees). 
\end{definition}
\NI We view the edges as representing nodes. We will say that $C$ is {\bf stable} if, for each nonconstant component $C_v$, the number of punctures plus the number of edges connected to $v$ is at least three. 

\begin{definition}\label{def:formbuild}
A {\bf formal building} in $X$ consists of:
\begin{itemize}
\item
formal nodal curves $C_1,\dots,C_a$ in $\bdy^+X$ for some $a \in \Z_{\geq 0}$
\item 
a formal nodal curve $C_0$ in $X$
\item
formal nodal curves $C_{-1},\dots,C_{-b}$ in $\bdy^-X$ for some $b \in \Z_{\geq 0}$,
\end{itemize}
such that the tuple of positive Reeb orbits for $C_i$ coincides with the tuple of negative Reeb orbits for $C_{i+1}$ for $i = -b,\dots,a-1$.
We also assume that the graph given naturally by concatenating the forest of each level is acyclic. 

Similarly, a formal building in $Y$ consists of formal nodal curves $C_1,\dots,C_a$ in $Y$ for some $a \in \Z_{\geq 1}$, that the tuple of positive Reeb orbits for $C_i$ coincides with the tuple of negative Reeb orbits for $C_{i+1}$ for $i = 1,\dots,a-1$, and such that the underlying graph is acyclic.
\end{definition}

We view a formal building as modeling a rational pseudoholomorphic building in $X$ or 
$\R\times Y$, with each constituent formal nodal curve representing a level.
Note that the acyclicity condition ensures total genus zero and could be relaxed, but for our purposes we will keep it.
Such a building has a {\bf total homology class} in $H_2(X;\Gamma^+\cup \Gamma^-)$ or $H_2(Y;\Gamma^+\cup\Gamma^-)$, where $\Gamma^+$ (resp. $\Gamma^-$) is the tuple of positive Reeb orbits of the top (resp. bottom) level. 
We will say that a formal building is {\bf stable} if each constituent formal nodal curve is stable, and no level is a union of trivial cylinders.
We denote the set of stable formal buildings in $X$ whose top (resp. bottom) level has positive (resp. negative) Reeb orbits $\Gamma^+$ (resp. $\Gamma^-$) by $\ovl{\formal}_{X,A}(\Gamma^+;\Gamma^-)$.
The set $\ovl{\formal}_{Y,A}(\Gamma^+;\Gamma^-)$ of stable formal buildings in $Y$ is defined similarly.

\sss

We denote the formal analogue of $\ovll{\calM}_{X,A}(\Gamma^+;\Gamma^-)\lll \T^{(m)}p\rrr$ by $\ovll{\formal}_{X,A}(\Gamma^+;\Gamma^-)\lll \T^{(m)}p\rrr$. This consists of two types of stable formal buildings, modeling curves where the marked point $z_0$ mapping to $p$ lies on a nonconstant component or constant component respectively.
In the first case, we have all stable formal buildings such that one of the components in $X$ is formally endowed with a constraint $\lll \T^{(m)}p\rrr$.
In the second case, we have all stable formal buildings such that some constant component $C_0$ in $X$ is formally endowed with a constraint $\lll p\rrr$, and the nearby nonconstant components $C_1,\dots,C_a$ (i.e. those nonconstant components which are nodally connected through constant components to $C_0$) are formally endowed with constraints $\lll \T^{(m_1)}p\rrr,\dots,\lll \T^{(m_a)}p\rrr$ respectively such that $m_1 + \dots + m_a \geq m$ (c.f. \S\ref{subsec:loc_tang_and_E_sk}).
Note that the extra constraint $\lll p \rrr$ is taken into account as a marked point when formulating stability, whereas the constraints $\lll \T^{(m_1)}p\rrr,\dots,\lll \T^{(m_a)}p\rrr$ do not affect stability since they lie on nonconstant components.

\subsubsection{Formal covers}

Next, we define the formal analogue of multiple covers of pseudoholomorphic curves.
Let $X$ be a symplectic filling, and let $\Gamma = (\gamma_1,\dots,\gamma_a)$ and $\ovl{\Gamma} = (\ovl{\gamma}_1,\dots,\ovl{\gamma}_{\ovl{a}})$ be tuples of Reeb orbits in $Y := \bdy X$.
Let $C = (\Gamma,\nil,A)$ and $\ovl{C} = (\ovl{\Gamma},\nil,\ovl{A})$ be formal curve components in $X$, satisfying constraints $\lll \T^{(m)}p\rrr$ and $\lll \T^{(\ovl{m})}p\rrr$ respectively.
We say that $C$ is a $\kappa$-fold {\bf formal cover} of $\ovl{C}$ if there exists
\begin{itemize}
\item
a sphere $\Sigma$ with marked points $(z_0,\dots,z_a)$
\item 
a sphere $\ovl{\Sigma}$ with marked points $(\ovl{z}_0,\dots,\ovl{z}_{\ovl{a}})$
\item a $\kappa$-fold branched cover $\pi: \Sigma \rightarrow \ovl{\Sigma}$
\end{itemize}
such that
\begin{itemize}
\item $\pi^{-1}(\{ \ovl{z}_1,\dots,\ovl{z}_{\ovl{a}}\}) = \{z_1,\dots,z_a\}$
\item $\pi(z_0) = \ovl{z}_0$
\item for each $i = 1,\dots,a$, $\gamma_i$ is the $\kappa_i$-fold cover of $\ovl{\gamma}_j$, where $j$ is such that $\pi(z_i) = \ovl{z}_j$ and $\kappa_i$ is the ramification order of $\pi$ at $z_i$
 \item we have  $\kappa\ovl{m} \geq m$, 
 where $\kappa$ is the ramification order of $\pi$ at $z_0$.
  \end{itemize}
A formal curve component is {\bf simple} if it cannot be written as a nontrivial (i.e. with $\kappa \geq 2$) formal cover of any other formal curve component.

\subsection{Formal perturbation invariance}\label{subsec:pert_inv}

The following is our main criterion for establishing upper bounds and proving stabilization for the capacities defined in \S\ref{sec:capacity}.

\begin{definition}\label{def:formal_pert_inv}
Let $X$ be a Liouville domain with nondegenerate contact boundary $Y$, and let $C$ be an index zero simple formal curve component in $X$ with positive asymptotics $\Gamma = (\gamma_1,\dots,\gamma_a)$, homology class $A \in H_2(X,\Gamma)$, and carrying a constraint $\lll \T^{(m)}p\rrr$ for some $m \in \Z_{\geq 1}$. 
We say that $C$ is {\bf formally perturbation invariant} if there exists a generic $J_Y \in \Jadm(Y)$ such that the following holds. Suppose that $C' \in \ovll{\formal}_{X,A}(\Gamma)\lll \T^{(m)}p\rrr$
is any stable formal building satisfying:
\begin{itemize}
\item[{\rm (A1)}] Each nonconstant
component of $C'$ in $X$ is a formal cover of some formal curve component $\ovl{C}'$ with $\ind(\ovl{C}') \geq -1$.
\item [{\rm (A2)}] Each nonconstant component of $C'$ in $Y$ is a formal cover of some 
formal curve component $\ovl{C}'$ which is either trivial or else satisfies $\ind(\ovl{C}') \geq 1$.
\end{itemize}
Then either:
\begin{itemize}
\item [{\rm (B1)}]
$C'$ consists of a single component, i.e. $C' = C$, or else
\item [{\rm (B2)}]
$C'$ is a two-level building, with bottom level in $X$ consisting of a single component $C_X$ which is simple
with index $-1$, 
and with top level in $Y$ 
represented by a union of some  trivial cylinders with a simple index $1$  component 
$C_Y$ in $\R\times Y$; moreover we require that $\calM_Y^{J_Y}(C_Y)$ is regular and satisfies 
 $\# \calM^{J_{Y}}_Y(C_Y)/\R = 0$.
\end{itemize}

More generally, if $C$ is any formal curve component in $X$, we say that it is formally perturbation invariant if it is a formal cover of an index zero simple formal curve component $\ovl{C}$ which is formally perturbation invariant as above.
\end{definition}

We will also say that the associated moduli space $\calM_X(C)$ is formally perturbation invariant if the formal curve component $C$ is.
Roughly, this means that for ``purely formal reasons'' the moduli space $\calM_X(C)$ cannot degenerate in a generic $1$-parameter family. 
More precisely, the condition is ``formal in $X$ but not in $Y$'', i.e. it takes into account pseudoholomorphic curves in $\R \times Y$ (via the last condition about $\calM_Y^{J_Y}(C_Y)$) but only formal curves in $X$.
\footnote{In our application, $C_Y$ will occur as a low energy cylinder between an elliptic orbit $e_{i,j}$ and the corresponding hyperbolic orbit $h_{i,j}$ in $\bdy\tX_\Omega$ (c.f. Lemma~\ref{lem:low_energy_cyls} below).}
We will also say that $C$ is ``formally perturbation invariant with respect to $J_Y$'' when we wish to emphasize the role of $J_Y$ in Definition~\ref{def:formal_pert_inv}.

The following is a consequence of structure transversality and gluing techniques for simple curves:

\begin{prop}\label{prop:count_J_indep}
Let $X$ be a Liouville domain with nondegenerate contact boundary $Y$, and let $C$ be a simple index zero formal curve component $X$ which carries a local tangency constraint $\lll \T^{(m)} p\rrr$. Assume that $C$ is formally perturbation invariant with respect to some generic $J_{Y} \in \Jadm(Y)$.
Then the associated moduli space $\calM_X^{J_X}(C)$ is regular and finite for generic $J_X \in \Jadm^{J_Y}(X;D)$, and moreover the signed count $\# \calM_X^{J_X}(C)$ is independent of $J_X$ 
provided that $\calM_X^{J_X}(C)$ is regular.
\end{prop}
\begin{proof}
If $J_Y \in \Jadm(Y)$ and $J_X \in \Jadm^{J_Y}(X;D)$ are generic, 
it follows by standard transversality techniques (c.f. \cite[\S 8]{wendl_SFT_notes}) that:
\begin{itemize}
\item
every simple $J_Y$-holomorphic curve component in $\R \times Y$ is either trivial or else has index at least one
\item
every simple $J_X$-holomorphic curve component in $\wh{X}$ has nonnegative index.
\end{itemize}
In particular, since $C$ is simple, $\calM^{J_X}_X(C)$ is regular and hence a zero-dimensional smooth oriented manifold.
It also follows by formal perturbation invariance of $C$ and the SFT compactness theorem (plus the discussion in \S\ref{subsec:loc_tang_and_E_sk})
that we must have $\ovll{\calM}_X^{J_X}(C) = \calM_X^{J_X}(C)$, whence $\calM_X^{J_X}(C)$ is finite.
Indeed, any element $C'$ of $\ovll{\calM}_X^{J_X}(C)$
defines a stable formal building in $\ovll{\formal}_{X,A}(\Gamma)\lll \T^{(m)}p\rrr$ satisfying (A1) and (A2), and 
since (B2) is impossible when $J_X$ is regular we must have $C' \in \calM_X^{J_X}(C)$.

Now assume that $J_0,J_1 \in \calJ^{J_Y}(X;D)$ are chosen such that $\calM_X^{J_i}(C)$ is regular for $i = 0,1$, and let $\{J_t \}_{t \in [0,1]}$ be a generic $1$-parameter family in $\calJ^{J_Y}(X;D)$ interpolating between them.
Standard transversality techniques imply that $\calM_X^{\{J_t\}}(C)$ is regular and hence a smooth oriented $1$-dimensional manifold. 
By formal perturbation invariance and the SFT compactness theorem,
the compactification 
$\ovll{\calM}_X^{\{J_t\}}(C)$ (defined similarly to Definition~\ref{def:doublebar_cpctification}) is a smooth cobordism between $\calM_X^{J_0}(C)$ and $\calM_X^{J_1}(C)$, with possibly some additional boundary configurations as in (B2). 
Our goal is to prove $\#\calM_X^{J_0}(C) = \#\calM_X^{J_1}(C)$. 
Note that this would be immediate if there were none of these additional boundary configurations.

Each of these additional boundary configurations occurs at some time $t_b \in (0,1)$ and consists of a two-level building, with:
\begin{itemize}
\item a top level $J_Y$-holomorphic curve in $\R \times Y$ having a single nontrivial component $C_Y$ which satisfies $\ind(C_Y) = 1$ and such that $\calM_Y^{J_Y}(C_Y)$ is regular with $\# \calM_Y^{J_Y}(C_Y)/\R = 0$
\item bottom level having a single component $C_X$ which has index $-1$ and is simple.
\end{itemize}
By standard transversality techniques we can assume that $C_X$ is regular in the parametrized sense.

We now invoke SFT gluing, using e.g. the general formulation given in \cite[Thm. 2.54]{Pardcnct} (see also \cite[\S2.5.3]{morse_homology_schwarz} for the simpler Morse homology analogue of our setting).
For ease of discussion let us make the following simplifying assumptions:
\begin{itemize}
\item
all of the additional boundary configurations occur at the same time $t_b \in (0,1)$
\item 
all of these configurations involve the same $-1$ component $C_X$
\item $\calM_Y^{J_Y}(C_Y)/\R$ consists of just two elements $C_{Y,1}, C_{Y,2}$ that have opposite signs.
\end{itemize}

For $i = 1,2$, gluing realizes the configuration $(C_{Y,i}, C_X)$ as an end of the moduli space $\calM_X^{\{J_t\}}(C)$, with gluing applying for $|t - t_b|$ sufficiently small and either $t < t_b$ or $t > t_b$ (but not both).
That is, an end of the moduli space $\calM_X^{\{J_t\}}(C)$ with $\pm (t - t_b) > 0$ is compactified by the point $(C_{Y,i}, C_X)$ at $t = t_b$, and it does not extend to $\pm (t - t_b) < 0$.

We assume orientation choices have been made as in \S\ref{subsec:aut_trans}. Together with the canonical orientation on $[0,1]$ this induces an orientation on the $1$-dimensional manifold $\calM_X^{\{J_t\}}(C)$, 
and hence also its compactification $\ovll{\calM}_X^{\{J_t\}}(C)$,
such that $\calM_X^{J_0}(C)$ appears as a negative boundary component (i.e. its sign as a boundary point is the opposite of its sign coming from the orientation on $\calM_X^{J_0}(C)$) and similarly $\calM_X^{J_1}(C)$ appears as a positive boundary component. 
The curves $C_{Y,i},C_X$ also inherit signs $\epsilon(C_{Y,i}),\epsilon(C_X) \in \{-1,1\}$, and by gluing compatibility the sign of each configuration $(C_{Y,i},C_X)$ as a boundary point of $\ovll{\calM}_X^{\{J_t\}}(C)$ matches the product sign $\epsilon(C_{Y,i})\epsilon(C_X)$.
Concretely, the sign associated with the boundary orientation of a boundary point on an oriented $1$-manifold is positive or negative according to whether the orientation points in the outgoing or incoming direction respectively.
Since $C_{Y,1}$ and $C_{Y,2}$ have opposite signs, we have also $\epsilon(C_{Y,1})\epsilon(C_X) \neq \epsilon(C_{Y,2})\epsilon(C_X)$, and hence as boundary points the configurations $(C_{Y,1},C_X)$ and $(C_{Y,2},C_X)$ have opposite orientations.
We then have four possibilities:
\begin{enumerate}[label=(\roman*)]
\item one gluing applies for $t < t_b$ with the corresponding boundary point outgoing, while the other gluing applies for $t > t_b$ with the corresponding boundary point incoming
\item one gluing applies for $t < t_b$ with the corresponding boundary point incoming, while the other gluing applies for $t > t_b$ with the corresponding boundary point outgoing
\item both gluings apply for $t < t_b$, with one corresponding boundary point incoming and the other outgoing
\item both gluings apply for $t > t_b$, with one corresponding boundary point incoming and the other outgoing.
\end{enumerate}
In case (i), by following the cobordism we get a sign-preserving identification of $\calM_X^{J_1}(C)$ with $\calM_X^{J_0}(C)$; case (ii) is similar.
In case (iii), we get a sign-preserving identification of $\calM_X^{J_0}(C)$ with $\calM_X^{J_1}(C)$, plus two extra points of opposite signs; case (iv) is similar.
In any case we have $\#\calM_X^{J_0}(C) = \#\calM_X^{J_1}(C)$.
\end{proof}

\begin{rmk}
One could imagine defining a weaker condition than Definition~\ref{def:formal_pert_inv} which is neither formal in $X$ nor in $Y$. However, this would not suffice for our proof of stabilization (see \S\ref{subsec:stab_ub}), since a priori there could be certain bad degenerations which are ruled out in dimension four for reasons which do not carry over to higher dimensions. 

One could also imagine defining a stronger condition which is formal in both $X$ and $Y$. However, this would be insufficient for our study of convex toric domains, since ``low energy cylinders'' joining an elliptic to a corresponding hyperbolic orbit always occur in the perturbed full rounding $\R \times \bdy\tX_\Omega$ (c.f. Lemma~\ref{lem:low_energy_cyls}).
\end{rmk}

\section{The capacity $\gt_k$}\label{sec:capacity}
In this section we define the main object of study in this paper and establish some of its fundamental properties, in particular proving Theorem~\ref{thm:gt}.
In \S\ref{subsec:basic_props} we give the precise definition of $\gt_k$ and point out its invariance properties. 
We then briefly compare $\gt_k$ with its SFT analogue in \S\ref{subsec:comp_with_SFT}.
\S\ref{subsec:mtone} and \S\ref{subsec:closed_curve_ub} cover the symplectic embedding monotonicity and closed curve upper bound properties, while the proof of the stabilization property occupies \S\ref{subsec:stab_lb} and \S\ref{subsec:stab_ub}.

\subsection{Definition and basic properties}\label{subsec:basic_props}

Given a Liouville domain $(X,\la)$ and a positive constant $c \in \R_{> 0}$, we use the shorthand $c \cdot X$ to denote the Liouville domain $(X, c\la)$.

\begin{definition}\label{def:gt}
Let $X$ be a Liouville domain with nondegenerate contact boundary, and let $D$ be a smooth local symplectic divisor passing through $p \in \Int X$. 
We put 
$$
\gt_k(X) := \sup\limits_{J \in \Jadm(X;D)}\inf\limits_{\Gamma} \calA_{\bdy X}(\Gamma),
$$
where the infimum is over all tuples $\Gamma = (\gamma_1,\dots,\gamma_b)$ of Reeb orbits such that 
\begin{align*}
\ovll{\calM}_X^J(\Gamma)\lll \T^{(k)}p\rrr \neq \nil.
\end{align*}
\end{definition}

\NI Here we put $\calA_{\bdy X}(\Gamma) := \sum_{i=1}^a \calA_{\bdy X}(\gamma_i)$, which is equivalently the energy of any curve with positive ends $\Gamma$.
Recall that $\ovll{\calM}_X^J(\Gamma)\lll \T^{(k)}p\rrr$ and $\Jadm(X;D)$ are defined in \S\ref{subsec:loc_tang_and_E_sk}. 
We emphasize that the moduli spaces 
$\calM_X^J(\Gamma_i)\lll \T^{(k_i)}p\rrr$
are not required to be regular or to have index zero.

\begin{rmk}\label{rmk:gt_reform}
In Definition~\ref{def:gt}, we could alternatively put
$$
\gt_k(X) := \sup\limits_{J \in \Jadm(X;D)}\inf\limits_{\Gamma_1,\dots,\Gamma_a} \left(\calA_{\bdy X}(\Gamma_1) + \dots + \calA_{\bdy X}(\Gamma_a)\right),  
$$
 where the infimum is over all tuples $\Gamma_1 = (\gamma^1_1,\dots,\gamma^1_{b_1}),\dots,\Gamma_a = (\gamma^a_1,\dots,\gamma^a_{b_a})$ of Reeb orbits in $\bdy X$ for which 
the moduli spaces $\calM^J_X(\Gamma_1)\lll \T^{(k_1)}p\rrr,\dots,\calM^J_X(\Gamma_a)\lll \T^{(k_a)}p\rrr$ are nonempty and $k_1,\dots,k_a \in \Z_{\geq 0}$ satisfy $k_1+\dots+k_a \geq k$.
This definition is equivalent and conceptually (if not notationally) cleaner.
Indeed, consider some $C \in \ovll{\calM}_X^J(\Gamma)\lll \T^{(k)}p\rrr$.
If the marked point $z_0$ mapping to $p$ lies on a nonconstant component $C_0$, then we simply note that $C_0$ lies in $\calM_X^J(\Gamma')\lll \T^{(k)}p\rrr$ for some tuple of Reeb orbits $\Gamma'$ satisfying 
$\calA_{\bdy X}(\Gamma') \leq \calA_{\bdy X}(\Gamma)$.
On the other hand, if $z_0$ lies on a ghost component $C_0$, then as in Definition~\ref{def:doublebar_cpctification} we can consider the nearby nonconstant components
$C_i \in \calM_X^J(\Gamma_i)\lll \T^{(k_i)}p\rrr$ for $i = 1,\dots, a$, 
and we necessarily have $\sum_{i=1}^a\calA_{\bdy X}(\Gamma_i) \leq \calA_{\bdy X}(\Gamma)$ and $\sum_{i=1}^a k_i \geq k$.

Conversely, any tuple of curves as above can viewed as an element of the compactified moduli space considered in Definition~\ref{def:gt}. 
\end{rmk}

The quantity $\gt_k(X)$ is manifestly independent of any choice of almost complex structure, and the scaling property $\gt_k(X,\mu \omega) = \mu \gt_k(X,\omega)$ is immediate from the corresponding property for symplectic action.
The nondecreasing property $\gt_1 \leq \gt_2 \leq \gt_3 \leq \cdots$ also follows directly, since by definition any curve satisfying the constraint $\lll \T^{(k)}p\rrr$ for $k \in \Z_{\geq 2}$ also satisfies the constraint $\lll \T^{(k-1)}p\rrr$.
Note that the subadditivity property in Theorem~\ref{thm:gt} is also immediate from Definition~\ref{def:gt}.

A priori $\gt_k$ does depend on the choice of local divisor $D$, but we have:
\begin{lemma}\label{lem:div_indep}
Let $X$ be a Liouville domain with nondegenerate contact boundary. Then $\gt_k(X)$ is independent of the choice of point $p \in \Int X$ and the local divisor $D$. 
\end{lemma}
\begin{proof}  If $p,D$ are fixed, then there is a contractible family of choices for $J_D$.  Further,
given two local symplectic divisors $D,D'$ near $p,p' \in \Int X$ respectively, using Moser's trick we can find a symplectomorphism $\Phi: X \rightarrow X$ which is the identity near $\bdy X$ and which maps the germ of $D$ near $p$ to the germ of $D'$ near $p'$. This induces a bijection $\Jadm(X;D) \overset{\approx}\rightarrow \Jadm(X;D')$
sending $J$ to $\Phi_*J := (d\Phi) \circ J \circ  (d\Phi)^{-1}$, and we get a corresponding bijection $\ovll{\calM}^J_{X}\lll \T^{(k)}p\rrr(\Gamma) \overset{\approx}\rightarrow \ovll{\calM}^{\Phi_*J}_{X}\lll \T^{(k)}p\rrr(\Gamma)$ sending $C$ to $\Phi \circ C$.
\end{proof}

In the next subsection we prove that $\gt_k(X) \leq \gt_k(X')$ whenever $X,X'$ are Liouville domains of the same dimension with nondegenerate contact boundaries for which there is a symplectic embedding $X \hooksymp X'$.
Taking this on faith for the moment, we extend the definition of $\gt_k$ to all symplectic manifolds:

\begin{definition}\label{def:gt_arb_symp_mfd}
If $M$ is any symplectic manifold, we put $$
\gt_k(M) := \sup\limits_{X} \gt_k(X),
$$
 where the supremum is over all Liouville domains $X$ with nondegenerate contact boundary for which there exists a symplectic embedding $X \hooksymp M$.
\end{definition}

\NI Evidently the above definition is consistent with Definition~\ref{def:gt} when $X$ is a Liouville domain with nondegenerate contact boundary (assuming Proposition~\ref{prop:mtone} below). It is also immediate that $\gt_k(M)$ is a symplectomorphism invariant (in particular, in the case of a Liouville domain $(X,\la)$, $\gt_k(X)$ depends on the symplectic form $d\la$ but not on its primitive $\la$).

\begin{rmk}[Local tangency versus skinny ellipsoidal constraints]
In light of \S\ref{subsec:loc_tang_and_E_sk}, to first approximation we can trade (at least in dimension four) the local tangency constraint $\lll \T^{(m)}p\rrr$ in Definition~\ref{def:gt} with a skinny ellipsoidal constraint $\lll (m)\rrr_E$. 
However, the resulting invariant is not immediately equivalent without additional assumptions, and in fact our proof of monotonicity in \S\ref{subsec:mtone} does not a priori apply to skinny ellipsoidal constraints due to the possibility of extra negative ends which bound pseudoholomorphic planes in lower levels.
Nevertheless, it will be fruitful to utilize skinny ellipsoidal constraints in \S\ref{sec:constructing_curves} when computing $\gt_k$ for convex toric domains, and in that setting the relevant moduli spaces are sufficiently nice so that Proposition~\ref{prop:T=E} applies.
\end{rmk}

\subsection{Monotonicity under symplectic embeddings}\label{subsec:mtone}

\begin{prop}\label{prop:mtone}
Let $X$ and $X'$ be Liouville domains of the same dimension with nondegenerate contact boundaries, and suppose there is a symplectic embedding $X \hooksymp \Int X'$. Then for $k \in \Z_{\geq 1}$ we have $\gt_k(X) \leq \gt_k(X')$.
\end{prop}
\begin{proof}
Let $\iota: X \hooksymp \Int X'$ be a symplectic embedding, let $D$ be a local symplectic divisor near $p \in \Int X$, and put $p' := \iota(p)$ and $D' := \iota(D)$.
Given $J \in \Jadm(X;D)$, let $J' \in \Jadm(X',D')$ be an admissible almost complex structure on $\wh{X'}$ which restricts to $\iota_*J$ on $\iota(X)$.
Let $\{J'_t\}_{t \in [0,1)}$ be a family of almost complex structures in $\Jadm(X';D')$ which realizes neck stretching along $\bdy \iota(X)$, with $J'_0 = J'$.
By definition of $\gt_k(X')$, for each $t \in [0,1)$ there is some collection of Reeb orbits $\Gamma^t = (\gamma_1^t,\dots,\gamma_k^t)$ in $\bdy X'$ satisfying $\calA_{\bdy X'}(\Gamma^t) \leq \gt_k(X')$ and $\ovll\calM_{X'}^{J_t'}(\Gamma^t)\lll \T^{(k)}p'\rrr \neq \nil$.
Since $\bdy X'$ has nondegenerate Reeb orbits, there are only finitely many Reeb orbits of action less than any given value, and hence we can find an increasing sequence $t_1,t_2,t_3,\dots \in [0,1)$ with $\lim\limits_{t \rightarrow \infty} t_i = 1$ such that $\Gamma^{t_i} = \Gamma^{t_1}$ is independent of $i$.
By the SFT compactness theorem 
there is some element in the compactified moduli space $\ovll{\calM}_{X'}^{\{J'_t\}}(\Gamma^{t_1})\lll \T^{(k)}p'\rrr$ corresponding to $t = 1$.
This is a pseudoholomorphic building in the broken cobordism $X \notccirc (X' \setminus X)$, and in particular by looking at the components mapping to $X$ we get an element in $\ovll{\calM}_X^{J}(\Gamma^{t_1})\lll \T^{(k)}p\rrr$ with energy at most $\gt_k(X')$.
Since $J$ was arbitrary, we then have $\gt_k(X) \leq \gt_k(X')$.
\end{proof}

\begin{rmk}
Fix any $J_{\bdy X'} \in \Jadm(\bdy X')$, and put $$
\gt_k^{J_{\bdy X'}}(X') := \sup\limits_{J_{X'} \in \Jadm^{J_{\bdy X'}}(X')}\inf\limits_{\Gamma} \calA_{\bdy X'}(\Gamma),
$$
 where the infimum is over all tuples $\Gamma = (\gamma_1,\dots,\gamma_a)$ of Reeb orbits in $\bdy X'$ for which $\ovll\calM^{J_{X'}}_{X'}(\Gamma)\lll \T^{(k)}p\rrr \neq \nil$.
In other words, $\gt_k^{J_{\bdy X'}}(X)$ is defined just like $\gt_k(X')$ except that we take the supremum over almost complex structures having fixed form on the cylindrical end.
Then the above proof actually shows that we have
$\gt_k(X) \leq \gt_k^{J_{\bdy X'}}(X')$.
\end{rmk}
As a consequence of the above remark, by considering symplectic embeddings of $X$ into a slight enlargement of itself we have:
\begin{cor}\label{cor:gt_equals_gt_J_fixed}
For any Liouville domain $X$ with nondegenerate contact boundary, and any $J_{\bdy X} \in \Jadm(\bdy X)$, we have $\gt_k^{J_{\bdy X}}(X) = \gt_k(X)$.
\end{cor}

The symplectic embedding monotonicity property of Theorem~\ref{thm:gt} is now an immediate consequence of Proposition~\ref{prop:mtone} and Definition~\ref{def:gt_arb_symp_mfd}:
\begin{cor}
If $M$ and $M'$ are symplectic manifolds of the same dimension with a symplectic embedding $M \hooksymp M'$, then we have $\gt_k(M) \leq \gt_k(M')$ for any $k \in \Z_{\geq 1}$. 
\end{cor}

\begin{rmk}\label{rmk:cont}
By a standard observation, it also follows that $\gt_k$ is continuous with respect to $C^0$ deformations of $X$ within $\wh{X}$.
\end{rmk}

\begin{rmk}
One could also in principle directly extend Definition~\ref{def:gt} to include all (not necessarily exact) symplectic fillings with nondegenerate contact boundary.
However, a priori our proof of Proposition~\ref{prop:mtone} does not extend,
since in principle there could be infinitely many homology classes with bounded energy.
\end{rmk}

\subsection{Word length filtration}

As in \cite{HSC}, we can also define a refinement $\gt_k^{\leq \ell}$ of $\gt_k$ for any $k,\ell \in \Z_{\geq 1}$ by restricting the allowed number of positive ends.
This gives a more general framework which includes, at least for four-dimensional convex toric domains, both $\{\gt_k\}$ and $\{c_k^\op{GH}\}$ as special cases (see \S\ref{subsec:GH} for more details).

\begin{definition}
Let $X$ be a Liouville domain with nondegenerate contact boundary, and let $D$ be a smooth local symplectic divisor passing through $p \in \Int X$. 
We put 
$$
\gt_k^{\leq \ell}(X) := \sup\limits_{J \in \Jadm(X;D)}\inf\limits_{\Gamma} \calA_{\bdy X}(\Gamma),
$$
 where the infimum is over all tuples $\Gamma = (\gamma_1,\dots,\gamma_a)$ of Reeb orbits in $\bdy X$ for which $\ovll{\calM}^J_{X}(\Gamma)\lll \T^{(k)}p\rrr \neq \nil$, and such that $a \leq \ell$.
\end{definition}
With only minor modifications, our proof Theorem~\ref{thm:gt} also gives the following:

\begin{thm}\label{thm:gt_l}
For each $k,\ell \in \Z_{\geq 1}$, $\gt_k^{\leq \ell}$ is independent of the choice of local divisor and is a symplectomorphism invariant. It satisfies the following properties:
\begin{itemize}
\item
{\bf Scaling:} it scales like area, i.e. $\gt_k^{\leq \ell}(M,\mu \omega) = \mu\gt_k^{\leq \ell}(M,\omega)$ for any symplectic manifold $(M,\omega)$ and $\mu \in \R_{> 0}$.
\item
{\bf Nondecreasing:} we have $\gt_1^{\leq \ell}(M) \leq \gt_2^{\leq \ell}(M) \leq \gt_3^{\leq \ell}(M) \leq \cdots$ for any symplectic manifold $M$.

\item 
{\bf Generalized Liouville embedding monotonicity:} Given equidimensional Liouville domains $X,X'$ and a generalized Liouville embedding of $X$ into $X'$ (see Remark~\ref{rmk:reln_GH}), we have $\gt_k^{\leq \ell}(X) \leq \gt_k^{\leq \ell}(X')$.

\item
{\bf Stabilization:} For any Liouville domain $X$ we have $\gt_k(X \times B^{2}(c)) = \gt_k(X)$ for any $c \geq \gt_k(X)$, provided that the hypotheses of Proposition~\ref{prop:stab_ub} (substituting $\gt_k$ with $\gt_k^{\leq \ell}$) are satisfied.
\end{itemize}
\end{thm}

\NI Compared with Theorem~\ref{thm:gt}, for a general symplectic embedding $X \hooksymp X'$ there may be curves in $X' \setminus X$ having no positive ends, and a curve with $\ell$ positive ends in $X'$ may produce a curve in $X$ with a greater number of 
positive 
ends after neck stretching 
since the top of the limiting building might contain a component with no positive ends.
Generalized Liouville embeddings carry an additional an exactness condition which precisely rules out curves in $X' \setminus X$ without 
 positive ends via Stokes' theorem. 

Note that if $X^{2n \geq 4}$ is a star-shaped domain then a symplectic embedding $X \hooksymp X'$ is automatically a generalized Liouville embedding, but this does not necessarily extend to cases with $H^1(\bdy X;\R)$ nontrivial.
Moreover, if $\bdy X$ has no contractible Reeb orbits then we have $\gt_k^{\leq 1}(X) = c_k^{\op{GH}}(X) = \infty$, 
and hence these capacities contain no quantitative information; $\gt_k^{\leq \ell}(X)$ is more often finite for $l$ sufficiently large.

\subsection{Comparison with SFT counterpart}\label{subsec:comp_with_SFT}
At first glance the definitions of $\gt_k$ and $\gapac_k$ look rather different, despite involving the same types of curves. 
Recall that $\gapac_k(X)$ is defined in \cite{HSC} using the $\calL_\infty$ algebra structure on the linearized contact homology chain complex $\chlin(X)$ of a Liouville domain $X$, along with the induced $\calL_\infty$ homomorphism $\auglin\lll \T^{(k)}p\rrr: \chlin(X) \rightarrow \K$ defined by counting rational curves with a local tangency constraint $\lll \T^{(k)}p\rrr$.
In brief, $\gapac_k(X)$ is the minimal action of an element of the bar complex $\bar\chlin(X)$ which is closed under the bar differential and whose 
image under the chain map $\bar\chlin(X) \rightarrow \K$ induced by $\auglin\lll \T^{(k)}p\rrr$ is nonzero.
Here $\bar\chlin(X)$ as a vector space is the (appropriately graded) symmetric tensor algebra on the vector space $\chlin(X)$ spanned by good Reeb orbits in $\bdy X$, and the bar differential is built out of the $\Li$ structure maps $\ell^1,\ell^2,\ell^3$ which count pseudoholomorphic buildings in $\R \times \bdy X$, anchored in $X$, with one negative and several positive ends.
In particular, this definition of $\gapac_k(X)$ typically requires virtual perturbations in order to set up the chain complex $\chlin(X)$ along with its $\Li$ structure, and its basic invariance and structural properties follow naturally from SFT functoriality. 

The precise virtual perturbation framework is not important for our present discussion, but we mention two important axioms: (a) a structure coefficient can only be nonzero if the corresponding SFT compactified moduli space is nonempty, and (b) if the naive pseudoholomorphic curve count for a given structure coefficient is already regular and there are other representatives in its corresponding SFT compactified moduli space, then this count remains valid after turning on virtual perturbations.
It is then easy to deduce that $\gt_k(X) \leq \gapac_k(X)$ for any Liouville domain $X$. Indeed, for any $J$, by (a) and the definition of $\gapac_k(X)$ there must be a pseudoholomorphic building $C \in \ovll{\calM}_X^J(\Gamma)\lll \T^{(k)}p\rrr$ having total energy at most $\gapac_k(X)$.
Since $J$ is arbitrary, we therefore have $\gt_k(X) \leq \gapac_k(X)$.

In principle we could have $\gt_k(X) < \gapac_k(X)$, if all curves in $\wh{X}$ with energy $\gt_k(X)$ are undetected by $\gapac_k(X)$. However, this cannot occur if $\gt_k(X)$ is carried by a suitably nice moduli space, e.g. as in Proposition~\ref{prop:stab_ub}.
In particular, it follows from the results of this paper that $\gt_k(X) = \gapac_k(X)$ whenever $X$ is a four-dimensional convex toric domain; we are not currently aware of any Liouville domain $X$ for which $\gt_k(X) \neq \gapac_k(X)$.

\subsection{Upper bounds from closed curves}\label{subsec:closed_curve_ub}

Here we prove the closed curve upper bound part of Theorem~\ref{thm:gt}. 
Recall from the introduction that $N_{M,A}\lll \T^{(k)}p\rrr \neq 0$ counted the number of curves in class $A$ that 
are tangent to the local divisor $D$ at $p$ to order $k$.

\begin{prop}
If $(M,\omega)$ is a closed semipositive symplectic manifold satisfying $N_{M,A}\lll \T^{(k)}p\rrr \neq 0$ for some $A \in H_2(M)$, then we have $\gt_k(M) \leq [\omega] \cdot A$.
\end{prop}
\begin{proof}
This is quite similar to the proof of Proposition~\ref{prop:mtone}.
It suffices to show that for any Liouville domain $X$ with nondegenerate contact boundary which admits a symplectic embedding $\iota: X \hooksymp M$, we have 
$\gt_k(X) \leq [\omega] \cdot A$.
Given $J \in \Jadm(X;D)$, we extend $\iota_*J$ to a compatible almost complex structure $J'$ on $M$.
 Let $\{J_t\}_{t \in [0,1)}$ be a family of compatible almost complex structures on $M$ realizing neck stretching along $\bdy \iota(X)$, with $J_0 = J'$.
Note that $\calM^{J_t}_{M,A}\lll \T^{(k)}p\rrr$ is nonempty for all $t \in [0,1)$, since otherwise this moduli space would be empty and in particular regular, contradicting the invariance of $N_{M,A}\lll \T^{(k)}p\rrr$ (see \cite[\S2.2]{McDuffSiegel_counting}).
Then, as in the proof of Proposition~\ref{prop:mtone}, the SFT compactness theorem implies that there must be a limiting building corresponding to $t = 1$, and in particular in the bottom level we can find $C \in \ovll{\calM}_X^J(\Gamma)\lll \T^{(k)}p\rrr$ for some tuple of Reeb orbits satisfying $\calA_{\bdy X}(\Gamma) \leq [\omega] \cdot A$.
\end{proof}

\subsection{Stabilization lower bounds}\label{subsec:stab_lb}

\begin{prop}\label{prop:stab_lb}
For any Liouville domain $X$, we have
$\gt_k(X \times B^{2}(c)) \geq \gt_k(X)$ for all $k\ge 1$ provided that $c \geq \gt_k(X)$.
\end{prop}

As a preliminary step, the next lemma allows us to identify the Reeb orbits after stabilizing (and suitably smoothing the corners)
with those before stabilizing, plus additional orbits of large action.
We denote by $\la_\std = \tfrac{1}{2}(xdy - ydx)$ the standard Liouville form on $B^{2}(c)$.
Given a Liouville form $\la$, recall that the Liouville vector field $V_\la$ is characterized by $d\la(V_\la,-) = \la$.

Suppose that $(Y,\alpha)$ is a strict contact manifold and $Z \subset Y$ is a submanifold of codimension 
$2$ such that $\alpha|_Z$ is a contact form on $Z$ and the Reeb vector field $R_\alpha$ is tangent to $Z$.  
Let $\xi_Y := \ker \alpha $ and $\xi_Z := \ker \alpha|_Z$ denote the contact hyperplane distributions of $Y$ and $Z$ respectively.
Since $\xi_Z$ is a subbundle of $\xi_Y$, we can consider its orthogonal complement $\xi_Z^\perp$ with respect to the symplectic form $d\alpha|_{\xi_Y}$.
Let $\gamma$ be a nondegenerate Reeb orbit of $Y$ which lies in $Z$, and let $\tau$ be a trivialization of the symplectic vector bundle $\gamma^*\xi_Y$ which splits as $\tau = \tau_Z + \tau_Z^\perp$ with respect to the direct sum decomposition $\xi_Y = \xi_Z \oplus \xi_Z^\perp$.
Since the latter decomposition is also preserved by the linearized Reeb flow of $Y$ along $\gamma$, the trivialization $\tau_Z^\perp$ in the normal direction identifies the linearized Reeb flow along $\gamma^*\xi_Z^\perp$ with a loop of 
$2\times 2$ symplectic matrices which starts at the identity and ends at a matrix without $1$ as an eigenvalue. 
Such a loop has a well-defined Conley--Zehnder index which is called the {\bf normal Conley--Zehnder index} of $\gamma$, denoted by $\cz_{\tau_Z^\perp}^{\perp}(\gamma)$. 

In the following we show that Reeb orbits of $\bdy X$ can be viewed as Reeb orbits in a suitable smoothing of $\bdy (X \times B^{2}(c))$,
and we apply the above discussion with $Y$ given by the smoothing of $\bdy (X \times B^{2}(c))$ and $Z$ given by $\bdy X$.
In this situation, there is a canonical trivialization of $\xi_Z^\perp$, coming from its identification with 
the normal bundle of $Z \subset Y$, which in turn is naturally identified with the restriction to $Z$ of $\{0\} \times TB^{2}(c) \subset TX \times TB^{2}(c)$.
By default we will always measure normal Conley--Zehnder indices by working with a split trivialization $\tau = \tau_Z + \tau_Z^\perp$ of $\gamma^*\xi_Y$, where $\tau_Z^\perp$ comes from this canonical trivialization of $\xi_Z^\perp$.

\begin{lemma}\label{lem:smoothing}
Let $(X,\la)$ be a Liouville domain.
For any $c,\eps \in \R_{> 0}$, there is a subdomain with smooth boundary $\wt{X} \subset X \times B^{2}(c)$ such that
\begin{itemize}
\item the Liouville vector field 
$V_{\la} + V_{\la_\std}$ is outwardly transverse along $\bdy \wt{X}$ 
\item $X \times \{0\} \subset \wt{X}$ and the Reeb vector field of $\bdy \wt{X}$ is tangent to $\bdy X \times \{0\}$
\item
any Reeb orbit of the contact form $(\la + \la_\std)|_{\bdy \wt{X}}$ with action less than $c - \eps$ is entirely contained in $\bdy X \times \{0\}$ and has normal Conley--Zehnder index equal to $1$. 
\end{itemize}
\end{lemma}
\begin{proof}
For notational convenience put $X_1 := X$ and $X_2 := B^{2}(c)$. We denote the associated Liouville forms by $\la_i$, the associated contact forms by $\alpha_i := \la_i|_{\bdy X_i}$,
and the associated Liouville vector fields by $V_{\la_i}$ for $i=1,2$.
Note that every closed Reeb orbit of $\bdy X_2$ has action at least $c$.
 
Recall that we can use the Liouville flow to identify a collar neighborhood $U_i$ of $\bdy X_i$ with $(-\eta,0] \times \bdy X_i$ for some small $\eta >0$, and under this identification we have $\la_i = e^{r_i}\alpha_i$, where $r_i$ denotes the coordinate on the first factor.
Given a smooth function $H_i: (-\eta,0] \times \bdy X_i \rightarrow \R$ of the form $H(r_i,y_i) = h(e^{r_i})$ for some $h_i: (e^{-\eta},1] \rightarrow \R$, the Hamiltonian vector field takes the form $\calX_{H_i} = h'_i(e^{r_i})R_{\alpha_i}$, where $R_{\alpha_i}$ is the Reeb vector field of $\alpha_i$. 
Note that for such a Hamiltonian we have $V_{\la_i}(H_i) = \la_i(\calX_{H_i}) = e^{r_i} h_i'(e^{r_i})$. 

By considering functions which depend only on the Liouville flow coordinate $r_i$ near the boundary and are otherwise sufficiently small, we can find smooth functions $H_i: X_i \rightarrow [0,1]$ for $i=1,2$ such that:
\begin{itemize} 
\item[(a)]
$\bdy X_i = H_i^{-1}(1)$ is a regular level set 
\item[(b)]
$H_i^{-1}(0) = \{p_i\}$ is a nondegenerate minimum, where we assume $p_2 = 0 \in B^{2}(c)$
\item[(c)] on $U_i \approx (-\eta,0] \times \bdy X_i$ we have $H_i(r_i,y_i) = h_i(e^{r_i})$ for some $h_i: (e^{-\eta},1] \rightarrow [0,1]$ with $h_i' > 0$
\item[(d)] on $X_i \setminus U_i$ we have $|V_{\la_i}(H_i)| < \eps/2$
\item[(e)] we have $H_i^{-1}([\delta,1]) \subset U_i$ for some small $\delta > 0$, 
and on $H_i^{-1}([\delta,1])$ we have $V_{\la_i}(H_i) > c + \eps$.
\end{itemize}
We can further arrange:
\begin{itemize}
\item[(f)]
$V_{\la_2}(H_2) > 0$ on $B^{2}(c) \setminus \{0\}$
\item[(g)]
for every $T$-periodic $\gamma_2$ orbit of $\calX_{H_2}$ with $T \leq 1$, we have $\int_{\gamma_2}\la_2 > c - \eps/2$
\item[(h)] using standard symplectic coordinates $x,y$, on a small neighborhood of $0 \in B^{2}(c)$ we have
$$H_2(x,y) = \tfrac{1}{2}\rho(x^2+y^2),$$ with $\rho < \pi$.
\end{itemize}

Put $\wt{X} := \{(x_1,x_2) \in X_1 \times X_2\;|\; H_1(x_1) + H_2(x_2) \leq 1\}$. 
It follows from the above properties that $\wt{X}$ has smooth boundary, and 
we have $$
(V_{\la_1} + V_{\la_2})(H_1 + H_2) > 0
$$
 along $\bdy \wt{X}$.
Indeed, consider $(x_1,x_2) \in \bdy \wt{X}$, and suppose first that $x_1 \in U_1$. Then we have $(V_{\la_1})_{x_1}(H_1) = e^{r_1}h_1'(e^{r_1}) > 0$ by (c) and $(V_{\la_2})_{x_2}(H_2) \geq 0$ by (f).
On the other hand, if $x_1 \in X_1 \setminus U_1$, then we must have $H_1(x_1) \in [0,\delta]$ by (e) and $|(V_{\la_1})_{x_1}(H_1)| < \eps/2$ by (d).
In this case we have $H_2(x_2) = 1-H_1(x_1) \in [1-\delta,1]$, whence $(V_{\la_2})_{x_2}(H_2) > c+\eps$ and therefore $(V_{\la_1})_{x_1}(H_1) + (V_{\la_2})_{x_2}(H_2) > 0$.

It follows from the above discussion that $\la + \la_\std$ is a Liouville form on $\wt{X}$, and in particular it restricts to a positive contact form on $\bdy \wt{X}$.
Observe that the corresponding Reeb vector field is at each point in $\bdy \wt{X}$ proportional to the Hamiltonian vector field of $H_1 + H_2$.
In particular, this is tangent to $\bdy (X \times \{0\})$, since along $\bdy(X \times \{0\})$ we have $\calX_{H_2} \equiv 0$.

We now prove the assertion about actions of Reeb orbits. Suppose that $\gamma$ is a $T$-periodic Reeb orbit of $\bdy \wt{X}$ for some $T \in \R_{> 0}$, let $\gamma_i$ denote its projection to $X_i$ for $i = 1,2$.
Note that we have $\gamma_i \subset H_i^{-1}(C_i)$ for some $C_i \in [0,1]$ with $C_1 + C_2 = 1$.
If $\gamma_2$ is constant, then $\gamma$ lies in $X_1 \times \{0\}$.
Otherwise, if $C_1 \in [0,\delta]$, then $C_2 \in [1-\delta,1]$, and we have
$|\int_{\la_1}\gamma_1| < T\eps/2$ by (d) and $\int_{\la_2}\gamma_2 > \max(c-\eps/2,T(c+\eps))$ by (g), and therefore we have
$$\int_\gamma \la = \int_{\gamma_1}\la_1 + \int_{\gamma_2}\la_2 > \max(c-\eps/2,T(c+\eps)) - T\eps/2 > c-\eps.$$
Lastly, if $C_1 \in [\delta,1]$ and $\gamma_2$ is not constant,
then if $T \geq 1$ we have $$\int_{\gamma}\la \geq \int_{\gamma_1}\la_1 > T(c+\eps) > c-\eps,$$ whereas if $T < 1$ then we have 
$$\int_{\gamma}\la \geq \int_{\gamma_2}\la_2 > c - \eps/2 > c-\eps.$$

As for the assertion about normal Conley--Zehnder indices, suppose that $\gamma$ is a Reeb orbit in $\bdy(X \times \{0\})$ with action $T \leq c$. 
Observe that Reeb vector field on $\bdy \wt{X}$ is given by $\tfrac{1}{\la_1(\calX_{H_1}) + \la_2(\calX_{H_2})}(\calX_{H_1} + \calX_{H_2})$, and 
along $\bdy(X \times \{0\})$ we have $\la_1(\calX_{H_1}) > c + \eps$ and $\la_2(\calX_{H_2}) = 0$.
We can therefore identify the linearized Reeb flow along $\gamma$ in the normal direction with the time-$T$ linearized Hamiltonian flow of $\tfrac{1}{\la_1(\calX_{H_1})}\calX_{H_2}$ at $0$.
By design, this is rotation by the angle $\tfrac{T\rho}{\la_1(\calX_{H_1})}$.
In particular, the Conley--Zehnder contribution for each factor is $1$ provided that we have $\tfrac{T\rho}{\la_1(\calX_{H_1})} < \pi$, for which $\rho < \pi$ suffices.
\end{proof}

In the sequel, we will denote any Liouville domain $\wt{X}$ satisfying the properties of Lemma~\ref{lem:smoothing} for some $\eps > 0$  sufficiently small by $X \smx B^{2}(c)$.

\begin{lemma}\label{lem:lie_in_slice}
Let $X$ be a Liouville domain, and let $X \smx B^{2}(c)$ be a smoothing of $X \times B^{2}(c)$ as in Lemma~\ref{lem:smoothing}. 

\begin{itemlist}

\item[{\rm (i)}] Let $J \in \Jadm(X \smx B^{2}(c))$ be an admissible almost complex structure on the symplectic completion of $X \smx B^{2}(c)$ for which $\wh{X} \times \{0\}$ is $J$-holomorphic. Let $C$ be an asymptotically cylindrical $J$-holomorphic curve in $\wh{X}$, all of whose asymptotic Reeb orbits are nondegenerate and lie in $\bdy X \times \{0\}$ with normal Conley--Zehnder index $1$.
Then $C$ is either disjoint from the slice $\wh{X} \times \{0\}$ or entirely contained in it. 

\item[{\rm (ii)}] Let $J \in \Jadm(\bdy(X \smx B^{2}(c)))$ be an admissible almost complex structure on the symplectization of $\bdy(X \smx B^{2}(c))$ for which $\R \times \bdy X \times \{0\}$ is $J$-holomorphic. Let $C$ be an asymptotically cylindrical $J$-holomorphic curve in $\R \times \bdy(X \smx B^{2}(c))$, all of whose asymptotic Reeb orbits are nondegenerate and lie in $\bdy X \times \{0\}$ with normal Conley--Zehnder index $1$. Then $C$ is either disjoint from the slice $\R \times \bdy X \times \{0\}$ or entirely contained in it. Moreover, only the latter is possible of $C$ has at least one negative puncture.
\end{itemlist}

\end{lemma}
To prove Lemma~\ref{lem:lie_in_slice}, we invoke the higher dimensional extension of \cite{Sief} (c.f the exposition in \cite[\S2]{MorS}).
Namely, let $C$ be an asymptotically cylindrical curve in the symplectic completion of $X \smx B^2(c)$ or the symplectization of $\bdy(X \smx B^2(c))$,
and let $Q$ denote the divisor $\wh{X} \times \{0\}$ or $\R \times \bdy (X \times \{0\})$ respectively.
Assume that each puncture of $C$ is asymptotic to a nondegenerate Reeb orbit in $\bdy X \times \{0\}$, and that $C$ is not entirely contained in $Q$.
For each puncture $z$ of $C$, we can consider the corresponding asymptotic winding number $\wind_z$ around $Q$ as we approach the puncture, as measured by the canonical trivialization discussed in the leadup to Lemma~\ref{lem:smoothing}.

We will need the following facts:
\begin{enumerate}[label=(\alph*)]
\item the curve $C$ intersects $Q$ in only finitely many points, each of which has a positive local intersection number
\item if $z$ is a positive puncture and $\gamma_z$ is the corresponding asymptotic Reeb orbit, we have $\wind_z \leq \lfloor \cz^\perp(\gamma_z)/2 \rfloor$
\item if $z$ is a negative puncture and $\gamma_z$ is the corresponding asymptotic Reeb orbit, we have $\wind_z \geq \lceil \cz^\perp(\gamma_z)/2 \rceil$
\item we have 
\begin{align}\label{eq:pushoff}
\op{push}(C) \cdot Q = C \cdot Q - \sum_{z\text{ pos. punc.}} \wind_z + \sum_{z\text{ neg. punc.}} \wind_z,
\end{align}
where $\op{push}(C)$ is a pushoff of $C$ whose direction near each puncture is a nonzero constant with respect to the canonical trivialization of the normal bundle.
\end{enumerate}
Here $C \cdot Q$ and $\op{push}(C) \cdot Q$ denote homological intersection numbers, i.e. the sum of local homological intersection numbers over all (necessarily finitely many) intersection points.
In particular, we have $\op{push}(C) \cdot Q = 0$ since there is an obvious displacement of $C$ from $Q$ which takes the specified form near each of the punctures.

The last fact (d) is elementary topology. 
The proof of (a) follows from an asymptotic description of $C$ in the normal direction near each puncture, which is written in terms of an eigenfunction of the corresponding normal asymptotic operator. 
Properties (b) and (c) follow from a characterization of normal Conley--Zehnder indices in terms of the corresponding normal asymptotic operators, together with bounds on the winding numbers of their eigenfunctions.

\begin{proof}[Proof of Lemma~\ref{lem:lie_in_slice}]

To prove (i), suppose that $C$ is not contained in $Q := \wh{X} \times \{0\}$.
Since each puncture of $C$ is positively asymptotic to a Reeb orbit in $\bdy X \times \{0\}$ with normal Conley--Zehnder index $1$, using \eqref{eq:pushoff} and (b) we have
\begin{align*}
0 = \op{push}(C) \cdot Q &= C \cdot Q - \sum_{z\text{ pos. punc.}} \wind_z\\ & \geq C \cdot Q - \sum_{z\text{ pos. punc.}} \lfloor \tfrac{1}{2} \rfloor\\ &= C \cdot Q,
\end{align*}
and hence $C \cdot Q \leq 0$. Since each local intersection between $C$ and $Q$ counts positively, this is only possible if $C$ is disjoint from $Q$.

The proof of (ii) is similar. Assume that $C$ is not contained in $Q := \R \times \bdy X \times \{0\}$. Using \eqref{eq:pushoff} we have
\begin{align*}
0 \geq C \cdot Q - \sum_{z\text{ pos. punc.}} \lfloor \tfrac{1}{2} \rfloor + \sum_{z\text{ neg. punc.}} \lceil \tfrac{1}{2} \rceil =  C \cdot Q + \sum_{z\text{ neg. punc.}} 1.
\end{align*}
This is only possible if $C$ has no negative punctures and $C$ is disjoint from $Q$.
\end{proof}

\begin{proof}[Proof of Proposition~\ref{prop:stab_lb}]

We can assume $c > \gt_k(X)$ and that $\bdy X$ is nondegenerate, since then  the result follows by continuity (c.f. Remark~\ref{rmk:cont}).
Let $X \smx B^{2}(c)$ be a smoothing of $X \times B^{2}(c)$ as in Lemma~\ref{lem:smoothing}, with $\eps > 0$ chosen sufficiently small so that $c - \eps > \gt_k(X)$.
Let $D$ be a local divisor near $p \in \Int X$, and let us take the local divisor $\wt{D}$ in $X \smx B^2(c)$ near $\wt{p} := (p,0)$ to be of the form $D \times B^{2}(\delta) \subset X \smx B^{2}(c)$ for some small $\delta > 0$.

Let $J_X \in \Jadm(X;D)$ be such that for every tuple of Reeb orbits $\Gamma$ such that $\ovll{\calM}_X^{J_X}(\Gamma)\lll \T^{(k)}p\rrr \neq \nil$ we have $\calA_{\bdy X}(\Gamma) \geq \gt_k(X)$.
Pick $\wt{J} \in \Jadm(X \smx B^{2}(c);\wt{D})$ such that $\wh{X} \times \{0\}$ is $\wt{J}$-holomorphic with $\wt{J}|_{\wh{X} \times \{0\}} = J_X$. 
It suffices to show that for any tuple of Reeb orbits $\Gamma'$ for which $\ovll{\calM}^{\wt{J}}_{X \smx B^{2}(c)}(\Gamma')\lll \T^{(k)}p\rrr \neq \nil$, we have $\calA_{\bdy (X \smx B^{2}(c))}(\Gamma') \geq \gt_k(X)$, since then we have $\gt_k(X \times B^{2}(c)) \geq \gt_k(X \smx B^{2}(c)) \geq \gt_k(X)$. 

Consider $C \in \ovll{\calM}^{\wt{J}}_{X \smx B^{2}(c)}(\Gamma')\lll \T^{(k)}p\rrr$.
For some $a \in \Z_{\geq 1}$, let $C_i \in \calM_{X \smx B^{2}(c)}^{\wt{J}}(\Gamma_i)\lll \T^{(k_i)}p\rrr$ for $i = 1,\dots a$ be nonconstant components of $C$ with $\sum_{i=1}^a k_i \geq k$ and 
$\sum_{i=1}^a E(C_i) \leq E(C)$ as in Remark~\ref{rmk:gt_reform}.
We need to establish the bound $\sum_{i=1}^a E(C_i) \geq \gt_k(X)$.
If any positive end of some $C_i$ is not asymptotic to the slice $\wh{X} \times \{0\}$, then the corresponding Reeb orbit must have action at least $c - \eps$, and hence $E(C_i) \geq c - \eps > \gt_k(X)$.
Otherwise, by Lemma~\ref{lem:lie_in_slice}, each $C_i$ must be entirely contained in $\wh{X} \times \{0\}$ (note that it cannot be disjoint from the slice due to the local tangency constraint at $p\in \wh{X}\times \{0\}$).
By our choice of $\wt{D}$, each $C_i$ then corresponds to a $J_X$-holomorphic curve in $\wh{X}$ satisfying the constraint $\lll \T^{(k_i)}p\rrr$ with local divisor $D$, from which the desired bound readily follows.
\end{proof}

\subsection{Stabilization upper bounds}\label{subsec:stab_ub}
In order to prove the stabilization property in Theorem~\ref{thm:gt}, we need to complement Proposition~\ref{prop:stab_lb} by proving an upper bound.
Our proof will require some additional assumptions which amount to saying that the capacity $\gt_k(X)$ is represented by elements in a well-behaved moduli space of curves.
Indeed, without such an assumption, after stabilizing and perturbing the almost complex structure it is conceivable that all curves with energy equal to $\gt_k(X)$ disappear, resulting in $\gt_k(X \times B^2(c)) > \gt_k(X)$.

\begin{prop}\label{prop:stab_ub}
Let $X$ be a Liouville domain, put $Y := \bdy X$, and let $C$ be a simple index zero formal curve component in $X$ with constraint $\lll\T^{(k)}p\rrr$ for some $k \in \Z_{\geq 1}$, such that $E_X(C) = \gt_k(X)$. 
Assume further that the following conditions hold:
\begin{enumerate}[label=(\alph*)]
\item $C$ is formally perturbation invariant with respect to some generic $J_{Y} \in \Jadm(Y)$ (c.f. \S\ref{subsec:pert_inv})
\item the moduli space $\calM_X^{J_{X}}(C)$ is regular and finite with nonzero signed count $\#\calM_X^{J_{X}}(C)$ for some $J_X \in \Jadm^{J_{Y}}(X;D)$.
\end{enumerate}
Then we have $\gt_{k}(X \times B^{2}(c)) \leq \gt_k(X)$ for any 
$c \in \R_{> 0}$.
The same conclusion also holds if we instead assume that the hypotheses hold with $k$ replaced by some divisor $\ell$ of $k$ such that $\gt_k(X) = \tfrac{k}{\ell}\gt_\ell(X)$.
\end{prop}

The last part of Proposition~\ref{prop:stab_ub} follows easily from the existence of multiple covers, or as a special case of subadditivity.
\begin{proof}[Proof of Proposition~\ref{prop:stab_ub}]
By monotonicity of $\gt_k$ under symplectic embeddings,
it suffices to prove establish
$\gt_k(\wt{X}) \leq \gt_k(X)$
for $\wt{X} := X \smx B^2(c)$ with $c$ arbitrarily large.
In particular, we can assume that any Reeb orbit in $\wt{Y} := \bdy \wt{X}$ which is not contained in $Y \times \{0\}$ has action greater than $\gt_k(X)$.

Let $J_{\wt{Y}} \in \Jadm(\wt{Y})$ be an almost complex structure which agrees with $J_Y$ on $\R \times Y \times \{0\}$.
By Corollary~\ref{cor:gt_equals_gt_J_fixed} we have $\gt_k(\wt{X}) = \gt^{J_{\wt{Y}}}_k(\wt{X})$, so it suffices to prove $\gt^{J_{\wt{Y}}}_k(\wt{X}) \leq \gt_k(X)$.

Let $J_{\wt{X}} \in \Jadm^{J_{\wt{Y}}}(\wt{X};\wt{D})$ be an admissible almost complex structure which agrees with $J_X$ on $\wh{X} \times \{0\}$.
Here we put $\wt{D} := D \times B^2(\delta)$ with $\delta > 0$ small as in the proof of  Proposition~\ref{prop:stab_lb}.
Since Reeb orbits of $Y$ can also be viewed as Reeb orbits of $\wt{Y}$, $C$ naturally corresponds to a formal curve component $\wt{C}$ in $\wt{X}$.
Note that $\wt{C}$ is again simple and has index zero, the latter being a consequence of the index formula and the fact that the Reeb orbits of $C$ have normal Conley--Zehnder index $1$ by Lemma~\ref{lem:smoothing}.

Moreover, we claim that $\wt{C}$ is formally perturbation invariant with respect to $J_{\wt{Y}}$.
Indeed, let $\Gamma$ (resp. $\wt{\Gamma}$) denote the positive asymptotic orbits of $C$ (resp. $\wt{C}$), let $A$ (resp. $\wt{A}$) denote its homology class, and let $\wt{C}' \in \ovll{\formal}_{\wt{X},\wt{A}}(\wt{\Gamma})\lll \T^{(k)}p\rrr$ be a hypothetical stable formal building satisfying conditions (A1) and (A2) of Definition~\ref{def:formal_pert_inv}. 
By action considerations we can assume that each asymptotic Reeb orbit involved in $\wt{C}'$ lies in $Y \times \{0\}$, and hence $\wt{C}'$ naturally corresponds to a stable formal building $C' \in \ovll{\formal}_{X,A}(\Gamma)\lll \T^{(k)}p\rrr$.
In particular, by formal perturbation invariance of $C$, we have either $C' = C$ (whence $\wt{C}' = \wt{C}$) or else $C'$ is a two-level building as in Definition~\ref{def:formal_pert_inv}(B2), with top level consisting of a 
union of a simple index $1$ component $C_Y$ and possibly some trivial cylinders, and moreover $\calM_Y^{J_Y}(C_Y)$ is regular and satisfies $\# \calM_Y^{J_Y}(C_Y)/\R = 0$.
Let $\wt{C}_Y$ denote the analogue of $C_Y$ in $\wt{Y}$.
By Lemma~\ref{lem:lie_in_slice}(ii), every curve in $\calM_{\wt{Y}}^{J_{\wt{Y}}}(\wt{C}_Y)$ must be contained in the slice $\R \times Y \times \{0\}$ because it has a negative end. In particular, we have a natural identification $\calM_{\wt{Y}}^{J_{\wt{Y}}}(\wt{C}_Y) \approx \calM_{Y}^{J_{Y}}(C_Y)$, and since each curve in $\calM_{\wt{Y}}^{J_{\wt{Y}}}(\wt{C}_Y)$ is also regular by Proposition~\ref{prop:reg_after_stab_symp} we have $\#\calM_{\wt{Y}}^{J_{\wt{Y}}}(\wt{C}_Y) / \R = 0$.
This establishes the above claim that $\wt{C}$ is formally perturbation invariant with respect to $J_{\wt{Y}}$.

Invoking now Lemma~\ref{lem:lie_in_slice}(i), we have a natural identification
$\calM_{\wt{X}}^{J_{\wt{X}}}(\wt{C}) \approx \calM_X^{J_X}(C)$, and the former is also regular by Proposition~\ref{prop:reg_after_stab}.
In particular, we have $\#\calM_{\wt{X}}^{J_{\wt{X}}}(\wt{C}) \neq 0$,
so by Proposition~\ref{prop:count_J_indep} we conclude that $\calM^{\wt{J}}_{\wt{X}}(\wt{C}) \neq \nil$ for all $\wt{J} \in \Jadm^{J_{\wt{Y}}}(\wt{X};\wt{D})$.
In particular, it follows that we have
\begin{align*}
\gt_k^{J_{\wt{Y}}}(\wt{X}) \leq E_{\wt{X}}(\wt{C}) = E_X(C) = \gt_k(X),
\end{align*}
as needed.
\end{proof}

\section{Fully rounding, permissibility, and minimality}\label{sec:rounding}

In this section we develop our main tools for getting lower bounds on the capacities of convex toric domains. In \S\ref{subsec:fr} we explain the fully rounding procedure, which standardizes the Reeb dynamics. In \S\ref{subsec:permis} we discuss the extent to which curves are obstructed by the relative adjunction formula and writhe bounds.
Lastly, in \S\ref{subsec:min_words} we analyze those words of Reeb orbits having minimal action for a given index. 
The proof that these minimal action words can all be represented by curves is deferred to \S\ref{sec:constructing_curves}.

\subsection{The fully rounding procedure}\label{subsec:fr}

We consider a four-dimensional\footnote{We note that the discussion in this subsection generalizes very naturally to higher dimensions, but for concreteness we restrict our exposition to dimension four.}
{\bf convex toric domain}, i.e. a subdomain of $\C^2$ of the form $X_\Omega := \mu^{-1}(\Omega)$, where
\begin{itemize}
\item
$\mu: \C^2 \rightarrow \R_{\geq 0}^2$ is the standard moment map defined by $\mu(z_1,z_2) = (\pi|z_1|^2,\pi|z_2|^2)$
\item 
$\Omega \subset \R_{\geq 0}^2$ is a subdomain such that 
$$\wh{\Omega} := \{(x_1,x_2) \in \R^2\;|\; (|x_1|,|x_2|) \in \Omega\} \subset \R^2$$
is compact and convex.
\end{itemize}
We equip $X_\Omega$ with the restriction of the standard Liouville form $\la_\std = \tfrac{1}{2}(xdy - ydx)$ on $\C^2$.
For example, if $\Omega \subset \R^2$ is a rational triangle with vertices $(0,0),(a,0),(0,b)$, then $X_{\Omega}$ is the ellipsoid $E(a,b) \subset \C^2$.

The ``fully rounding procedure'' replaces $X_\Omega$ with a $C^0$-small perturbation whose Reeb orbits are indexed in a straightforward way which is essentially insensitive to the shape of $\Omega$. 
We proceed in two steps: 
\begin{enumerate}
\item
replace $X_\Omega$ with another convex toric domain $X_\Omega^\fr := X_{\Omega^\fr}$, where $\Omega^\fr \subset \R_{\geq 0}^2$ is a $C^0$-small perturbation of $\Omega$ with smooth boundary as in \cite[Fig. 5.1]{chscI} (see also \cite[\S2.2]{Gutt-Hu})
\item let $\wt{X}_\Omega$ denote the result after a further $C^0$-small smooth perturbation of $X_\Omega^\fr$ which replaces each Morse--Bott circle of Reeb orbits of action less than some large constant $K$ with two nondegenerate Reeb orbits, one elliptic and one 
positive 
hyperbolic (see also \cite{Bourgeois_MB} or \cite[\S5.3]{hutchings2016beyond}). 
\end{enumerate}

In more detail, we assume $\Omega^\fr$ is bounded by the axes and a smooth function $h: [0,a] \rightarrow [0,b]$ for some $a,b \in \R_{> 0}$ such that:
\begin{itemize}
\item $h$ is strictly decreasing and strictly concave down
\item $h(0) = b$ and $h(a) = 0$
\item $-\nu < h'(0) < 0$ and $h'(a) < -1/\nu$ for some $\nu > 0$ sufficiently small, and $h'(0),h'(a) \in \R \setminus \Q$.
\end{itemize}

The Reeb orbits after fully rounding are as follows.
For each $(i,j) \in \Z_{\geq 1}^2$ with $\nu < j/i < 1/\nu$,
there is an $S^1$-family of Reeb orbits lying in the two-torus $\mu^{-1}(p_{i,j}) \subset \bdy X_\Omega^\fr$, where $p_{i,j} \in \bdy \Omega^\fr$ is such that the outward normal to $\bdy \Omega^\fr$ at $p_{i,j}$ is parallel to $(i,j)$. 
The Reeb orbits in this family are $\gcd(i,j)$-fold covers of their underlying simple orbits.
In $\bdy \wt{X}_\Omega$, these $S^1$-families having action less than $K$ get replaced by a corresponding pair of nondegenerate elliptic and hyperbolic orbits, which we denote by $e_{i,j}$ and $h_{i,j}$ respectively.
There are also nondegenerate elliptic Reeb orbits of $\bdy X_\Omega^\fr$ which lie in $\mu^{-1}(a,0)$ and $\mu^{-1}(0,b)$. We denote these by $e_{i,0}$ and $e_{0,j}$ respectively for $j \in \Z_{\geq 1}$, and we use the same notation for their natural analogues in $\bdy \wt{X}_\Omega$.
We refer to the Reeb orbits of $\bdy \wt{X}_\Omega$ of the form $e_{i,j}$ or $h_{i,j}$ as above as {\bf acceptable}.
Note that each acceptable orbit has action less than $K$.

For the acceptable Reeb orbits in $\wt{X}_\Omega$ described above, we have 
\begin{align}\label{eq:CZ}
\cz(e_{i,j}) = 2i + 2j + 1, \qquad \cz(h_{i,j}) = 2i + 2j,
\end{align}
 where over each Reeb orbit $\ga$ we use by default the trivialization of the contact distribution that extends over a  disc in $\p X$ with boundary $\ga$.
 \footnote
 {
 This is the trivialization called $\tau_{\op{ex}}$ in \cite[S3.2]{McDuffSiegel_counting}.}
There are also three slightly different associated action filtrations. We denote by $||-||_{\Omega}^*$ the dual of the norm on $\R^2$ whose unit ball is $\Omega$. 
Viewing $e_{i,j}$ and $h_{i,j}$ as formal symbols, we put
\begin{itemize}
\item
$\calA_\Omega(e_{i,j}) = \calA_\Omega(h_{i,j}) = ||(i,j)||_{\Omega}^* = \max\limits_{\vec{v} \in \Omega} \langle \vec{v},(i,j)\rangle$, the ``idealized action''
\item
$\calA_\Omega^\fr(e_{i,j}) = \calA_\Omega^\fr(h_{i,j}) = ||(i,j)||_{\Omega^\fr}^* = \max\limits_{\vec{v} \in \Omega^\fr} \langle \vec{v},(i,j)\rangle$, the ``fully rounded action''
\item
$\tcalA_\Omega(e_{i,j})$ and $\tcalA_\Omega(h_{i,j})$ denote the actions of the corresponding Reeb orbits in the domain $\tX_\Omega$, the ``perturbed action''.
\end{itemize}
We will sometimes refer to any of these as simply ``the action'' if which one we are referring to is clear from the context or irrelevant, and we will often omit $\Omega$ from the notation if it is implicit.
Note that $\tcalA_\Omega$ is a small perturbation of $\calA_\Omega^\fr$, although its precise values are sensitive to the choices involved in constructing $\tX_\Omega$.

\sss

Let $w = \gamma_1 \times \cdots \times \gamma_k$ be an 
(unordered) tuple of acceptable Reeb orbits in $\bdy \tX_\Omega$.
We will refer to such a $w$ as a {\bf word}, and we often view it as simply a collection of formal symbols of the form $e_{i,j}$ or $h_{i,j}$. 
As a convenient shorthand we define the {\bf index} of $w$ to be the sum
\begin{align}\label{eq:indw}
\ind(w) := \sum_{i=1}^k \cz(\gamma_i) + k-2.
\end{align}
More generally, for any trivialization $\tau$, the Fredholm index of a curve $C$ with top ends on
the orbits  $\gamma_1,\dots,\gamma_k$ and negative ends on $\gamma_1',\dots,\gamma'_{k'}$ is given by
\begin{align}\label{eq:Frind}
\ind(u) = - \chi(C) + 2c_\tau(C) + \sum_{i=1}^k \cz_\tau(\gamma_i) -  \sum_{j=1}^{k'} \cz_\tau(\gamma'_j).
\end{align}
Note that the relative first Chern class term in \eqref{eq:Frind} vanishes if we use the trivialization $\tau_{\op{ex}}$,
so the formula in \eqref{eq:indw} is the contribution of the top end of a curve to its Fredholm index.  
In particular,
 $\ind(\gamma_1 \times \dots \times\gamma_k) = 2m$ is an even integer,  a 
 (rational) curve in $\tX_\Omega$ with top ends $\gamma_1,\dots,\gamma_k$ and satisfying the constraint $\lll \T^{(m)}p\rrr$ has Fredholm index zero. As we will see in \S\ref{sec:constructing_curves}, the strong permissibility condition introduced below ensures that every connected curve with strongly permissible top end is somewhere injective.

We note also that if $w$ is ``elliptic'', meaning that all of the constituent Reeb orbits are elliptic, then its half-index is given by
$$
\tfrac{1}{2}\ind(e_{i_1,j_1} \times \cdots \times e_{i_k,j_k}) = \sum_{s=1}^q (i_s + j_s) + k-1.
$$
We extend the definition of idealized action to words by putting
$$\calA(\gamma_1,\dots,\gamma_k) := \sum_{i=1}^k \calA(\gamma_i),$$
and similarly for the fully rounded action $\calA^\fr$ and perturbed action $\tcalA$.
We will say that a word $w$ is acceptable if each of its constituent orbits is.

\begin{lemma}\label{lem:fr_action_conds}
We can arrange the fully rounding procedure such that the following further conditions are satisfied:
\begin{enumerate}
\item[{\rm (a)}]
For each pair of acceptable orbits $e_{i,j},h_{i,j}$, we have 
$$0 < \tcalA(e_{i,j}) - \tcalA(h_{i,j}) < \tfrac{1}{2} | \tcalA(e_{i',j'}) - \tcalA(e_{i'',j''})|$$ for any pair of acceptable orbits $e_{i',j'},e_{i'',j''}$ with $(i',j') \neq (i'',j'')$.

\item[{\rm (b)}] 
Given any two acceptable words $w,w'$ such that $\calA(w) < \calA(w')$, we have also $\calA^\fr(w) < \calA^\fr(w')$.
\item[{\rm (c)}]  
Given any two acceptable words such that $\calA^\fr(w) < \calA^\fr(w')$, we have also $\tcalA(w) < \tcalA(w')$.
\item[{\rm (d)}]  For any two
 distinct 
 acceptable orbits $\gamma,\gamma'$ we have $\tcalA(\gamma) \neq \tcalA(\gamma')$, and moreover the set of $\tcalA$ values of acceptable orbits which are simple (i.e. have $\gcd(i,j)=1$) is linearly independent over $\Q$.
\end{enumerate}
\end{lemma}

\NI In the sequel, we will take  $K > 0$ (the upper bound of the energy of  acceptable  orbits) sufficiently large and 
 $\nu > 0$ (which measures the size of the perturbation) sufficiently small  that for action reasons the unacceptable Reeb orbits play essentially no role; thus without much harm we can pretend that the Reeb orbits of $\bdy \wt{X}_\Omega$ are precisely $e_{i,j}$ for any $(i,j) \in \Z_{\geq 0}^2$ with $i,j$ not both zero and $h_{i,j}$ for any $(i,j) \in \Z_{\geq 1}^2$.

\subsection{Strong and weak permissibility}\label{subsec:permis}

In this subsection we prove Lemma~\ref{lem:no_nonperm_curves}, which states that the positive orbits of a somewhere injective curve in a fully rounded convex toric domain must be strongly permissible in the sense of the following definition:

\begin{definition}
Consider a word $w = \gamma_1 \times \cdots \times \gamma_q$, where for each $s = 1,\dots,q$ $\gamma_i = e_{i_s,j_s}$ or $\gamma_i = h_{i_s,j_s}$ for some $i_s,j_s$. 
We say that $w$ is {\bf strongly permissible} if one of the following holds:
\begin{itemize}
\item
$w = e_{1,0}$ or $w = e_{0,1}$, or else

\item
$i_1,\dots,i_q$ are not all zero, and similarly $j_1,\dots,j_q$ are not all zero.
\end{itemize}
We say $w$ is {\bf weakly permissible} if it is either strongly permissible or it is of the form $e_{k,0}$ or $e_{0,k}$ for some $k \in \Z_{\geq 2}$.
\end{definition}

\begin{lemma}\label{lem:no_nonperm_curves}
Let $C$ be an asymptotically cylindrical $J$-holomorphic rational curve in $\wt{X}_\Omega$, where $\wt{X}_\Omega$ is a fully rounded four-dimensional convex toric domain and $J \in \Jadm(\wt{X}_\Omega)$. If $C$ is somewhere injective, then its word of positive orbits is strongly permissible. 
\end{lemma}

Before proving the lemma, we recall how to compute the terms in the relative adjunction formula in the case of a four-dimensional fully rounded convex toric domain. 
Following \cite[\S3.3]{Hlect}, the relative adjunction formula for a somewhere injective curve asymptotically cylindrical curve in a four-dimensional symplectic cobordism reads
$$
c_\tau(C)  = \chi(C) + Q_\tau(C) + w_\tau(C) - 2\delta(C).
$$
Here $\tau$ denotes a choice of trivialization over each Reeb orbit, $\chi(C)$ is the Euler characteristic of the curve $C$, and $\delta(C)$ is a count of singularities which is necessarily nonnegative.
The computation of the remaining terms for $\tX_\Omega$ with respect to a certain choice\footnote
{
This is different from the trivialization used before in which $c_\tau(e_{i,j})=c_\tau(h_{i,j})=0$.}
 of trivialization $\tauhut$, is described in \cite[\S 5.3]{hutchings2016beyond}, which we briefly summarize as follows.
Let $C$ be a curve in $\wt{X}_\Omega$, and let $\Gamma = (\gamma_1,\dots,\gamma_k)$ denote its positive asymptotic Reeb orbits.
\begin{itemize}

\item \textit{Relative self-intersection:} We have $Q_\tauhut(C) = Q_\tauhut(\Gamma) = 2\op{Area}(R),$ where:
\begin{itemize}
\item for each constituent orbit (including repeats) of $\Gamma$ of the form $e_{i,j}$ or $h_{i,j}$, we consider the corresponding ``edge vector'' $(j,-i)$
\item we reorder the collection of edge vectors and place them end-to-end so that they form a concave down path $\Lambda \subset \R^2_{\geq 0}$ from $(0,y(\Lambda))$ to $(x(\Lambda),0)$ for some $x(\Lambda),y(\Lambda) \in \Z_{\geq 0}$.
\item $R$ is the lattice polygon bounded by $\Lambda$ and the axes.
\end{itemize}
For example, we have $Q_\tauhut(h_{i,j}) = Q_\tauhut(e_{i,j}) = ij$.

\item
\textit{Relative first Chern class:} 
We have $c_\tauhut(C) = c_\tauhut(\Gamma) = \sum_{i=1}^k c_\tauhut(\gamma_i)$, where
$$c_\tauhut(h_{i,j}) = c_\tauhut(e_{i,j}) = i+j.$$

\item \textit{Asymptotic writhe:} $w_\tauhut(C)$ measures the total asymptotic writhe of $C$ around its asymptotic Reeb orbits. Although this is difficult to compute directly, we have the writhe bound (3.2.9) in \cite[\S3.2]{McDuffSiegel_counting} (see \cite[\S5.1]{hutchings2016beyond} for more details). This is formulated in terms of the {\bf monodromy angle} $\theta$ of each simple Reeb orbit.  In particular, since we can take this  to be $0$ for the hyperbolic orbits $h_{i,j}$ and positive but very small for the elliptic orbits $e_{i,j}$, the writhe bound implies that the top writhe of any   curve with positive ends on a word in $e_{i,j},\ h_{i,j}$ is always $\le 0$. 
\end{itemize}

\begin{proof}[Proof of Lemma~\ref{lem:no_nonperm_curves}]
Without loss of generality, consider a somewhere injective curve in $\wt{X}_\Omega$ with positive ends $(\gamma_1,\dots,\gamma_k)$, and suppose that for each $s =1,\dots,k$ we have $\gamma_s = e_{i_s,0}$ for some $i_s \in \Z_{\geq 1}$.
The writhe bound gives $w_\tauhut(C) \leq 0$.
Meanwhile, we have $c_\tauhut(C) = \sum_{s=1}^k i_s$ and $Q_\tauhut(C) = 0$, and hence 
\begin{align*}
w_\tauhut(C) = c_\tauhut(C) - \chi(C) - Q_\tauhut(C) + 2\delta(C) = \sum_{s=1}^k i_s - (2 - k) + 2\delta \leq 0,
\end{align*}
and consequently
$\sum_{s=1}^k (i_s + 1) \leq 2$,
which forces $k = i_1 = 1$.  A similar calculation rules out the possibility that $i_s=0$ for all $s$.
\end{proof}

\sss

Using Lemma~\ref{lem:no_nonperm_curves}, we prove the following lower bound on $\gt_k(\tX_\Omega)$, which will be further refined in the next subsection.

\begin{lemma}\label{lem:gt_geq_gtalg}
For any four-dimensional convex toric domain $X_\Omega$ we have $$\gt_k(X_\Omega) \geq  \min\limits_{\substack{\ind(w) \geq 2k \\ w \;\,\text{wk.p.}}} \calA(w),$$
where we minimize over all weakly permissible words $w$ satisfying $\ind(w) \geq 2k$. 
\end{lemma}
\begin{proof}
By $C^0$-continuity it suffices to prove the analogous lower bound after fully rounding, namely 
$\gt_k(\tX_\Omega) \geq  \min\limits_{\substack{\ind(w) \geq 2k \\ w \;\,\text{wk.p.}}} \tcalA(w)$.
Pick a generic $J \in \Jadm(\wt{X}_\Omega;D)$. By definition of $\gt_k(\wt{X}_\Omega)$, we can find a curve $C$ in $\wt{X}_\Omega$ satisfying the constraint $\lll \T^{(k)}p\rrr$ with $E(C) \leq \gt_k(\wt{X}_\Omega)$ (a priori we should also consider the case $a \geq 2$ as in Remark~\ref{rmk:gt_reform}, but it is easy to check that these do not affect the infimum). 
Let $w$ denote the word of positive orbits corresponding to $C$.
Note that the underlying simple curve $\ovl{C}$ is somewhere injective and has nonnegative index by genericity of $J$, and therefore its word $\ovl{w}$ of positive orbits is strongly permissible by Lemma~\ref{lem:no_nonperm_curves}. 
Then the word $w$ is also strongly permissible unless we have $\ovl{w} = e_{1,0}$ or $\ovl{w} = e_{0,1}$.
Moreover, we have $\ind(C) \geq \kappa \ind(\ovl{C}) \geq 0$ by Lemma~\ref{lem:C_index} below, where $\kappa$ is the covering index of $C$ over $\ovl{C}$, and hence we have
$\ind(w) \geq 2k$.
If $w$ is strongly permissible then it is also weakly permissible and we have $\gt_k(\tX_\Omega) \geq \tcalA(w) \geq \min\limits_{\substack{\ind(w) \geq 2k \\ w \;\,\text{wk.p.}}} \tcalA(w)$.

We can therefore assume $\ovl{w} = e_{1,0}$ or $\ovl{w} = e_{0,1}$, since otherwise 
the proof is already complete. 
Observe that since $C$ satisfies the constraint $\lll \T^{(k)}p\rrr$, we must have $k \leq \kappa$.
Then $\tcalA(w) =\kappa\tcalA(e_{1,0}) \geq \tcalA(e_{k,0})$ or $\tcalA(w) = \kappa\tcalA(e_{0,1}) \geq \tcalA(e_{0,k})$ respectively. Since $e_{k,0}$ and $e_{0,k}$ are weakly permissible with index $2k$, this again implies the desired result.
\end{proof}

\begin{definition}\label{def:wmin}
We will denote by $\wmin$ the weakly permissible word with minimal $\tcalA_\Omega$ value subject to $\ind(\wmin) = 2k$.
\end{definition}
Since distinct words have different actions by
condition (d) in Lemma~\ref{lem:fr_action_conds},
  $\wmin$  is unique for each $k$.

\subsection{Minimal words}\label{subsec:min_words}

As before, let $X_\Omega$ be a four-dimensional convex toric domain with full rounding $\wt{X}_\Omega$. In light of Lemma~\ref{lem:gt_geq_gtalg}, we seek to understand which weakly permissible words have minimal $\tcalA_\Omega$ value. 
We begin with some preliminary lemmas.
In the following, put $a := \max\{x\;|\; (x,0) \in \Omega^\fr\}$ and $b := \max\{y\;|\; (0,y) \in \Omega^\fr\}$ as in \S\ref{subsec:fr}.

\begin{lemma}\label{lem:action_exceeds_ellipsoid}
For any $(i,j) \in \Z_{\geq 1}^2$, we have $ \max(ia,jb) < ||(i,j)||_{\Omega^\fr}^* < ai + jb$.
\end{lemma}
\begin{proof}
Let $\vec{v} = (v_1,v_2) \in \bdy \Omega^\fr \cap \R_{> 0}^2$ be such that $||(i,j)||_{\Omega^\fr}^* = \langle \vec{v},(i,j)\rangle$.
Then the line in $\R^2$ passing through $\vec{v}$ and orthogonal to $(i,j)$ is tangent to $\p \Om^\fr$, and is given by
$$\{(x,y) \in \R^2 \;|\; \langle (x,y),(i,j)\rangle = \langle \vec{v},(i,j)\rangle = iv_1 + jv_2 < ai + bj.
$$
This gives the upper bound.  To derive the lower bound, notice that
 the $y$ intercept is given by $\tfrac{iv_1+jv_2}{j}$, and this is strictly greater than $b$ since $h:[0,a] \rightarrow [0,b]$ is strictly concave down.
That is, we have $||(i,j)||_{\Omega^\fr}^* = iv_1 + jv_2 > jb$.
Similarly, the $x$ intercept is given by $\tfrac{iv_1+jv_2}{i}$, and by strict convexity this is strictly greater than $a$, i.e. we have
$||(i,j)||_{\Omega^\fr}^* = iv_1 + jv_2 > i a$.
\end{proof}

\begin{lemma}\label{lem:smaller_tuples_smaller_norm}
Given distinct pairs $(i,j),(i',j') \in \Z_{\geq 0}^2$ with $i' \leq i$ and $j' \leq j$, we have $||(i',j')||_{\Omega^\fr}^* < ||(i,j)||_{\Omega^\fr}^*$.
\end{lemma}

\begin{proof}
Without loss of generality we can assume $(i',j') = (i-1,j)$, since the case $(i',j') = (i,j-1)$ is completely analogous and then the general case follows by induction.
Let $\vec{v} = (v_1,v_2) \in \bdy\Omega^\fr$ be such that $||(i-1,j)||_{\Omega^\fr}^* = \langle \vec{v},(i-1,j)\rangle$. Then we have
$$||(i-1,j)||_{\Omega^\fr}^* = (i-1)v_1 + jv_2 \leq  iv_1+jv_2 \leq ||(i,j)||_{\Omega^\fr}^*,$$
and the inequality is strict unless $v_1 = 0$, which is only possible if $(i-1,j)$ lies on the $y$-axis, i.e. $i = 1$. 
In this case, by Lemma~\ref{lem:action_exceeds_ellipsoid} we have 
\begin{align*}
||(i,j)||_{\Omega^\fr}^* = ||(1,j)||_{\Omega^\fr}^* > \max(a,jb) \geq jb = ||(0,j)||_{\Omega^\fr}^* = ||(i-1,j)||_{\Omega^\fr}^*,
\end{align*}
as desired.
\end{proof}

We next show that we can effectively ignore the hyperbolic orbits. Recall that a word $w = \gamma_1 \times \cdots \times \gamma_k$ is called ``elliptic'' if each constituent orbit $\gamma_i$ is elliptic.
\begin{lemma}\label{lem:make_elliptic}
Given any word $w$ which is not elliptic, we can find an elliptic word $w'$ with $\ind(w') \geq \ind(w)-1$ and $\tcalA(w') < \tcalA(w)$. 
Moreover, if $w$ is strongly (resp. weakly) permissible, then we can arrange that the same is true for $w'$.
\end{lemma}
\begin{proof}

Firstly, if $\tfrac{1}{2}\ind(w)$ is not an integer, then we replace some hyperbolic orbit $h_{i,j}$ by $e_{i',j'}$ with $(i',j') = (i-1,j)$ or $(i',j') = (i,j-1)$. 
Note that this replacement decreases the index by $1$.
Moreover, we have $||(i',j')||_{\Omega^\fr}^* < ||(i,j)||_{\Omega^\fr}^*$ by 
Lemma~\ref{lem:smaller_tuples_smaller_norm}, and hence 
 $\tcalA(e_{i',j'}) < \tcalA(e_{i,j})$ by Lemma~\ref{lem:fr_action_conds}(c).
Then by Lemma~\ref{lem:fr_action_conds}(a) we also have 
$$\tcalA(e_{i,j}) - \tcalA(h_{i,j}) < \tcalA(e_{i,j}) - \tcalA(e_{i',j'}),$$ and hence
$\tcalA(e_{i',j'}) < \tcalA(h_{i,j})$, so this shows that the above replacement strictly decreases $\tcalA$.

Now suppose there are $2\ell$ hyperbolic orbits in $w$ for some $\ell \in \Z_{\geq 0}$. 
For $\ell$ of these replace $h_{i,j}$ with $e_{i,j}$, and for the other $\ell$ replace $h_{i,j}$ with $e_{i-1,j}$ or $e_{i,j-1}$.
Each pair of such  replacements
strictly deceases $\tcalA$ by the same lemma,
and the total index is unchanged.
For example, we have $\tcalA(h_{i,j} \times h_{i',j'}) > \tcalA(e_{i,j}\times e_{i'-1,j'})$ using
\begin{align*}
\tcalA(e_{i',j'}) - \tcalA(h_{i',j'}) < \tfrac{1}{2}(\tcalA(e_{i',j'}) - \tcalA(e_{i'-1,j'}))\\
\tcalA(e_{i,j}) - \tcalA(h_{i,j})) < \tfrac{1}{2}(\tcalA(e_{i',j'}) - \tcalA(e_{i'-1,j'})).
\end{align*}

Lastly, it is straightforward to check that each of these replacements can be done so as to preserve strong or weak permissibility.
\end{proof}
\begin{rmk}\label{rmk:ell_parity}
For future reference, note that in Lemma~\ref{lem:make_elliptic} if $\ind(w) = 2k$ for some $k \in \Z_{\geq 1}$ then we must also have $\ind(w') = 2k$ since the index of any elliptic word is even.
If particular, if $\ind(w) \geq 2k$ for some $\Z_{\geq 1}$ then we have $\ind(w') \geq 2k$ as well.
\end{rmk}

\sss

The following lemma will be our most useful tool for iteratively reducing the action of a word:

\begin{lemma}\label{lem:remove_0,1}
Assume $a > b$. Then we have $\tcalA(e_{0,1} \times e_{i,j}) < \tcalA(e_{i+1,j+1})$ for any $(i,j) \in \Z_{\geq 0}^2 \setminus \{(0,0)\}$.
\end{lemma}
\begin{proof}
Let $\vec{v}  = (v_1,v_2) \in \bdy\Omega^\fr$ be such that $||(i,j)||_{\Omega^\fr}^* = \langle \vec{v},(i,j)\rangle$.
Suppose first that have $v_1,v_2 \geq 1$
Note that $(v_1,v_2)$ lies above or on the line joining $(a,0)$ and $(0,b)$, i.e. we have
 $av_2 + bv_1 \geq ab$.
 Since $a>b$,  we have $v_1 + v_2 \geq b$, with equality only if $v_1 = 0$.
We then have
\begin{align*}
\calA^\fr(e_{0,1} \times e_{i,j}) = ||(0,1)||_{\Omega^\fr}^* + ||(i,j)||_{\Omega^\fr}^* &= b + iv_1 + jv_2 \\&\leq v_1+v_2 + iv_1 + jv_2 \\&= \langle \vec{v},(i+1,j+1)\rangle\\ &\leq ||(i+1,j+1)||_{\Omega^\fr}^* \\&= \calA^\fr(e_{i+1,j+1}),
\end{align*}
where the first inequality is strict unless $v$ lies on the $y$-axis, in which case we must have $i = 0$.
If $i = 0$, by Lemma~\ref{lem:action_exceeds_ellipsoid} we have
$$
\calA^\fr(e_{1,j+1}) > \max(a,(j+1)b) \geq (j+1)b = \calA^\fr(e_{0,1} \times e_{0,j}).
$$
Thus in any case we have $\calA^\fr(e_{0,1} \times e_{i,j}) < \calA^\fr(e_{i+1,j+1})$,
and by Lemma~\ref{lem:fr_action_conds}(c) we also have $\tcalA(e_{0,1} \times e_{i,j}) < \tcalA(e_{i+1,j+1})$.
\end{proof}
\begin{rmk}
Note that the assumption $a > b$ is not very restrictive, since if $a < b$ we can simply replace $\Omega^\fr$ by its reflection about the diagonal.
\end{rmk}

\sss

Using the above tools, we first consider ways to reduce action without any regard to permissibility:
\begin{lemma}\label{lem:nonperm_reduc}
Assume $a > b$. Given any elliptic word $w$, there is another elliptic word $w'$ with $\ind(w') = \ind(w)$ and $\tcalA(w') \leq \tcalA(w)$, where $w$ takes one of the following forms:
\begin{enumerate}[label=(\alph*)]
\item $e_{0,1}^{\times i}$ for $i \geq 1$
\item $e_{0,1}^{\times i}\times e_{1,1}^{\times j}$ for $i \geq 0$ and $j \geq 1$
\item $e_{0,1}^{\times i} \times e_{0,2}$ for $i \geq 0$. 
\end{enumerate}
Moreover, we have $\tcalA(w') < \tcalA(w)$ unless $w$ and $w'$ differ by a reordering.
\end{lemma}
\begin{proof}
We first iteratively apply Lemma~\ref{lem:remove_0,1} as many times as possible, replacing $e_{i+1,j+1}$ with $e_{0,1} \times e_{i,j}$ if $(i,j) \in \Z_{\geq 0}^2 \setminus \{(0,0\}$.
Note that the resulting word contains only orbits of the forms $e_{1,1},e_{k,0},e_{0,k}$ for $k \geq 1$, and each replacement strictly decreases $\tcalA$.

Next, we replace each $e_{k,0}$ for $k \geq 1$ with $e_{0,k}$. Similarly, we replace each $e_{0,2k-1}$ such that $2k-1 \geq 3$ with $e_{0,1}^{\times k}$, and we replace each $e_{0,2k}$ such that $2k \geq 4$ with $e_{0,1}^{\times (k-1)} \times e_{0,2}$. We also replace each $e_{0,2}\times e_{0,2}$ with $e_{0,1}\times e_{0,1} \times e_{0,1}$.
Each of these replacements strictly decreases $\tcalA$.

The resulting word is of the form $e_{0,1}^{\times i}\times e_{1,1}^{\times j} \times e_{0,2}^{\times k}$ for some $i,j \in \Z_{\geq 0}$ and $k \in \{0,1\}$.
By Lemma~\ref{lem:fr_action_conds}(d) we have either $\tcalA(e_{1,1}) < \tcalA(e_{0,2})$ or $\tcalA(e_{1,1}) > \tcalA(e_{0,2})$.
In the former case, we also replace any remaining $e_{0,2}$ with $e_{1,1}$. 
In the latter case, we replace each $e_{1,1}$ with $e_{0,2}$, and then further replace each $e_{0,2} \times e_{0,2}$ with $e_{0,1} \times e_{0,1} \times e_{0,1}$ as above.

The resulting word $w'$ satisfies $\ind(w') = \ind(w)$ and $\tcalA(w') \leq \tcalA(w)$ and takes one of the forms (1),(2),(3). Moreover, up to reordering these are the only cases when none of the above reductions are applicable, and otherwise we have $\tcalA(w') < \tcalA(w)$.
\end{proof}

Next, we investigate reductions in actions which preserve the strong permissibility condition. Perhaps surprisingly, there are only a few possibilities 
for minimal words,
 regardless of $\Omega$:

\begin{prop}\label{prop:strong_perm}
Assume $a > b$. Given any strongly permissible elliptic word $w$ with $\tfrac{1}{2}\ind(w) > 1$, there is another strongly permissible elliptic word $w'$ with $\ind(w') = \ind(w)$ and $\tcalA(w') \leq \tcalA(w)$, where $w'$
takes one of the following forms:
\begin{enumerate}
\item[{\rm (1)}] $e_{0,1}^{\times i} \times e_{1,1}^{\times j}$ for $i \geq 0$ and $j \geq 1$
\item[{\rm (2)}] $e_{0,1}^{\times i} \times e_{1,s}$ for $i \geq 0$ and $s \geq 2$
\item[{\rm (3)}] $e_{0,1}^{\times i} \times e_{1,0}$ for $i \geq 1$.
\end{enumerate}
Moreover, we have $\tcalA(w') < \tcalA(w)$ unless $w$ and $w'$ differ by a reordering. 
\end{prop}

\begin{proof}
To start, we iteratively apply Lemma~\ref{lem:remove_0,1} as many times as possibly without spoiling strong permissibility, and let $w = e_{i_1,j_1} \times \cdots \times e_{i_q,j_q}$ denote the resulting word.
Note that we must have $i_s \leq 1$ or $j_s \leq 1$ for each $s = 1,\dots,q$, since otherwise a further application of Lemma~\ref{lem:remove_0,1} would be possible.
Furthermore, we can assume without loss of generality that we have $i_1 \neq 0$ (note that $e_{0,1}$ is ruled out using $\tfrac{1}{2}\ind(w) > 1$).

Next, by applying Lemma~\ref{lem:nonperm_reduc} to the subword 
$e_{i_2,j_2}\times\cdots \times e_{i_q,j_q}$, we obtain a word of the form $e_{i_1,j_1} \times w''$, where $w''$ is a word having one of the forms $(a),(b),(c)$.
This replacement leaves the index unchanged and strictly decreases $\tcalA$ (unless it is vacuous). Moreover, by our assumption $i_1 \neq 0$ and inspection of the forms (a),(b),(c), the word $e_{i_1,j_1} \times w''$ is strongly permissible.
We also have $(i_1,j_1) \notin \Z_{\geq 2} \times \Z_{\geq 2}$, and hence $(i_1,j_1)$ must be one of the following:
 \begin{align}\label{eqn:i1j1_forms}
(1,0),(1,1),(1,s),(s,0),(s,1),\;\;\;\;\; s \geq 2.
\end{align}

Our goal is to apply further reductions to $e_{i_1,j_1} \times w''$ which decrease $\tcalA$ and leave the index unchanged, in order to arrive at one of the forms (1),(2),(3).
Observe that for $s \geq 2$ we have
\begin{align*}
||(s-1,1)||_{\Omega^\fr}^* \leq ||(s-1,0)||_{\Omega^\fr}^* + ||(0,1)||_{\Omega^\fr}^* = (s-1)a + b < sa  = ||(s,0)||_{\Omega^\fr}^* 
\end{align*}
and hence $\tcalA(e_{s-1,1}) < \tcalA(e_{s,0})$ by Lemma~\ref{lem:fr_action_conds}(c).
Also, by Lemma~\ref{lem:remove_0,1} we have $\tcalA(e_{0,1} \times e_{s-1,0}) < \tcalA(e_{s,1})$ for $s \geq 2$.
This shows that we can ignore the last two items in \eqref{eqn:i1j1_forms}.

We now consider each of the remaining possibilities for $e_{i_1,j_1} \times w''$ and explain the necessary reductions:
\begin{itemize}
\item $(i_1,j_1) = (1,0)$: 
\begin{enumerate}[label=(\alph*)]
\item $e_{1,0} \times e_{0,1}^{\times i}$ for $i \geq 1$: already of form (3)
\item $e_{1,0} \times e_{0,1}^{\times i} \times e_{1,1}^{\times j}$ for $i \geq 0$ and $j \geq 1$: replace $e_{1,0}$ with $e_{0,1}$, becomes of form (1)
\item $e_{1,0} \times e_{0,1}^{\times i} \times e_{0,2}$ for $i \geq 0$: replace $e_{1,0} \times e_{0,2}$ with $e_{0,1} \times e_{1,1}$ by (i), becomes of form (1)
\end{enumerate}

\item $(i_1,j_1) = (1,1)$:
\begin{enumerate}[label=(\alph*)]
\item $e_{1,1} \times e_{0,1}^{\times i}$ for $i \geq 1$: already of form (1)
\item $e_{1,1} \times e_{0,1}^{\times i} \times e_{1,1}^{\times j}$ for $i \geq 0$ and $j \geq 1$: already of form (1)
\item $e_{1,1} \times e_{0,1}^{\times i} \times e_{0,2}$ for $i \geq 0$: replace $e_{1,1} \times e_{0,2}$ with $e_{1,0} \times e_{0,1} \times e_{0,1}$ by (ii) below, becomes of form (3)
\end{enumerate}

\item $(i_1,j_1) = (1,s)$ for $s \geq 2$:
\begin{enumerate}[label=(\alph*)]
\item $e_{1,s} \times e_{0,1}^{\times i}$ for $i \geq 1$: already of form (2)
\item $e_{1,s} \times e_{0,1}^{\times i} \times e_{1,1}^{\times j}$ for $i \geq 0$ and $j \geq 1$: replace $e_{1,s}$ with $e_{0,1}^{\times(k+1)}$ if $s = 2k$ is even, or with $e_{1,1}\times e_{0,1}^{\times k}$ if $s = 2k+1$ is odd by (iii), 
becomes of form (1)
\item $e_{1,s} \times e_{0,1}^{\times i} \times e_{0,2}$ for $i \geq 0$: replace $e_{0,2} \times e_{1,s}$ with $e_{0,1}\times e_{1,s+1}$ by (iv), becomes of form (2)
\end{enumerate}
\end{itemize}

We justify the above replacements by applying Lemma~\ref{lem:action_exceeds_ellipsoid} as follows:
\begin{itemize}
\item[{\rm (i)}] We have $\tcalA(e_{1,1}\times e_{0,1}) < \tcalA(e_{1,0} \times e_{0,2})$ since 
\begin{align*}
||(1,1)||_{\Omega^\fr}^* + ||(0,1)||_{\Omega^\fr}^* &< ||(1,0)||_{\Omega^\fr}^* + 2||(0,1)||_{\Omega^\fr}^* \\ &= a + 2b \\ &= ||(1,0)||_{\Omega^\fr}^* + ||(0,2)||_{\Omega^\fr}^*
\end{align*}
\item[{\rm (ii)}]  We have $\tcalA(e_{1,0} \times e_{0,1} \times e_{0,1}) < \tcalA(e_{1,1} \times e_{0,2})$ since 
\begin{align*}
||(1,0)||_{\Omega^\fr}^* + 2||(0,1)||_{\Omega^\fr}^*&= a + 2b \\&< ||(1,1)||_{\Omega^\fr}^* + 2b \\&= ||(1,1)||_{\Omega^\fr}^* + ||(0,2)||_{\Omega^\fr}^*.
\end{align*}

\item[{\rm (iii)}] We have $\tcalA(e_{0,1}^{\times(k+1)})  < \tcalA(e_{1,2k})$ since 
\begin{align*}
(k+1)||(0,1)||_{\Omega^\fr}^* = (k+1)b \leq 2kb < ||(1,2k)||_{\Omega^\fr}^*
\end{align*}
and $\tcalA(e_{1,1}\times e_{0,1}^{\times k}) < \tcalA(e_{1,2k+1)}$ since
\begin{align*}
||(1,1)||_{\Omega^\fr}^* + k||(0,1)||_{\Omega^\fr}^* \leq (k+2)b \leq (2k+1)b < ||(1,2k+1)||_{\Omega^\fr}^*.
\end{align*}

\item[{\rm (iv)}] 
We have $\tcalA(e_{0,1}\times e_{1,s+1}) < \tcalA(e_{0,2} \times e_{1,s})$ since 
\begin{align*}
||(0,1)||_{\Omega^\fr}^* + ||(1,s+1)||_{\Omega^\fr}^* < 2b + ||(1,s)||_{\Omega^\fr}^* = ||(0,2)||_{\Omega^\fr}^* + ||(1,s)||_{\Omega^\fr}^*.
\end{align*}
\end{itemize}
This completes the proof.  \end{proof}

\begin{cor}\label{cor:weak_perm}
Given any weakly permissible elliptic word $w$, there is another weakly permissible elliptic word $w'$ with $\ind(w') = \ind(w)$ and $\tcalA(w') \leq \tcalA(w)$, where $w'$
takes one of the following forms:
\begin{enumerate}
\item[{\rm (1)}] $e_{0,1}^{\times i} \times e_{1,1}^{\times j}$ for $i \geq 0$ and $j \geq 1$
\item[{\rm (2)}] $e_{0,1}^{\times i} \times e_{1,s}$ for $i \geq 0$ and $s \geq 2$
\item[{\rm (3)}] $e_{0,1}^{\times i} \times e_{1,0}$ for $i \geq 1$,
\item[{\rm (4)}] $e_{0,s}$ for $s \geq 1$
\end{enumerate}
Moreover, we have $\tcalA(w') < \tcalA(w)$ unless $w$ and $w'$ differ by a reordering. 
\end{cor}

We next refine Lemma~\ref{lem:gt_geq_gtalg} so that the minimization involves only words which are elliptic and satisfy $\ind(w) = 2k$ (rather than $\ind(w) \geq 2k$).
This completes the proof of half of Theorem~\ref{thm:main_comp}.

\begin{cor}\label{cor:gt_refined_lb}
For any four-dimensional convex toric domain $X_\Omega$ we have $$\gt_k(X_\Omega) \geq  \min\limits_{\substack{\ind(w) = 2k \\ w \;\mathrm{wk.p.}}} \calA(w),$$
where we minimize over all weakly permissible elliptic words $w$ satisfying $\ind(w) = 2k$. 
\end{cor}
\begin{proof}
The restriction to elliptic words follows from Remark~\ref{rmk:ell_parity}.
Now it suffices to show that given any weakly permissible elliptic word $w$ with $\tfrac{1}{2}\ind(w) > 1$, there is another weakly permissible elliptic word $w'$ with $\tfrac{1}{2}\ind(w') = \tfrac{1}{2}\ind(w) - 1$ and $\tcalA(w') < \tcalA(w)$.
After applying Corollary~\ref{cor:weak_perm}, we can assume that $w$ has one of the forms (1),(2),(3),(4), and we then make the following respective replacements:
\begin{enumerate}
\item $e_{0,1}^{\times i} \times e_{1,1}^{\times (j-1)} \times e_{1,0}$
\item $e_{0,1}^{\times i} \times e_{1,s-1}$
\item $e_{0,1}^{\times (i-1)} \times e_{1,1}$
\item $e_{0,s-1}$.\qedhere
\end{enumerate}
\end{proof}

\sss 

We end this section by proving Corollary~\ref{cor:latt_form}, which claims that (in dimension four) $\gt_k(X_{\Omega})$ is the minimal length $\ell_\Om( \p P)$ of the boundary $\p P$ of a convex lattice polygon $P$ such that $\p P$ contains exactly $k+1$ lattice points.
For the moment we assume Theorem~\ref{thm:main_comp}, the proof of which is completed in \S\ref{sec:constructing_curves} below.

\begin{proof}[Proof of Corollary~\ref{cor:latt_form}]
We first prove that the right hand side of \eqref{eq:latt_form} is less than or equal to the right hand side of \eqref{eq:gt_orig}; in other words,
 for each minimal word $w$  there is a lattice polygon $P$ with
$\ell_\Om( \p P)$ less than or equal to $\calA(w)$. To this end,
let $(i_1,j_1),\dots,(i_q,j_q) \in \Z_{\geq 0}^2 \setminus \{(0,0)\}$ be a minimizer, which we can assume takes one of the forms (1),(2),(3),(4) given in Corollary~\ref{cor:weak_perm}.
Then we have $$
\sum_{s=1}^q (i_s + j_s) + q - 1 = k \;\;\mbox{  and }\;\; \gt_k(X_\Omega) = \sum_{s=1}^q ||(i_s,j_s)||_{\Omega}^*,
$$
 and we need to find a convex lattice polygon $P$ with $\ell_{\Omega}(\bdy P) \leq \sum_{s=1}^q ||(i_s,j_s)||_{\Omega}^*$ such that $|\bdy P \cap \Z^2| = k+1$.
In case (4), we take $P$ to be the degenerate polygon given by the convex hull of $(0,0),(0,k)$, which contains $k+1$ lattice points and satisfies 
$\ell_{\Omega}(\bdy P) = ||(0,k)||_{\Omega}^*.$
In cases (1)-(3), let $p_1,\dots,p_{q+1} \in \Z_{\leq 0}^2$ be the unique ordered list of lattice points such that
\begin{enumerate}
\item 
the displacement vectors 
$p_2-p_1,\dots,p_{q+1}-p_{q}$ equal $(i_1,j_1),\dots,(i_q,j_q)$ up to order,
\item $p_1 = (0,-\beta)$ and $p_{q+1} = (\alpha,0)$ for $\alpha = \sum_{s=1}^q i_s$ and $\beta = \sum_{s=1}^qj_s$,
\item the lower boundary $G$ of the convex hull of
$p_1,\dots,p_{q+1}$  is the graph of a convex piecewise linear function 
$[0,\alpha] \rightarrow [0,-\beta]$.
\end{enumerate}
Let $P \subset \R_{\leq 0}^2$ be the convex lattice polygon given by the convex hull of 
$(0,0),p_1,\dots,p_{q+1}$,
i.e. 
$P$ is the union of $G$ with the line segments joining $(\alpha,0)$ to $(0,0)$ and $(0,0)$ to $(0,-\beta)$.
Using the definition of $||-||_{\Omega}^*$ and the fact that $X_{\Omega}$ is a convex toric domain, observe that for any $(v_x,v_y) \in \R^2$ we have
\begin{align}\label{eq:Q3_disp}
||(v_x,v_y)||_{\Omega}^* = ||(\max(v_x,0),\max(v_y,0)||_{\Omega}^*.
\end{align}
In particular, we have $||(v_x,v_y)||_{\Omega}^* = 0$ if $(v_x,v_y) \in \R_{\leq 0}^2$,
and hence $$\ell_{\Omega}(\bdy P) = \ell_{\Omega}(G) = \sum_{s=1}^q ||(i_s,j_s)||_{\Omega}^*.$$
Moreover, since $\gcd(i_s,j_s) = 1$ for $s = 1,\dots, q$, the number of lattice points along $G$ is $q+1$, and hence we have
\begin{align*}
|\bdy P \cap \Z^2| = q+1 + \alpha + \beta - 1 = q + \sum_{s=1}^q (i_s + j_s) = k+1.
\end{align*}

Now we prove that the reverse inequality.
Let $P$ be a convex lattice polygon which is a minimizer for the right hand side of \eqref{eq:latt_form}, that is, it minimizes $\ell_\Om(\p P)$. 
We will assume that $P$ is nondegenerate, the degenerate case being a straightforward extension.
Let $A$ (resp. $B$) denote the minimal (resp. maximal) $x$ coordinate of any point in $P$, and similarly let $C$ (resp. $D$) denote the minimal (resp. maximal) $y$ coordinate of any point in $P$.
Let $P'$ denote the convex lattice polygon given by the convex hull of $P$ with the additional points $(A,D),(B,D),(A,C)$. 
Note that we have $P \subset P'$, and moreover $|P \cap \Z^2| \leq |P' \cap \Z^2|$.
Let $p_1,\dots,p_{q+1} \in \Z^2$ denote the lattice points encountered as we traverse $\bdy P'$ in the counterclockwise direction from $(A,C)$ to $(B,D)$.
For $s = 1,\dots,q$, let $(i_s,j_s) := p_{s+1} - p_s$ denote the corresponding displacement vectors.
Then we have
$$k+1 = |\bdy P \cap \Z^2| \leq |\bdy P' \cap \Z^2| =   \sum_{s=1}^q (i_s+j_s) + q.$$
Moreover, using \eqref{eq:Q3_disp} we have $\ell_\Omega (\bdy P) = \ell_{\Omega}(\bdy P')$.
Therefore the right hand side of \eqref{eq:gt_orig} is less than or equal to
$\sum_{s=1}^q ||(i_s,j_s)||_{\Omega}^* = \ell_{\Omega}(\bdy P') = \ell_{\Omega}(\bdy P)$.
\end{proof}

\section{Constructing curves in four-dimensional convex toric domains}\label{sec:constructing_curves}

In this section we complement Corollary~\ref{cor:gt_refined_lb} by proving a corresponding upper bound for $\gt_k(X_\Omega)$, thereby completing the proof of Theorem~\ref{thm:main_comp}.
In \S\ref{subsec:inv_min_words} we prove that the formal curve component $C$ in $\tX_\Omega$ with local tangency constraint $\lll \T^{(k)}p\rrr$ and positive asymptotics the minimal word
$\wmin$ of index $2k$ is formally perturbation invariant with respect to some generic 
$J_{\bdy \tX_\Omega} \in \Jadm(\bdy \tX_\Omega)$. 
After establishing this, 
we then  show that the moduli space $\calM^J_{\tX_\Omega}(\wmin)\lll \T^{(k)}p\rrr$ is in fact nonempty for some (and hence any) $J \in \Jadm^{J_{\bdy\wt{X}_\Omega}}(\tX_\Omega)$,
thereby achieving our desired upper bound.

More precisely, we show in Proposition~\ref{prop:curves_exist} that (except for the case $\wmin = e_{0,k}$ with $k \geq 2$) 
there is $J \in \Jadm^{J_{\bdy\wt{X}_\Omega}}(\tX_\Omega)$ such that
$\calM^J_{\tX_\Omega}(\wmin)$ is regular with nonzero signed count.
By Proposition~\ref{prop:count_J_indep}, this implies that $\calM^J_{\tX_\Omega}(\wmin) \neq \nil$ for any $J \in \Jadm^{J_{\bdy\wt{X}_\Omega}}(\tX_\Omega)$ (recall that the empty moduli space is automatically regular).
Since we  will then have verified all the hypotheses of Proposition~\ref{prop:stab_ub}, this also proves the stabilization property for four-dimensional convex toric domains.

To prove that suitable curves exist we argue as follows.  
By Proposition~\ref{prop:T=E} we have $$
\#\calM_{\tX_\Omega}(\wmin)\lll \T^{(k)}p\rrr = \#\calM_{\tX_\Omega}(\wmin)\lll (k)\rrr_E,
$$
 i.e. we can swap the local tangency constraint with a skinny ellipsoidal constraint.
In \S\ref{subsec:aut_trans}, we show that every curve in the latter moduli space counts positively, so it suffices to show that it is nonempty.
In \S\ref{subsec:obg} we give a biased summary of Hutchings--Taubes' obstruction bundle gluing, adapted to the case of cobordisms, and in \S\ref{subsec:induction} we explain how to apply obstruction bundle gluing in order to piece together  the curves we need  inductively from certain basic curves with very simple top ends.
Finally, in \S\ref{subsec:base_cases} we use the cobordism map in ECH to establish the base cases for our induction.

\subsection{Invariance of minimal word counts}\label{subsec:inv_min_words}

Our main goal in this subsection is to prove Proposition~\ref{prop:min_pert_inv}, which establishes formal perturbation invariance for those moduli spaces corresponding to weakly permissible words $\wmin$ of minimal action;
see Definition~\ref{def:wmin}. 
At first glance it seems plausible that we can rule out degenerations using minimality, but some care is needed due to the possibility of multiply covered curves of negative index. 
Recall that we are considering degenerations that might occur for a generic path $J_t, 0\le t\le 1,$  in 
$ \Jadm^{J_{\bdy\wt{X}_\Omega}}(\tX_\Omega)$.  Thus the (fixed) almost complex structure $J_{\bdy\wt{X}_\Omega}$ on the symplectization levels can be assumed to be generic.   
If a curve with top $\wmin$ does degenerate, the resulting building has a main component $C_0$ in $\wt{X}_\Omega$ that satisfies the tangency constraint, as well as some other components that may be assembled into 
representatives of a union of connected  formal buildings each of which has one negative end that attaches to $C_0$.  We first consider the properties of such a  formal building.    For definitions of the language used here, see Definitions~\ref{def:curvecomp}, ~\ref{def:formcurve} and~\ref{def:formbuild}.

\begin{lemma}\label{lem:symp_configs}
Let $C$ be a connected formal building with main level in $\wt{X}_\Omega$ and some number of symplectization levels in $\R \times \bdy \wt{X}_\Omega$, except that exactly one negative end of some curve component is not paired with any positive end of a curve component in a lower level.
Assume that each component of $C$ in a symplectization level is a (possibly trivial) formal cover of some formal curve component $\ovl{C}$ which is either trivial or else satisfies $\ind(\ovl{C}) \geq 1$.
Then we have $\ind(C) \geq 0$, with equality if and only if every component of $C$ is trivial.
\end{lemma}
\NI Note that Lemma~\ref{lem:symp_configs} does not involve any local tangency constraints.
\begin{proof}
Let $C_1,\dots,C_q$ denote the components of $C$ which have at least one negative end, and let $b_1,\dots,b_q$ denote the corresponding numbers of negative ends. 
Observe that since $C$ has genus zero, there must have at least $\sum_{i=1}^q (b_i-1)$ components without any negative ends, and by \eqref{eq:CZ} and \eqref{eq:Frind} each of these has index at least $3$. 
Therefore we have 
\begin{align*}
\ind(C) \geq \sum_{i=1}^q\left( \ind(C_i) + 3(b_i-1)\right).
\end{align*}
We will show that for $i = 1,\dots,q$ we have $\ind(C_i) + 3(b_i-1) \geq 0$, with equality if and only if $C_i$ is a trivial cylinder, from which the result immediately follows.

Let $D$ denote one of the components\footnote
{
Note that in this paper  each component lies in a single level; it is not a \lq\lq matched component'' in the sense of \cite{McDuffSiegel_counting}.}
$C_1,\dots,C_q$. Let $a$ and $b$ denote the respective numbers of positive and negative ends $D$, and let $e^+ \leq a$ and $e^- \leq b$ denote the numbers of positive and negative ends which are elliptic. 
We assume that $D$ is a $\kappa$-fold cover of $\ovl{D}$ for some $\kappa \in \Z_{\geq 1}$, where by assumption $\ovl{D}$ is either trivial or satisfies $\ind(\ovl{D}) \geq 1$. 
We denote by $\ovl{a},\ovl{b},\ovl{e}^+,\ovl{e}^-$ the analogs of the above for $\ovl{D}$. 

For each puncture or point in the domain of $D$, let us define its {\bf excess branching} to be one less than its ramification order as a cover of $\ovl{D}$.\footnote
{
That is, if $u$ is locally given by $z\mapsto z^k$ then the point in the domain corresponding to the origin has excess branching $k-1$.}
Let $E^{\pm}$ be the total excess branching at all positive (resp. negative) elliptic ends of $D$, and similarly let $H^{\pm}$ be the total excess branching at all positive (resp. negative) hyperbolic ends of $D$.
By elementary Riemann--Hurwitz considerations we have the following:
\begin{itemize}
\item
$a = \kappa\ovl{a} - E^+ - H^+$ and $b = \kappa\ovl{b} - E^- - H^-$
\item
$e^+ = \kappa\ovl{e}^+ - E^+$ and $e^- = \kappa\ovl{e}^- - E^-$
\item
$0 \leq E^+,E^-,H^+,H^- \leq \kappa - 1$
\item
$ E^+ + E^- + H^+ + H^- = 2(\kappa -1)$.
\end{itemize}
By \eqref{eq:Frind} we then have
\begin{align*}
\ind(D) - \kappa\ind(\ovl{D}) &= (a+b - 2) - \kappa(\ovl{a} + \ovl{b} - 2) + (e^+ - e^-) - \kappa(\ovl{e}^+ - \ovl{e}^-)\\
&= 2\kappa-2 - 2E^+ - H^+ - H^-.
\end{align*}

Consider first the case that $\ovl{D}$ is trivial. Then we have 
$\ind(\ovl{D}) = 0$ and $\ovl{a} = \ovl{b} = 1$.
If the ends of $\ovl{D}$ are elliptic, then we have 
 $H^+ = H^- = 0$, and hence 
\begin{align*}
\ind(D) = 2\kappa -2 - 2E^+ \geq 0.
\end{align*}
Similarly, if the ends of $\ovl{D}$ are hyperbolic, then we have
$E^+ = E^- = 0$, and hence
\begin{align*}
\ind(D) = 2\kappa - 2 - H^+ - H^- \geq 0
\end{align*}
In either case we have $\ind(D) + 3(b - 1) = 0$ if and only if $a = b = 1$, in which case $D$ is a trivial cylinder. 

Now consider the case that $\ovl{D}$ is nontrivial, and hence $\ind(\ovl{D}) \geq 1$. 
We have
\begin{align*}
\ind(D) + 3(b-1) &\geq 3\kappa - 2 - 2E^+ - H^+ - H^- + 3(b-1) \\
&= 3\kappa - 2 - E^+ - (E^+ - H^+ - H^-) + 3(b-1)\\
&\geq 3\kappa -2 - (\kappa-1) - 2(\kappa-1) + 3(b-1)\\
&= 1 + 3(b-1) \geq 1. \qedhere
\end{align*}
\end{proof}

\begin{lemma}\label{lem:C_index}
Let $C$ be curve in $\tX_\Omega$ satisfying a constraint $\lll \T^{(m)} p\rrr$ for some $m \in \Z_{\geq 1}$, and assume that $C$ is a $\kappa$-fold cover of its underlying simple curve $\ovl{C}$ for some $\kappa \in \Z_{\geq 1}$.
Let 
\begin{itemize}
\item
$e$ (resp. $\ovl{e}$) denote the number of elliptic positive ends of $C$ (resp. $\ovl{C}$)
\item
$h$ (resp. $\ovl{h}$) denote the number of hyperbolic positive ends of $C$ (resp. $\ovl{C}$)
\item
$q = e + h$ (resp. $\ovl{q} = \ovl{e} + \ovl{h}$) denote the total number of positive ends of $C$ (resp. $\ovl{C}$).
\end{itemize}
Then we have $$\ind(C) - \kappa \ind(\ovl{C}) \geq \max(q - 2 - \kappa\ovl{q} + 2\kappa + e - \kappa\ovl{e},\kappa \ovl{h} - h).$$
\end{lemma}

\NI Note that in particular we have $h \leq \kappa\ovl{h}$ and hence $\ind(C) \geq \kappa \ind(\ovl{C})$.

\begin{proof}
Let $E$ (resp. $H$) denote the sum of the excess branching at all elliptic (resp. hyperbolic) punctures of $C$, and let $B$ denote the excess branching of the point in the domain of $C$ which satisfies the constraint $\lll \T^{(m)}p\rrr$.
The curve $\ovl{C}$ satisfies a constraint $\lll \T^{(\ovl{m})}p\rrr$ for some $\ovl{m} \in \Z_{\geq 1}$.
With the help of the Riemann--Hurwitz formula we have
\begin{itemize}
\item
$B \leq \kappa - 1$,
\item
$e = \kappa\ovl{e} - E$
\item
$h = \kappa\ovl{h} - H$
\item
$B + E + H \leq 2\kappa - 2$
\item
$m \leq (B+1)\ovl{m}$
\end{itemize}

For $s = 1,\dots,q$, let $\gamma_s$ be the $s$th positive end of $C$, which we take to be either $e_{i_s,j_s}$ or $h_{i_s,j_s}$.
Similarly, for $s = 1,\dots,\ovl{q}$, let $\ovl{\gamma}$ be the $s$th positive end of $\ovl{C}$, which we take to be either $e_{\ovl{i}_s,\ovl{j}_s}$ or $h_{\ovl{i}_s,\ovl{j}_s}$.

By \eqref{eq:Frind}, we have
\begin{align*}
\ind(C) = q - 2 + 2\sum_{s=1}^q (i_s + j_s) + e - 2m\\
\ind(\ovl{C}) = \ovl{q} - 2 + 2\sum_{s=1}^{\ovl{q}} (\ovl{i}_s + \ovl{j}_s) + \ovl{e} - 2\ovl{m},
\end{align*}
and therefore
\begin{align*}
\ind(C) - \kappa\ind(\ovl{C}) &= (q-2) - \kappa(\ovl{q}-2) + e - \kappa \ovl{e} - 2m + 2\ovl{m}\kappa\\
&\geq q - 2 - \kappa \ovl{q} + 2\kappa + e - \kappa\ovl{e} - 2(B+1)\ovl{m} + 2\ovl{m}\kappa\\
&= q - 2 - \kappa\ovl{q} + 2\kappa + e - \kappa\ovl{e} + 2\ovl{m}(\kappa - B - 1)\\
&\geq q - 2 - \kappa\ovl{q} + 2 \kappa + e - \kappa\ovl{e} + 2\kappa - 2B - 2\\
&\geq q - 2 - \kappa\ovl{q} + 2\kappa + e - \kappa\ovl{e}.
\end{align*}
Note that we have $q - \kappa\ovl{q} = -E-H \geq B - 2\kappa + 2$, so we also have
\begin{align*}
\ind(C) - \kappa\ind(\ovl{C}) &\geq B - 2\kappa + 2 - 2 + 2\kappa + e - \kappa\ovl{e} + 2\kappa - 2B - 2\\
&= -B + e - \kappa\ovl{e} + 2\kappa - 2\\
&= -B - E + 2\kappa - 2\\
&\geq H = \kappa\ovl{h} - h. \qedhere
\end{align*}
\end{proof}

Recall that, for $(i,j) \in \Z_{\geq 0}^2 \setminus \{(0,0)\}$, the pair of acceptable Reeb orbits $e_{i,j},h_{i,j}$ come from perturbing an $S^1$-family of Reeb orbits in $\mu^{-1}(p_{i,j}) \subset \bdy X_\Omega^\fr$.
The precise perturbation is controlled by a choice of Morse function $f: S^1 \rightarrow \R$, which we can assume is perfect.
We take $\wt{X}_\Omega$ to be an arbitrarily small perturbation of $X_\Omega^\fr$, and, fixing $J_\mb \in \Jadm(X_{\Omega}^\fr)$, we can correspondingly consider $J \in \Jadm(\wt{X}_\Omega)$ which is a small perturbation of $J_\mb$.
Then by the standard correspondence between Morse gradient flowlines and Morse--Bott cascades, one expects $J$-pseudoholomorphic cylinders with positive asymptotic $e_{i,j}$ and negative asymptotic $h_{i,j}$ to correspond to gradient flow lines for $f$, of which there are precisely two and they have cancelling signs.
Indeed, by the Morse--Bott techniques developed in \cite{Bourgeois_MB,bourgeois2009symplectic} (see also \cite[\S10.3]{wendl_SFT_notes} for a detailed discussion and also an alternative perspective) we have the following standard result:

\begin{lemma}\label{lem:low_energy_cyls}
There exists generic $J_{\bdy\wt{X}_\Omega} \in \Jadm(\bdy\wt{X}_\Omega)$ such that for each acceptable pair $e_{i,j},h_{i,j}$ there are precisely two $J$-holomorphic cylinders in $\R \times \bdy \wt{X}_\Omega$ with positive asymptotic $e_{i,j}$ and negative asymptotic $h_{i,j}$. Moreover, these are regular and count with opposite signs.
\end{lemma}
\NI Note that the cylinders in Lemma~\ref{lem:low_energy_cyls} have energy $\tcalA(e_{i,j}) - \tcalA(h_{i,j})$, which by Lemma~\ref{lem:fr_action_conds}(a) is very small; we will refer to them as {\bf low energy cylinders}.

\sss

\begin{prop}\label{prop:min_pert_inv}
Assume $a > b$. For $k \in \Z_{\geq 1}$, let $\wmin$ be the weakly permissible word of index $2k$ with minimal $\tcalA_\Omega$ value.
Then the formal curve component in $\tX_\Omega$ having positive asymptotics $\wmin$ is formally perturbation invariant with respect to any generic $J_{\bdy \tX_\Omega} \in \Jadm(\bdy \tX_\Omega)$ as in Lemma~\ref{lem:low_energy_cyls}.
\end{prop}

\begin{proof}
By Lemma~\ref{lem:make_elliptic} $\wmin$ must be elliptic,
and it must take one of the forms (1),(2),(3),(4) from Corollary~\ref{cor:weak_perm}.
Let $C$ denote the formal curve component in $\tX_\Omega$ having positive asymptotics $\wmin$.
After possibly replacing $C$ with another formal curve component which it formally covers, we can assume that $C$ is simple, i.e. we can ignore the case $\wmin = e_{0,k}$ with $k \geq 2$.
Now consider a stable formal building $C' \in \ovl{\formal}_{X,A}(\Gamma)\lll \T^{(k)}p\rrr$ satisfying conditions (A1) and (A2) from Definition~\ref{def:formal_pert_inv}.
We seek to show that $C'$ satisfies either (B1) or (B2) with respect to $J_{\bdy\wt{X}_\Omega}$.

Let $C_0$ denote the main component of $C'$, i.e. the one in $\wt{X}_\Omega$ which carries the local tangency constraint.
We can assume that $C'$ involves at least one symplectization level, since otherwise we must have $C' = C_0$, whence (B1) holds.
Let $q \in \Z_{\geq 1}$ denote the number of positive ends of $C_0$.
Excluding $C_0$, we can view $C'$ as some number $q$ of connected buildings with one unpaired negative end precisely as in Lemma~\ref{lem:symp_configs}. Denote these by $C_1,\dots,C_q$.
We have $\ind(C_s) \geq 0$, with equality if and only if $C_s$ consists entirely of trivial cylinders.
In particular, if the unpaired negative end of $C_s$ is hyperbolic, the fact that its top is elliptic implies that
 $\ind(C_s) \geq 1$.   Thus if $C_s$ has a hyperbolic end, we have $\ind(C_s)\ge1$ so that
$\sum_{s=1}^q \ind(C_s)$ $ \geq h$, where $h$ denotes the number of hyperbolic ends of $C_0$.

Next suppose $D$ is one of $C_1,\dots,C_q$ with $\ind(D) = 1$. Then we claim that $D$ is a low energy cylinder (that is, a cylinder connecting some $e_{i,j}$ and $h_{i,j}$), possibly along with extra trivial cylinders in other levels.
Indeed, for parity reasons the unpaired negative end must be hyperbolic, say $h_{i,j}$ for some $(i,j) \in \Z_{\geq 1}^2$.
Let $\wmin'$ denote the word obtained from $\wmin$ by replacing the set of the positive ends of $D$ by
 $e_{i,j}$.
Then $\wmin'$ is strongly permissible and satisfies $\ind(\wmin') = \ind(\wmin)$ and $\tcalA(\wmin') \leq \tcalA(\wmin)$, with equality only if $\wmin = e_{i,j}$.
Then by minimality of $\wmin$ we must have $\wmin = e_{i,j}$, and the claim follows by energy considerations.

\sss

Assume now that $C_0$ is a $\kappa$-fold cover of a simple formal curve component $\ovl{C}_0$ for some $\kappa \in \Z_{\geq 1}$.
By assumption we have $\ind(\ovl{C}_0) \geq -1$.
Let $e$ denote the number of elliptic ends of $C_0$ and define $\ovl{q},\ovl{e},\ovl{h}$ analogously for $\ovl{C}_0$.
Suppose first that we have $\ovl{h} = 0$ and hence $h = 0$. In this case, $\ovl{C}_0$ has only elliptic ends and hence its index must be even, so we have a fortiori $\ind(\ovl{C}_0) \geq 0$. 
Applying Lemma~\ref{lem:C_index}, we have
$$
0 = \ind(C_0) + \sum_{s=1}^q \ind(C_s) \geq \kappa\ind(\ovl{C}_0) + \kappa\ovl{h} - h + \sum_{s=1}^q\ind(C_s) \geq \sum_{s=1}^q \ind(C_s).$$
This is only possible if $C_s$ consists entirely of trivial cylinders for $s = 1,\dots,q$, but this contradicts the stability of $C'$.

Now suppose that $\ovl{h} \geq 1$, and moreover that the covering $C_0 \rightarrow \ovl{C}_0$ is not ramified at any positive punctures.
In this case we have $q = \kappa \ovl{q}$, $e = \kappa \ovl{e}$, and $h = \kappa\ovl{h}$, and hence
\begin{align*}
\ind(C_0) - \kappa\ind(\ovl{C}_0) &\geq q - 2 - \kappa\ovl{q} + 2\kappa + e - \kappa\ovl{e}\\
&= -2 + 2\kappa.
\end{align*}
We then have
\begin{align*}
0 = \ind(C_0) + \sum_{s=1}^q \ind(C_s) &\geq \kappa\ind(\ovl{C}_0) -2 + 2\kappa + h\\
&\geq \kappa - 2 + \kappa \ovl{h}\\
&\geq 2\kappa - 2.
\end{align*}
This is only possible if $\kappa = 1$, and hence $\ind(C_0) = -1$.
Then we have $\ind(C_s) \leq 1$ for $s = 1,\dots,q$, with equality for at most one $s$. By the above discussion and stability considerations, we conclude that $C'$ is a breaking of the form (B2).

Finally, suppose that $\ovl{h} \geq 1$ and also one of the positive punctures of $C_0$ is ramified.
Then the corresponding component $C_s$ cannot be a low energy cylinder, and so
as explained above, 
it  must then satisfy $\ind(C_s) \geq 2$.  Thus we have
$\sum_{s=1}^q \ind(C_s) \geq h + 1$, and hence 
\begin{align*}
 0 = \ind(C) + \sum_{s=1}^q \ind(C_s) &\geq -\kappa + \kappa\ovl{h} - h + \sum_{s=1}^q\ind(C_s) \\&\geq \kappa(\ovl{h} - 1)+1 \geq 1,
\end{align*}
which is impossible.
\end{proof}

\subsection{Automatic transversality and positive signs}\label{subsec:aut_trans}

Our main goal in this subsection is to prove Proposition~\ref{prop:count_positively}, which roughly states that rigid curves in dimension four count with positive sign as long as none of the punctures are asymptotic to positive hyperbolic Reeb orbits (e.g. $h_{i,j}$).
 This will later allow us conclude that certain moduli spaces have nonzero signed counts simply by showing that they are nonempty. 
We note that the content of this subsection is likely well-known to experts, but we include a precise statement and proof for the sake of completeness.

To begin, let us recall a version of the automatic transversality criterion
 from\cite{Wendl_aut}.
A pseudoholomorphic curve satisfying this criterion is regular even without any genericity assumption on the almost complex structure. 
It is natural to state the results in this subsection in arbitrary genus.

\begin{thm}\label{thm:aut_trans}
Let $X$ be a four-dimensional compact symplectic cobordism, take $J \in \Jadm(X)$, and let $C$ be a nonconstant asymptotically cylindrical $J$-holomorphic curve component of genus $g(C)$ in $\wh{X}$ such that all of the asymptotic Reeb orbits are nondegenerate. Let $h_+(C)$ denote the number of punctures (positive or negative) which are asymptotic to positive hyperbolic Reeb orbits, and let $Z(C)$ be the count (with multiplicities) of zeroes of the derivative of a map representing $C$.
If 
$$2g(C) - 2 + h_+(C) + 2Z(C) < \ind(C),$$
then $C$ is regular.
\end{thm}
\NI It is useful to point out that the quantity $Z(C)$ is always nonnegative, and is zero if and only if $C$ is immersed.

As above let $X$ be a four-dimensional compact symplectic cobordism with $\bdy^{\pm}X$ nondegenerate.
Let us pick coherent orientations for all moduli spaces of immersed asymptotically cylindrical pseudoholomorphic curves in $X$ following the framework of \cite[\S 9]{HTII} (this is quite to similar to the approach of \cite{bourgeois2004coherent}; see also \cite[\S 11]{wendl_SFT_notes}).
This involves the following main ingredients.
An {\bf orientation triple} is a triple $(\Sigma,E,\{S_k\})$, where
\begin{itemize}
\item
$\Sigma$ is a  Riemann surface with positive and negative cylindrical ends
\item
$E$ is a Hermitian complex line bundle over $\Sigma$, trivialized over each end
\item
at the $k$th end we have a smooth family of symmetric $2 \times 2$ matrices, $S_k \in C^{\infty}(S^1,\op{End}^{\op{sym}}_{\R}(\R^2))$, such that the asymptotic operator 
\begin{align*}
\mathbb{A}: C^{\infty}(S^1,\C) \rightarrow C^{\infty}(S^1,\C),\;\;\;\;\; \eta(t) \mapsto -J_0\bdy_t\eta(t) - S_k(t)\eta(t)
\end{align*}
is nondegenerate, i.e. does not have $0$ as an eigenvalue.
\end{itemize}
Here $J_0$ denotes the matrix $\left(\begin{smallmatrix} 0 & -1 \\ 1 & 0 \end{smallmatrix}\right)$. 

For each orientation triple $(\Sigma,E,\{S_k\})$, we denote by $\mathcal{D}(\Sigma,E,\{S_k\})$ the space of differential operators $D: C^{\infty}(E) \rightarrow C^{\infty}(T^{0,1}\Sigma \otimes E)$ which look locally like a zeroth order perturbation of the Cauchy--Riemann operator $\ovl{\bdy}$ on $E$ for some choice of conformal structure on $\Sigma$, and where on the $k$th end in cylindrical coordinates $D$ has the form
$$\psi(s,t) \mapsto (\ovl{\bdy}+ M_k(s,t))\psi(s,t)d\ovl{z}$$ with $\lim\limits_{|s| \rightarrow \infty}M_k(s,t) = S_k(t)$.
Each $D \in \mathcal{D}(\Sigma,E,\{S_k\})$ extends to an operator $W^{1,2}(E) \rightarrow L^2(T^{0,1}\Sigma \otimes E)$, and this is Fredholm since the corresponding asymptotic operators are nondegenerate. Moreover, the space of such operators is an affine space and thus contractible, and hence the set of orientations of the determinant lines of any two elements of $\mathcal{D}(\Sigma,E,\{S_k\})$ are naturally identified. We denote the set of these two possible orientations by $\mathcal{O}(\Sigma,E,\{S_k\})$.

Now, to orient moduli spaces of curves we choose preferred orientations in $\mathcal{O}(\Sigma,E,\{S_k\})$ ranging over all possible orientation triples $(\Sigma,E,\{S_k\})$, subject to axioms (OR1), (OR2), (OR3),(OR4). These axioms roughly correspond to compatibility under gluing and disjoint unions and agreement with the natural complex orientation whenever $D$ happens to be complex-linear.
Henceforth we will implicitly assume that a choice of coherent orientations has been made.
Given such a choice, any moduli space of regular, immersed, asymptotically cylindrical curves in $X$ naturally inherits an orientation.
Indeed, for a curve $C$ in such a moduli space we have an associated orientation triple $(\Sigma,E,\{S_k\})$, where $\Sigma$ is the domain of the curve, $E = N_C$ its normal bundle, and $\{S_k\}$ is given by the induced asymptotic operators at each puncture (see e.g. \cite[\S3]{wendl_SFT_notes}). 
Then the associated deformation operator $D_C$ lies in $\mathcal{D}(\Sigma,E,\{S_k\})$, and by regularity its determinant line is its kernel, which is also the tangent space to the corresponding moduli space.

In the special case of Fredholm index zero, surjectivity of $D_C$ means that we have an identification $\det(D_C) = \R$, and the associated sign $\eps(C) \in \{1,-1\}$ is determined by whether our chosen orientation of $\det(D_C)$ agrees or disagrees with the canonical orientation of $\R$.

\begin{prop}\label{prop:count_positively}
Let $C$ be an immersed, 
somewhere injective, 
asymptotically cylindrical $J$-holomorphic rational curve in a four-dimensional symplectic cobordism $X$. Assume that we have $\ind(C) = 0$, and all of the asymptotic Reeb orbits of $C$ are nondegenerate and are either elliptic or negative hyperbolic. Then we have $\eps(C) = 1$.
\end{prop}
\begin{proof}
Since $C$ is immersed, it has a well-defined normal bundle $N_C \rightarrow C$ and associated deformation operator $D_C$, which we can view as a Fredholm operator 
$W^{1,2}(N_C) \rightarrow L^2(T^{0,1}\Sigma \otimes E)$ (here $\Sigma$ denotes the domain of $C$).
According to \cite[Theorem 3.53]{wendl_SFT_notes}, any two nondegenerate asymptotic operators with the same Conley--Zehnder index are homotopic through nondegenerate asymptotic operators. In particular, if $\gamma$ is an elliptic or negative hyperbolic Reeb orbit, we can deform its asymptotic operator $\mathbb{A}_\gamma$ through nondegenerate asymptotic operators to be of the form given in \cite[Ex. 3.60]{wendl_SFT_notes}, i.e. $\mathbb{A} = -J_0 \bdy_t - \epsilon$ for some $\epsilon \in \R \setminus 2\pi\Z$. 
Note that in this case the associated symplectic parallel transport rotates the contact planes along $\gamma$ by total angle $\epsilon$ in the chosen trivialization.
It follows that we can deform $D_C$ through Fredholm operators, after which the asymptotic operator at each end is complex-linear. 
The resulting Cauchy--Riemann type operator might not be complex-linear, but we can further deform it to its complex-linear part.
We can take this latter deformation to be along an affine line and hence asymptotically constant on each end, meaning that it is a deformation through Fredholm operators.
Combining these two deformations, the corresponding $\Z/2$ spectral flow gives the sign $\eps(C)$.

At the same time, by automatic transversality, the Fredholm operators in this deformation are isomorphisms throughout, and hence the spectral flow is trivial.
Indeed, this follows by invoking the criterion $2g(\Sigma) - 2 +  h_+(C) < \ind(C)$, after noting that Theorem~\ref{thm:aut_trans} holds also on the level of operators in $\mathcal{D}(\Sigma,E,\{S_k\})$.
Finally, observe that we have endowed the determinant line of the complex linear operator at the end of the deformation with its canonical complex orientation, which is necessarily positive.
\end{proof}

\begin{rmk}\label{rmk:aut_trans} (i)
The above discussion has a natural analogue in a symplectization $\R \times Y$, in which we consider the signed count of index one curves modulo target translations. 
Note that positivity does not hold for the low energy cylinders in $\R \times \bdy \tX_\Omega$ that connect $e_{i,j} $ to $h_{i,j}$, and indeed in that case the negative end is positive hyperbolic.\MS

\NI (ii)  In Proposition~\ref{prop:T=E} we assert that each curve in
the moduli space $\#\calM_{X,A}^J(\Gamma^+;\Gamma^-)\lll \T^{(m)}p\rrr$ also counts positively when the orbits in $\Ga^+, \Ga^-$ are elliptic or negative hyperbolic.    To prove this, one must check that the tangency constraint is always compatible with the orientation.   This is proved in \cite[Lem.2.3.5]{McDuffSiegel_counting}.
\end{rmk}

\subsection{Obstruction bundle gluing}\label{subsec:obg}

In this subsection we briefly review the Hutchings--Taubes theory of obstruction bundle gluing \cite{HuT,HTII}, after making the minor adaptations necessary to glue curves in cobordisms rather than symplectizations.
As noted also in \cite{Mint}, since the gluing is essentially local to the neck region, which is the same in both cases, the underlying analysis of \cite{HuT,HTII} still applies in the cobordism setting.

Let $X^+$ and $X^-$ be four-dimensional compact symplectic cobordisms 
with common strict contact bounday $Y := \bdy^-X^+ = \bdy^+X^-$.
We will assume that all Reeb orbits of $Y$ under discussion are nondegenerate. By concatenating, we can form the compact symplectic cobordism $X := X^+ \circledcirc X^-$. Fix a generic admissible almost complex structure $J_Y \in \Jadm(Y)$, and let $J^{\pm}$ be generic admissible almost complex structures on $\wh{X^{\pm}}$
which restrict to $J_Y$ on the corresponding ends, i.e. we have $J^+ \in \Jadm_{J_Y}(X^+)$ and $J^- \in \Jadm^{J_Y}(X^-)$.
Let $\alpha^+, \beta^+,\beta^-,\alpha^-$ be tuples of Reeb orbits in $\bdy^+X^+,Y,Y,\bdy^-X^-$ respectively.

\begin{definition}[c.f. Definition 1.9 in \cite{HuT}]
A {\bf gluing pair} is a pair $(u_+,u_-)$ consisting of immersed pseudoholomorphic curves $u^+ \in \calM_{X^+}^{J^+}(\alpha^+;\beta^+)$ and $u_- \in \calM_{X^-}^{J^-}(\beta^-;\alpha^-)$ such that:
\begin{enumerate}
\item[{\rm (a)}] $\ind(u_+) = \ind(u_-) = 0$
\item[{\rm (b)}] $u_+$ and $u_-$ are simple\footnote{In the symplectization setting, Hutchings--Taubes also allow some components of $u_+$ and $u_-$ to be trivial cylinders, subject to a certain combinatorial condition.}
\item[{\rm (c)}] for each simple Reeb orbit $\gamma$ in $Y$, the total covering multiplicity of Reeb orbits covering $\gamma$ in the list $\beta^+$ is the same as the total for $\beta^-$.
\item[{\rm (d)}] each component of $u_+$ has exactly one negative end, and each component of $u_-$ has exactly one positive end.
\end{enumerate}
\end{definition}
\NI Here we consider a possibly disconnected curve to be simple if and only if each component is simple and no two components have the same image.
\begin{rmk}
We point out that condition (d) is a somewhat artificial simplifying assumption which is used to ensure that we do not encounter higher genus curves after gluing rational curves. Alternatively, the following discussion holds equally if we drop this condition and simply allow $u_\pm$ and also the gluing result to have higher genus.
\end{rmk}

Observe that $J^+$ and $J^-$ can also be concatenated to give $J \in \Jadm(X)$ satisfying $J|_{X^{\pm}} = J^{\pm}|_{X^{\pm}}$.
For each $R > 0$, let
$$ X_R := X^+ \circledcirc ([-R,R] \times Y) \circledcirc X^-$$
denote the compact symplectic cobordism given by inserting a finite piece of the symplectization of $Y$ in between $X^+$ and $X^-$. 
Let also $J_R \in \Jadm(X_R)$ denote the concatenated almost complex structure which satisfies 
$J_R|_{X^{\pm}} = J^{\pm}|_{X^{\pm}}$ and $J_R|_{[-R,R] \times Y} = J_Y|_{[-R,R] \times Y}$.
Note that the family $\{J_R\}_{R \in [0,\infty)}$ realizes neck-stretching along $Y$, with the limit $R \rightarrow \infty$ corresponding to $(J^+,J^-)$-holomorphic buildings in the broken cobordism $X^+ \notccirc X^-$.
We denote the corresponding parametrized moduli space by $\calM_{X}^{\{J_R\}}(\alpha^+;\alpha^-)$ and its SFT compactification by $\ovl{\calM}_{X}^{\{J_R\}}(\alpha^+;\alpha^-)$.

Given a gluing pair $(u_+,u_-)$, Hutchings--Taubes glue together $u_+$ and $u_-$ after possibly inserting a union $u_0$ of index zero branched covers of trivial cylinders in an intermediate symplectization level $\R \times Y$. 
This is more complicated than the typical gluing encountered in SFT, where the intermediate level $u_0$ would be barred from participating in the gluing since it is irregular. 
Indeed, note that $u_0$ lives in a moduli space $\calM_Y^{J_Y}(\beta^+;\beta^-)$ of branched covers which has dimension $2b$, where $b$ corresponds to the number of interior branch points.
The main computation of \cite{HuT} determines the signed number $\#G(u_+,u_-)$
of ends of $\calM_X^{\{J_R\}}(\alpha^+;\alpha^-)$  which arise by gluing $(u_+,u_-)$ in this way.

Analogously to \cite[\S5]{HTII}, one can perform pregluing to produce an approximately $J_R$-holomorphic curve in $\wh{X_R}$ which interpolates via cutoff functions between $u_+$ on $\wh{X^+}$, $u_0$ on $\R \times Y$, and $u_-$ on the $\wh{X^-}$.
The index of the normal deformation operator of $u_0$ is $-2b$ and the kernel can be shown to be trivial, so the cokernels as $u_0$ varies form a well-defined rank $2b$ real vector bundle over (a large compact subspace of) $\calM_Y^{J_Y}(\beta^+;\beta^-)$, called the ``obstruction bundle''.
From the gluing analysis we get a section $\mathfrak{s}$, such that the gluing successfully goes through for $u_0 \in \calM_Y^{J_Y}(\beta^+;\beta^-)$ precisely if $\mathfrak{s}(u_0) = 0$. 
The computation of $\#G(u_-,u_+)$ therefore amounts to counting zeros of $\mathfrak{s}$.

\sss

More precisely, the number $\#G(u_-,u_+)$ is defined in several steps as follows.
For each $R \geq 0$, fix a metric on $\wh{X_R}$ which is a product metric on the cylindrical ends $(-\infty,0]\times \bdy^-X$ and $[0,\infty) \times \bdy^+X$ and on the neck region $[-R,R] \times Y$. We assume this metric does not depend on $R$ except for the varying length of the neck.

\begin{definition}[c.f. Definition 1.10 in \cite{HuT}]
For $\delta > 0$, let $\calC_\delta(u_+,u_-)$ denote the union over $R \in (1/\delta,\infty)$ of the set of surfaces in $\wh{X_R}$ which are immersed apart from finitely many points and can be decomposed as $C_- \cup C_0 \cup C_+$, where:
\begin{itemize}
\item there is a section $\psi_+$ of the normal bundle of $u_+$ restricted to $$\left([-1/\delta,0] \times Y\right) \cup X^+ \cup \left([0,\infty) \times \bdy^+X^+\right)$$ such that $||\psi_+|| < \delta$ and $C_+$ is the exponential map image of $\psi_+$ after identifying $[-1/\delta,0] \times Y$ with $[R-1/\delta,R] \times Y$
\item 
there is a section $\psi_-$ of the normal bundle of $u_-$ restricted to $$\left( (-\infty,0] \times \bdy^-X^-\right) \cup X^+ \cup \left([0,1/\delta] \times Y\right)$$ such that $||\psi_-|| < \delta$ and $C_+$ is the exponential map image of $\psi_-$ after identifying $[0,1/\delta] \times Y$ with $[-R,-R+1/\delta] \times Y$
\item
$C_0$ lies in the $\delta$-tubular neighborhood of $[-R,R] \times (\beta^+ \cup \beta^-) \subset [R,R] \times Y$, and we have $\bdy C_0 = \bdy C_- \cup \bdy C_+$, with the positive boundary of $C_0$ coinciding with the negative boundary of $C_+$ and the negative boundary of $C_0$ coinciding with the positive boundary of $C_-$.
\end{itemize}
\end{definition}

\begin{definition}
Let $\calG_{\delta}(u_+,u_-)$ denote the set of index zero
curves in 
$$\calM^{\{J_R\}}_{X}(\alpha^+;\alpha^-) \cap \mathcal{C}_\delta(u_+,u_-).$$
\end{definition}
By the following lemma, $\calG_{\delta}(u_+,u_-)$ represents curves in $\calM_{X}^{\{J_R\}}(\alpha^+;\alpha^-)$ which are ``$\delta$-close'' to breaking into an SFT building corresponding to the gluing pair $(u_+,u_-)$:

\begin{lemma}[c.f. Lemma 1.11 in \cite{HuT}]\label{lem:close_to_breaking}
Given a gluing pair $(u_+,u_-)$, there exists $\delta_0 > 0$ such that for any $\delta \in (0,\delta_0)$ and any sequence of curves $u_1,u_2,u_3,\dots \in \mathcal{G}_{\delta}(u_+,u_-)$, there is a subsequence which converges in the SFT sense to either a curve in $\calM^{J_{R_\infty}}_{X_{R_\infty}}(\alpha^+;\alpha^-)$ for some $R_\infty \in [0,\infty)$, or else to an SFT building with top level $u_+$ in $\wh{X^+}$, bottom level $u_-$ in $\wh{X^-}$, and some number (possibly zero) of intermediate symplectization levels in $\R \times Y$ each consisting entirely of unions of index zero branched covers of trivial cylinders.
\end{lemma}

Finally, we define the count of ends $\#G(u_+,u_-)$:
\begin{definition}
For a gluing pair $(u_+,u_-)$ and $\delta_0$ as above, choose $0 < \delta' < \delta < \delta_0$ and an open subset $U \subset \calG_{\delta}(u_+,u_-)$ containing $\calG_{\delta'}(u_+,u_-)$ such that $\ovl{U}$ has finitely many boundary points.
We then define $\#G(u_-,u_+)$ to be minus the signed count of boundary points of $\ovl{U}$.
\end{definition}
\NI By Lemma~\ref{lem:close_to_breaking}, the count $\#G(u_+,u_-)$ is independent of the choice of $\delta',\delta,U$.

The analogue of the main result of Hutchings--Taubes is as follows:

\begin{thm}[c.f. Theorem 1.13 in \cite{HuT}]
If $J^+ \in \Jadm(X^+)$ and $J^- \in \Jadm(X^-)$ are generic and $(u_+,u_-)$ is a gluing pair, then we have
$$ 
\# G(u_+,u_-) = \epsilon(u_+)\epsilon(u_-)\prod_\gamma c_{\gamma}(u_+,u_-),
$$
where the product is over all simple Reeb orbits whose covers appear in $\beta^+$ and $\beta^-$, and $c_{\gamma}(u_+,u_-)$ depends only on $\gamma$, the multiplicities of the negative ends of $u_+$ at covers of $\gamma$, and the multiplicities of the positive ends of $u_-$ at covers of $\gamma$.
\end{thm}

\sss

For simplicity, let us now assume that the orbits in $\beta^+$ and $\beta^-$ are all covers of the same simple Reeb orbit $\gamma$ which is elliptic. 
Denote the corresponding partitions by $(a_1,\dots,a_k)$ and $(b_1,\dots,b_\ell)$, where $\sum_{i=1}^k a_i = \sum_{j=1}^\ell b_j$.
Following \cite[\S 1]{HuT}, there is a purely combinatorial algorithm for computing $c_{\gamma}(u_+,u_-)$ in terms of the monodromy angle $\theta$ of $\gamma$ and the partitions $(a_1,\dots,a_k)$ and $(b_1,\dots,b_\ell)$, but it is rather elaborate to state. 
For our purposes, it is enough to observe that, by \cite[Rmk. 1.21]{HuT}, $c_{\gamma}(u_-,u_+)$ is a positive integer provided that there is a branched cover $u_0$ of the trivial cylinder $\R \times \gamma \subset \R \times Y$ which is connected with genus zero and index zero (this is the analogue of $\kappa_\theta = 1$ in \cite{HuT}).
Namely, this criterion holds exactly if
\begin{align}\label{eq:index_cond}
k + \ell - 2 + \sum_{i=1}^k \cz_\tau(\gamma^{a_i}) - \sum_{j=1}^\ell \cz(\gamma^{b_j}) = 0.
\end{align}
Here $\tau$ is any choice of trivialization along $\gamma$, and the left hand side of \eqref{eq:index_cond} is simply the index of $u_0$, noting that the first Chern class term vanishes since we are using the same trivialization along $\gamma$ at the positive and negative ends.
Explicitly, if $\theta$ denotes the monodromy angle of $\gamma$ with respect to $\tau$, 
then we have $\cz_\tau(\gamma^m) = \lfloor m\theta \rfloor + \lceil m\theta \rceil$, and hence
\eqref{eq:index_cond} is equivalent to
\begin{align}\label{eq:monod}
\ell - 1 + \sum_{i=1}^k \lceil a_j\theta \rceil - \sum_{j=1}^\ell \lceil b_j\theta \rceil = 0.
\end{align}
Note that the left hand side of \eqref{eq:monod} is indeed independent of the choice of trivialization, since $\theta$ modulo the integers is independent of $\tau$ and by assumption we have $\sum_{i=1}^k a_i = \sum_{j=1}^\ell b_j$.

We summarize the above discussion as follows:
\begin{thm}\label{thm:obg_cob_summary}
Let $X^\pm$ be four-dimensional compact symplectic cobordisms with common nondegenerate strict contact boundary $Y := \bdy^-X^+ = \bdy^+ X^-$.
Let $J_Y \in \Jadm(Y)$, $J^+ \in \Jadm_{J_Y}(X^+)$ and $J^- \in \Jadm^{J_Y}(X^-)$ be generic admissible almost complex structures. For $R \geq 0$, let
$J_R \in \Jadm(X_R)$ be the concatenated almost complex structure on the symplectic completion of $X_R := X^+ \ccirc \left([-R,R] \times Y\right) \ccirc X^-$ which satisfies $J_R|_{X^\pm} = J^\pm|_{X^\pm}$ and $J_R|_{[-R,R] \times Y} = J_Y|_{[-R,R] \times Y}$.
Let $u_\pm$ be simple immersed $J^\pm$-holomorphic curves in $\wh{X^\pm}$,
such that each component of $u_+$ has exactly one negative end and each component of $u_-$ has exactly one positive end. Assume that the negative ends of $u_+$ are $(\gamma^{a_1},\dots,\gamma^{a_k})$ and the positive ends of $u_-$ are $(\gamma^{b_1},\dots,\gamma^{b_\ell})$, where $\gamma$ is a simple elliptic Reeb orbit in $Y$. 
Assume further that we have $\sum_{i=1}^k a_i = \sum_{j=1}^\ell b_j$ and \eqref{eq:index_cond} holds.
Then for any $R$ sufficiently large there is a simple immersed regular $J_R$-holomorphic curve $u$ in $\wh{X_R}$ with positive asymptotics agreeing with those of $u_+$ and negative asymptotics agreeing with those of $u_-$.
\end{thm}

\begin{rmk}
If $X$ is a compact symplectic cobordism and $C$ is an asymptotically cylindrical $J$-holomorphic curve in $X$ which is simple and has index zero, then $C$ is automatically immersed provided that $J \in \Jadm(X)$ is generic (c.f. \cite[Thm. 4.1]{HTII}).
\end{rmk}

\subsection{Curves with many positive ends via induction}\label{subsec:induction}

We now seek to apply Theorem~\ref{thm:obg_cob_summary} in order to produce genus zero pseudoholomorphic curves in $X \setminus E_\sk$ with one negative end, building on the main construction of \cite{Mint}.

Recall that $E_\sk$ denotes the ellipsoid $E(\eps,\eps s)$ with $s > 1$ sufficiently large and $\eps > 0$ sufficiently small, and by slight abuse we also use the same notation to denote its image under any symplectic embedding $\iota: E_\sk \hooksymp X$.
Here the role of $\eps$ is just to ensure the existence of a symplectic embedding of $E(\eps,\eps s)$ into $X$, while $s$ is the ``skinniness'' factor.
More precisely, in the following context of curves in $X \setminus E(\eps,\eps s)$ with one negative end asymptotic to $\eta_k$ (the $k$-fold cover of the short simple Reeb orbit in $\bdy E(\eps,\eps s)$),
we will say that $E(\eps,\eps s)$ is ``$k$-skinny'' (or simply ``skinny'') if $s > k$.
In this case we have $\cz_{\tau_{\op{ex}}}(\eta_i) = 2i+1$ for $i = 1,\dots,k$, and hence at least for the purposes of index computations we can treat $s$ as being arbitrarily large.
On the other hand, note that for $k < s < k+1$, $E(\eps,\eps s)$ is $k$-skinny but not $(k+1)$-skinny, a fact which we will exploit in the proof of Lemma~\ref{lem:ind_step} given below.

Before proving the aforementioned lemma, we must deal with the following point.
We showed in \cite[Prop.3.1.5]{McDuffSiegel_counting} that if $X$ is  closed then 
 the number of index zero curves with fixed top end and a single negative end on $E_\sk$ is independent of the choice of $\io, \eps$, and $s$.  However, in our situation with $\p X\ne \emptyset$ we must be a little 
 careful  since in general (for example, if $C$ is not formally perturbation invariant as in Proposition~\ref{prop:stab_ub}) there may not be   a well defined count of curves of the given type.
 Therefore, our arguments only establish that there is  a generic $J$  on $X\less E_{\sk}$ for which certain curves exist.

\begin{lemma}\label{lem:skinny}  Let $X$ be a four-dimensional Liouville domain with nondegenerate contact boundary, and suppose that for some generic 
 $J \in \Jadm(X \setminus E_\sk)$ there is a simple 
immersed index zero $J$-holomorphic curve $C$ in $X \setminus E_\sk$  
with negative end asymptotic to $\eta_k$.
Then given any $s>k$ we may take $E_\sk = \io(E(\eps,\eps s))$ for some $\eps>0$ and some $\io: E(\eps,\eps s) \hooksymp X$.
\end{lemma}
\begin{proof}  Let $E' =\eps'\cdot E(1, s)$ where $\eps'>0$ is so small that we can identify $E'$ with a subset of $E_{\sk}$.
Let $J_{X \setminus E'} \in \Jadm(X \setminus E')$ be a generic admissible almost complex structure satisfying
$J_{X \setminus E'}|_{X \setminus E} = J$, and put 
$J_{E \setminus E'} := J_{X \setminus E'}|_{E \setminus E'} \in \Jadm(E \setminus E')$.
By e.g. \cite[Thm. 2]{hind2020j},  there is a regular $J_{E_\sk \setminus E'}$-holomorphic cylinder $Z$ in $E_\sk \setminus E'$ with positive end on $\eta_1$ and negative end on $\eta_1'$, and its $k$-fold cover is regular.   We can glue (in the ordinary SFT sense) $C$ to $Z$ along cylindrical ends to produce a simple $J'_{X \setminus E'}$-holomorphic curve $C'$ in $X \setminus E'$.
 Here $J'_{X \setminus E'} \in \Jadm(X \setminus E')$ corresponds to the concatenation of $J_{E_\sk \setminus E'}$ and $J_{X \setminus E_\sk}$ after inserting a sufficiently long neck region in between and reidentifying the resulting compact symplectic cobordism with $X \setminus E'$.
Note that we can  assume without loss of generality that $J'$ is generic since the curve $C'$ will persist under small perturbations of $J'$.
\end{proof}

\begin{lemma}\label{lem:ind_step}
Let $X$ be a four-dimensional Liouville domain with nondegenerate contact boundary, let $J \in \Jadm(X \setminus E_\sk)$ be generic, 
and let $C_1$ and $C_2$ be simple immersed index zero $J$-holomorphic curves in $X \setminus E_\sk$ that have distinct images.
For $i = 1,2$, assume that $C_i$ has positive ends $\Gamma_i$ and a single negative end $\eta_{k_i}$. 
Then there exists a generic $J' \in \Jadm(X \setminus E_\sk)$ with $J'|_{\bdy X} = J|_{\bdy X} \in \Jadm(\bdy X)$ and a simple immersed index zero $J'$-holomorphic curve $C$ in $X \setminus E_\sk$ which has positive ends $\Gamma_1 \cup \Gamma_2$ and a single negative end $\eta_{k_1+k_2+1}$.
\end{lemma}

\begin{proof}  By Lemma~\ref{lem:skinny} we may suppose that $C_1, C_2$ lie in 
$X \setminus E'$ where $E' =  \io( E(\eps,\eps s))$ for $s = k_1 + k_2 -1 + \de'$ (where $0<\de'<1$).

Next put $E'' := \eps'' \cdot E(1,k_1 + k_2 + 1  + \delta'')$ for $\eps'',\delta'' > 0$ sufficiently small, and choose $ \io'', \eps''$ so that  we have $E'' \subset E'$.
Let $J_{X \setminus E''} \in \Jadm(X \setminus E'')$ be a generic admissible almost complex structure satisfying
$J_{X \setminus E''}|_{X \setminus E'} = J_{X \setminus E'}'$, and put $J_{E' \setminus E''} := J_{X \setminus E''}|_{E' \setminus E''} \in \Jadm(E' \setminus E'')$.
Again by \cite[Thm. 2]{hind2020j} there is a (necessarily simple) $J_{E' \setminus E''}$-holomorphic cylinder $Z$ in $E' \setminus E''$ with positive end $\eta_{k_1+k_2}$ in $\bdy E'$ and negative end $\eta_{k_1+k_2+1}$ in $\bdy E''$. Note that the bottom ellipsoid  $E''$ is skinny, since $k_1 + k_2 + 1 + \delta'' > k_1 + k_2 + 1$. 
However, the top ellipsoid is not, since $s< k_1+k_2$.  In fact, if we choose the split trivialization $\tau_{\op{sp}}$  of the contact distribution on $\p E(1,x)$ as in \cite[\S 3.2]{McDuffSiegel_counting}, then
the monodromy angle of the short orbit is $1/x$, which implies that the cylinder $Z$ has Fredholm index
$$
2(k_1 + k_2 + \lfloor (k_1+k_2)/x\rfloor) -  2 (k_1 + k_2 + 1) = 0.
$$

We now apply Theorem~\ref{thm:obg_cob_summary} with $u_+ := C_1' \cup C_2'$ in $X \setminus E'$ and $u_- := Z$ in $E' \setminus E''$, in other words we glue  in the neck $\R\times \bdy E'$.
Note that \eqref{eq:monod} holds since in $\bdy E' = \p (\eps\cdot E(1,s'))$  the monodromy angle is 
$1/s'$ where $k_i< s'< k_1+ k_2$ so that
$$
\lceil k_1/s\rceil + \lceil k_2/s\rceil  = 2 = \lceil (k_1 + k_2)/s\rceil .
$$
Therefore, there is a curve $C$ as stated.
\end{proof}

\sss

The above lemma suggests a natural inductive strategy for constructing curves. Fix a generic $J_{\bdy \tX_\Omega} \in \Jadm(\bdy \tX_\Omega)$ as in Lemma~\ref{lem:low_energy_cyls}.
As before, $\wmin$ denotes the weakly permissible word with $\tcalA_\Omega$ minimal subject to $\ind(\wmin) = 2k$.  We prove the following lemmas in the next subsection.

\begin{lemma}\label{lem:cyls_exist}
Let $J \in \Jadm^{{J_{\bdy \tX_\Omega}}}(\tX_\Omega \setminus E_\sk)$ be generic.
Consider an elliptic orbit $e_{i,j}$ in $\bdy \tX_\Omega$ such that either $i = 1$ or $j = 1$ (or both).
Then there is a $J$-holomorphic cylinder in $\tX_\Omega \setminus E_\sk$ which is positively asymptotic to $e_{i,j}$ and negatively asymptotic to $\eta_{i+j}$.
\end{lemma}

\begin{lemma}\label{lem:pair_of_pants}
Let $J \in \Jadm^{{J_{\bdy \tX_\Omega}}}(\tX_\Omega \setminus E_\sk)$ be generic.
There is a $J$-holomorphic pair of pants in $\tX_{\Omega}$ 
which is positively asymptotic to $e_{1,1} \times e_{1,1}$ and negatively asymptotic to $\eta_5$.
\end{lemma}

\begin{prop}\label{prop:curves_exist}
Fix $k \in \Z_{\geq 1}$, and assume that $\wmin \neq e_{0,k}$ if $k \geq 2$.
Then there exists $J \in \Jadm^{J_{\bdy \tX_\Omega}}(\tX_\Omega)$ for which the moduli space $\calM_{\tX_\Omega}^J(\wmin)\lll \T^{(m)}p\rrr$ is regular with nonzero signed count. 
\end{prop}

\begin{proof}
Let $C$ be the formal curve in $\tX_\Omega$ with positive ends $\wmin$ and constraint $\lll \T^{(k)}p\rrr$.
Recall that by Proposition~\ref{prop:min_pert_inv} $C$ is formally perturbation invariant with respect to $J_{\bdy \tX_\Omega}$. 
We explained in the introduction to this section that curves in 
this moduli space are robust and always 
count positively.  Hence 
 at this point
it suffices to find a $J$-holomorphic curve in $\tX_\Omega \setminus E_\sk$ with positive asymptotics $\wmin$ and negative asymptotic $\eta_k$ for some $J \in \Jadm^{J_{\bdy \tX_\Omega}}(\tX_\Omega)$.

We proceed to construct the desired curve, whose positive asymptotics $\wmin$ take one of the forms (1),(2),(3) in Proposition~\ref{prop:strong_perm} or possibly $e_{0,1}$, by iteratively applying Lemma~\ref{lem:ind_step}.
Firstly, observe that by Lemma~\ref{lem:cyls_exist}, we can construct any cylinder whose positive asymptotic is one of the Reeb orbits $e_{0,1},e_{1,1},e_{1,s},e_{1,0}$ appearing in Proposition~\ref{prop:strong_perm}.
Similarly, by Lemma~\ref{lem:pair_of_pants} we can construct a pair of pants with positive asymptotics $e_{1,1} \times e_{1,1}$.
We now iteratively construct curves with two or more positive ends by applying Lemma~\ref{lem:ind_step}, with $C_1$ a previously constructed curve in $\tX_{\Omega} \setminus E_\sk$ and $C_2$ a cylinder in $\tX_{\Omega} \setminus E_\sk$ guaranteed by Lemma~\ref{lem:cyls_exist}.
Here we need $C_1$ and $C_2$ to have distinct images, and since neither is a multiple cover this is automatic as long as $C_1$ is not a cylinder with the same positive asymptotic Reeb orbit as $C_2$.
In particular, the curve we seek with positive 
asymptotics  $\wmin$ is readily constructed by this iterative construction.
\end{proof}

\subsection{Existence of cylinders and pairs of pants}\label{subsec:base_cases}
It remains to prove Lemmas~\ref{lem:cyls_exist} and ~\ref{lem:pair_of_pants}.
For this, we will use
various results from the ECH literature, roughly as follows.
Firstly, we use the computation of the ECH of $\tX_\Omega$ from \cite{hutchings2016beyond,choi2016combinatorial}, together with the holomorphic curve axiom for the ECH cobordism map, to establish the existence of a broken current in $\tX_\Omega \setminus E_\sk$ whose positive ends represent the same orbit set as our desired curve.
We then argue that this broken pseudoholomorphic current must in fact be a genuine somewhere injective curve $C$ of Fredholm index zero, but possibly of higher genus, whose ends satisfy the ECH partition conditions.
Using this, we conclude that in specified situations $C$ must have one negative end, the maximal possible number of positive ends, and genus zero.

\sss

Here are  the details. 
Recall that an {\bf orbit set} is a finite set of simple Reeb orbits, along with a choice of positive integer multiplicity for each. 
In the following we will view a word of Reeb orbits as an orbit set by remembering only the total multiplicity of each underlying simple orbit and forgetting the corresponding partition into iterates.
Note this association from words to orbit sets is evidently not one-to-one, e.g. the words $\eta_3$, $\eta_2\times \eta_1$ and $\eta_1\times \eta_1\times \eta_1$ of Reeb orbits in $\bdy E_\sk$ all define the same orbit set.

Similarly, a {\bf pseudoholomorphic current} is a finite set of simple pseudoholomorphic curves (each modulo biholomorphic reparametrizations as usual), along with a choice of positive integer multiplicity for each. 
We defer the reader to e.g. \cite[\S3.4]{Hlect} for the definition of the ECH index $I(C)$.
Since the first and second cohomology groups of $\tX_\Omega$ and $E_\sk$ vanish, their ECH chain complexes (over $\Z/2$ for simplicity) have natural $\Z$-gradings, denoted again by $I$, such that $I(C) = I(\alpha) - I(\beta)$ if $C$ is a holomorphic current which is positively asymptotic to the orbit set $\alpha$ and negatively asymptotic to the orbit set $\beta$.
Also, the compact symplectic cobordism $\tX_\Omega \setminus E_\sk$ induces a grading-preserving cobordism map from the ECH of $\bdy \tX_{\Omega}$ to that of $\bdy E_\sk$.
If $w$ is an elliptic word,\footnote{There is a more general formula computing $I$ for ECH generators involving hyperbolic orbits but we will not need this.} it is shown in \cite[Lem. 5.4]{hutchings2016beyond} that $I(w) = 2(\calL(R)-1) - h(R)$, where:
\begin{itemize}
\item
$R$ denotes the lattice polygon in $\R^2_{\geq 0}$ defined in \S\ref{subsec:permis}
\item 
$\calL(R)$ denotes the number of integer lattice points in the interior and boundary of $R$.
\end{itemize}

\begin{lemma}\label{lem:ECH_cob_curve}
Let $w$ be a word of elliptic Reeb orbits in $\bdy \tX_\Omega$, each of which is simple, and let $J \in \Jadm^{{J_{\bdy \tX_\Omega}}}(\tX_\Omega \setminus E_\sk)$ be generic. 
Then there is a curve $C$ in $\tX_\Omega \setminus E_\sk$, possibly of higher genus, with $\ind(C) = I(C) = 0$ and with positive asymptotics $w$ and negative asymptote $\eta_m$ with $m := \tfrac{1}{2}I(w)$.
\end{lemma}
\begin{proof}

By \cite[Prop. A.4]{hutchings2016beyond} (which assumes \cite[Conj A.3]{hutchings2016beyond}, proved in \cite{choi2016combinatorial}), $w$ (when viewed as a generator of the ECH chain complex) represents a nontrivial homology class in the ECH of $\bdy \tX_\Omega$. 
Let $\beta$ denote its image in the ECH of $\bdy E_\sk$ under the ECH cobordism map $\Phi$ induced by $\tX_\Omega \setminus E_\sk$, and note that $\beta$ is necessarily nontrivial since $\Phi$ is an isomorphism.
Recall (see \cite[\S3.7]{Hlect}) that the ECH chain complex of an irrational four-dimensional ellipsoid has trivial differential, and the orbit set with $k$th largest action has $I = 2k$.
Then $\beta$ is uniquely represented by the orbit set of $\eta_m$ with $m := I(w)$.

Recall that the ECH cobordism map is defined via the isomorphism with Seiberg--Witten Floer homology, yet it is known to satisfy a ``holomorphic curve axiom'', which states that a coefficient can only be nonzero if it is represented by an ECH index zero {\bf broken pseudoholomorphic current}, i.e. the analogue of a stable pseudoholomorphic building but with each level a pseudoholomorphic current.
As a result, we obtain a broken pseudoholomorphic current in $\tX_\Omega \setminus E_\sk$ with positive orbit set $w$ and negative orbit set $\eta_m$.
By \cite[Prop. 3.7]{Hlect}, each symplectization level has nonnegative ECH index,
with ECH index zero if and only if it is a union of trivial cylinders with multiplicities.
 By Lemma~\ref{lem:ind_of_currents_covs} below, the main level in $\tX_\Omega \setminus E_\sk$ also has nonnegative ECH index.
Using the SFT compactness stability condition (recall \S\ref{subsubsec:SFT_cpct}) and the fact that the total ECH index is zero, we conclude that there is only a single level $D$, which is a current $(\ovl{D},\kappa)$ in $\tX_\Omega \setminus E_\sk$, where $\ovl{D}$ is simple and $\kappa \in \Z_{\geq 1}$ represents its multiplicity,
and we have $I(D) = 0$. 
By Lemma~\ref{lem:ind_of_currents_covs} again, we also have $I(\ovl{D}) = 0$. 

By \cite[Thm. 4.15]{hutchings2009embedded}, we must have $\ind(\ovl{D}) = 0$, and $\ovl{D}$ satisfies the positive and negative partition conditions. 
Since the monodromy angle is positive and very small for each acceptable elliptic orbit, the positive partition conditions stipulate that each positive asymptotic orbit of $\ovl{D}$ is simple, i.e. the positive ends are ``as spread out as possible''.
Meanwhile, the negative partition condition implies that $\ovl{D}$ has a single negative end.
Finally, the desired curve $C$ is given by taking a $\kappa$-fold cover of $\ovl{C}$ which is fully ramified at the negative end and unramified at the each of the positive ends.
\end{proof}
\begin{lemma}\label{lem:ind_of_currents_covs}
If $C = (\ovl{C},\kappa)$ is a $J$-holomorphic current in $\tX_\Omega \setminus E_\sk$ with $J \in \Jadm(\tX_\Omega \setminus E_\sk)$ generic, we have $I(C) \geq 0$, with equality only if $I(\ovl{C}) = 0$.
\end{lemma}
\begin{proof}
As in the proof of \cite[thm. 1.19]{hutchings2016beyond}, we can assume that the cobordism $\tX_\Omega \setminus E_\sk$ is ``$L$-tame'' with $L$ sufficiently large, whence the result follows immediately by \cite[Prop. 4.6]{hutchings2016beyond}.
\end{proof}

\begin{proof}[Proof of Lemma~\ref{lem:cyls_exist}]

By Lemma~\ref{lem:ECH_cob_curve}, there is a $J$-holomorphic curve $C$ in $\tX_\Omega \setminus E_\sk$, possibly of higher genus, with $\ind(C) = 0$, with positive asymptotics $e_{i,j}$ and negative asymptotics $\eta_{m}$ for $m := \tfrac{1}{2}I(e_{i,j})$.
As explained above, we have $I(e_{i,j}) = 2(\calL(R)-1)$, where $R$ is the lattice triangle with vertices $(0,0),(0,i),(j,0)$ and $\calL(R)$ denotes the number of integer lattice points in the interior or boundary of $R$.
By our assumption that $i=1$ or $j=1$, we have $\calL(R) = i+j+1$ and hence $m = i+j$.
It now follows immediately using the index formula~\eqref{eq:Frind} and $\ind(C) = 0$ that $C$ has genus zero.
\end{proof}

\begin{proof}[Proof of Lemma~\ref{lem:pair_of_pants}]

This is similar to the above proof. 
In this case Lemma~\ref{lem:ECH_cob_curve} produces a $J$-holomorphic curve $C$ in $\tX_\Omega \setminus E_\sk$ with $\ind(C) = 0$ and with positive ends $e_{1,1} \times e_{1,1}$ and negative end $\eta_m$ for $m := \tfrac{1}{2} I(e_{1,1} \times e_{1,1}) = 5$. The condition $\ind(C) = 0$ then forces the genus to be zero.
\end{proof}

\begin{rmk}
Note that, for $e_{i,j}$ with $i,j \geq 2$ or $e_{1,1}^{\times k}$ with $k \geq 3$, the curve $C$
coming from Lemma~\ref{lem:ECH_cob_curve} will typically be forced to have higher genus.
\end{rmk}

\subsection{Comparison with Gutt--Hutchings capacities}\label{subsec:GH}

The following result likely holds in any dimension, but for concreteness we give the proof in dimension four:

 \begin{prop}\label{prop:GH}
For $X_\Omega$ any four-dimensional convex toric domain, we have:
\begin{align} \label{eq:lequals1}
\gt_k^{\leq 1}(X_\Omega) = c_k^{\op{GH}}(X_\Omega) = \min_{\substack{(i,j) \in \Z_{\geq 0}^2\\ i+j=k}} ||(i,j)||_{\Omega}^*.
\end{align}
 \end{prop}
 \begin{proof}
 The second equality is \cite[Thm. 1.6]{Gutt-Hu}.
 In order to compute $\gt_k^{\leq 1}(X_\Omega)$, we can replace $X_\Omega$ with its full rounding $\tX_\Omega$ as in \S\ref{subsec:fr}.
 As shorthand put $X := X_\Omega$ and $\tX := \tX_\Omega$.
Fix a generic almost complex structure $J_{\bdy \tX} \in \Jadm(\bdy \tX)$
as in Lemma~\ref{lem:low_energy_cyls}, and a generic extension $J_{\tX} \in \Jadm^{J_{\bdy \tX}}(\tX;D)$.

To prove that $\gt_k^{\leq 1}(X)\ge c_k^{\op{GH}}(X_\Omega)$,
observe that by definition we can find a $J_{\tX}$-holomorphic plane $C$ in $\tX$ satisfying the local tangency constraint $\lll \T^{(k)}p\rrr$ and having $E(C) \leq \gt_k^{\leq 1}(\tX)$.
Let $\gamma$ denote the asymptotic Reeb orbit of $C$, which we can take to be $e_{i,j}$ or $h_{i,j}$ for some $i,j$.
If $C$ is simple then by genericity it must be regular and hence satisfy $\ind(C) \geq 0$, and inspection of the index formula shows that this is also true if $C$ is a multiple cover.
In particular, we must have $i+j \geq k$, from which it follows that $\calA(\gamma) = ||(i,j)||_{\Omega}^*$ is greater than or equal to the right hand side of ~\eqref{eq:lequals1}.
Since $E(C) = \tcalA(\gamma)$ is arbitrarily close to $\calA(\gamma)$, this gives the desired lower bound.

To establish the upper bound for $\gt_k^{\leq 1}(X)$, let $(i,j)$ be a minimizer for the right hand side of ~\eqref{eq:lequals1}. 
We can assume that there are no common divisors of $i,j,k$, and we will then show that $\gt_k^{\leq 1}(X) \leq ||(i,j)||_{\Omega}^*$.
Indeed, if there is a greatest common divisor $q \geq 2$ of $i,j,k$, then after putting $i' := i/q$, $j' := j/q$, $k' := k/q$ it will follow that we have
$\gt_{k'}^{\leq 1}(X) \leq ||(i',j')||_{\Omega}^*$, whence we have
$\gt_{k}^{\leq 1}(X) \leq q\gt_{k'}^{\leq 1}(X) \leq q||(i',j')||_{\Omega}^* = ||(i,j)||_{\Omega}^*$.

Now let $C$ be the (necessarily simple by the above) formal plane in $\tX$ with positive end $e_{i,j}$ and carrying the constraint $\lll \T^{(k)}p\rrr$. 
By an argument paralleling the proof of Proposition~\ref{prop:min_pert_inv}, we find that $C$ is formally perturbation invariant with respect to $J_{\bdy \tX}$.
In particular, $C$ cannot be represented by any nontrivial stable $J_{\tX}$-holomorphic building.
We claim that the signed count $\#\calM_{\tX}^{J_{\tX}}(C)$ is nonzero, from which it follows that we have $\gt_k^{\leq 1}(\tX) \leq E(C) = \tcalA(e_{i,j}) \approx ||(i,j)||_{\Omega}^*$.

To justify the claim, note that we can use Proposition~\ref{prop:T=E} to trade the local tangency constraint $\lll \T^{(k)}p\rrr$ for a skinny ellipsoidal constraint $\lll (k)\rrr_E$. 
Namely, letting $E_\sk = E(\eps,\eps x) \subset \tX$ denote an ellipsoid with $x > k$ and $\eps > 0$ sufficiently small, it suffices to show that the moduli space of pseudoholomorphic cylinders in $\tX \setminus E_\sk$ with positive end $e_{i,j}$ and negative end $\eta_k$ has nonzero signed count.
By slight abuse of notation we will denote the corresponding formal cylinder again by $C$. 
Recall that by Proposition~\ref{prop:count_positively} it suffices to show that this moduli space is nonempty.
For this we invoke linearized contact homology as in \cite{Pardcnct}, similar to the proof of \cite[Thm. 2]{hind2020j}.
Indeed, observe that $e_{i,j}$ is necessarily a cycle with respect to the linearized contact homology differential thanks to Lemma~\ref{lem:low_energy_cyls} and the fact that any orbit $h_{i',j'}$ with $(i',j') \neq (i,j)$ necessarily has greater action by Lemma~\ref{lem:fr_action_conds}.
Since the cobordism map on linearized contact homology induced by $\tX \setminus E_\sk$ is an isomorphism, it follows that there is a stable pseudoholomorphic cylindrical building representing $C$, and by formal perturbation invariance this must be an honest pseudoholomorphic cylinder in $\tX \setminus E_\sk$.
 \end{proof}
 
\begin{rmk}
As mentioned earlier, there is a natural higher dimensional analogue of the fully rounding procedure, but for concreteness we 
have kept our discussion in \S\ref{subsec:fr} to dimension four and hence  restrict 
Proposition~\ref{prop:GH} to dimension four. 
In order to extend the above argument to higher dimensions, one first ought to show that the higher dimensional the analogue of $C$ is formally perturbation invariant. Since the results in \cite{Pardcnct} hold in arbitrary dimension, one can then still invoke the cobordism map on linearized contact homology in higher dimensions in order to produce cylindrical buildings.

We also refer the reader to \cite[Thm 7.6.4]{Pereira_thesis} for the analogous statement $\gapac_k^{\leq 1}(X) = c_k^\op{GH}(X)$ for any Liouville domain $X$ satisfying $\pi(X) = 2c_1(TX) = 0$.
\end{rmk}

\begin{rmk}
We expect that the methods in this paper could be extended to compute $\gt_k^{\leq \ell}(X_\Omega)$ for all $k,\ell \in \Z_{\geq 1}$, and it is an interesting question whether the entire family $\{\gt_k^{\leq \ell}\}$ sometimes give stronger embedding obstructions than the sequence $\gt_1,\gt_2,\gt_3,\dots$ alone.
A natural guess is that Theorem~\ref{thm:main_comp} generalizes to a formula for $\gt_k^{\leq \ell}(X_\Omega)$ by requiring $q \leq \ell$ in the minimization.
\end{rmk}

\section{Ellipsoids, polydisks, and more}

In this section we apply our formalism to several examples, proving the remaining three theorems from the introduction.
In each case, using Theorem~\ref{thm:main_comp} and the specific form of $||-||_{\Omega}^*$, it reduces to a purely combinatorial optimization problem. The latter is tractable thanks to Corollary~\ref{cor:main_thm_4_cases}, which implies that we can look for a minimizer taking one of the following forms:
 \begin{enumerate}
 \item
 $(0,1)^{\times i} \times (1,1)^{\times j}$ for $i \geq 0$, $j \geq 1$
 \item 
 $(0,1)^{\times i} \times (1,s)$ for $i \geq 0$ and $s \geq 2$
 \item
 $(0,1)^{\times i} \times (1,0)$ for $i \geq 1$
 \item
 $(0,s)$ for $s \geq 1$.
 \end{enumerate}

\begin{proof}[Proof of Theorem~\ref{thm:ell_comp}]
We consider $E(a,1)$, and by continuity we can assume $a > 1$ is irrational. Let $\Omega$ be the triangle with vertices $(0,0),(a,0),(0,1)$.
Observe that for $\vec{v} = (v_x,v_y) \in \R_{\geq 0}^2$ we have 
$$||\vec{v}||_{\Omega}^* = \max_{\vec{w} \in \Omega}\langle \vec{v},\vec{w}\rangle = \max(v_xa,v_y).$$

We can 
 ignore case (3), 
since we have
\begin{align*}
||(1,2)||_{\Omega}^* = \max(a,2) < 1 + a    = ||(0,1)||_{\Omega}^* + ||(1,0)||_{\Omega}^*,
\end{align*}
and hence $(0,1)^{\times i} \times (1,0)$ with $i \geq 1$ cannot be a minimizer.

\sss

Suppose first that $a > 3/2$.
Then we have
\begin{align*}
||(1,2)||_{\Omega}^* + ||(0,1)||_{\Omega}^* = \max(a,2) + 1 < 2a = 2||(1,1)||_{\Omega}^*,
\end{align*}
and hence $(0,1)^{\times i} \times (1,1)^{\times j}$ with $i \geq 0$, $j \geq 1$ can only be a minimizer if $j = 1$.
If $s > a+1$, then we have
\begin{align*}
||(0,1)||_{\Omega}^* + ||(1,s-2)||_{\Omega}^* = 1 + \max(a,s-2) < \max(a,s)  < ||(1,s)||_{\Omega}^*,
\end{align*}
and therefore $(0,1)^{\times i} \times (1,s)$ with $i \geq 0$, $s \geq 2$ can only be a minimizer if $s \leq a+1$.
Similarly, if $s < a-1$ then we have
\begin{align*}
||(0,s+1)||_{\Omega}^* = s+1 < \max(a,s) = ||(1,s)||_{\Omega}^*
\end{align*}
and
\begin{align*}
||(1,s+2)||_{\Omega}^* = \max(a,s+2) < 1 + \max(a,s) = ||(0,1)||_{\Omega}^* + ||(1,s)||_{\Omega}^*,
\end{align*}
and therefore $(0,1)^{\times i} \times (1,s)$ with $i \geq 0$, $s \geq 1$ can only be a minimizer if $s \geq a-1$.

Since $a$ is irrational, we have $[a-1,a+1] \cap \Z = \{\lfloor a \rfloor, \lfloor a \rfloor + 1\}$.
Therefore, there must be a minimizer taking one of the following forms:
\begin{itemize}
\item $(0,1)^{\times i} \times (1,\lfloor a \rfloor)$ for $i \geq 0$
\item $(0,1)^{\times i} \times (1,\lfloor a \rfloor+1)$ for $i \geq 0$
\item $(0,s)$ for $s \geq 1$,
\end{itemize}
from which \eqref{eq:ellip_comp2} readily follows.

\sss

Now suppose that we have $a < 3/2$.
For $s \geq 3$ we have
\begin{align*}
||(1,1)||_{\Omega}^* + ||(0,s-2)||_{\Omega}^* = a + s-2 < s = ||(1,s)||_{\Omega}^*,
\end{align*}
and hence $(0,1)^{\times i} \times (1,s)$ with $i \geq 0$, $s \geq 3$ cannot be a minimizer.
We have also
\begin{align*}
2||(1,1)||_{\Omega}^* = 2a < 3 = 3||(0,1)||_{\Omega}^*,
\end{align*}
and hence $(0,1)^{\times i} \times (1,1)^{\times j}$ for $i \geq 0$, $j \geq 1$ can only be a minimizer if $i \in \{0,1,2\}$.

Therefore, there must be a minimizer taking one of the following forms:
\begin{itemize}
\item $(0,1)^{\times i} \times (1,1)^{\times j}$ for $i \in \{0,1,2\}$ and $j \geq 1$
\item $(0,1)^{\times i} \times (1,2)$ for $i \geq 0$
\item $(0,s)$ for $s \geq 1$.
\end{itemize}
Since $2||(0,1)||_{\Omega}^* = 2 = ||(1,2)||_{\Omega}^*$, we can effectively ignore the second bullet by artificially allowing $j = 0$ in the first bullet.
For $s \geq 2$ we have 
\begin{align*}
||(1,s-1)||_{\Omega}^* = \max(a,s-1) < s = ||(0,s)||_{\Omega}^*,
\end{align*}
and hence $(0,s)$ can only be minimal if $s = 1$, so we can also effectively ignore the third bullet.
Therefore we have 
\begin{align*}
\gt_k(E(a,1)) = i||(0,1)||_{\Omega}^* + j||(1,1)||_{\Omega}^* = i + ja,
\end{align*}
for $i \in \{0,1,2\}$, $j \geq 0$ satisfying $2i + 3j - 1 = k$.
Note that $i$ and $j$ are uniquely determined via $i \equiv -k-1\;(\text{mod}\; 3)$ and $j = \frac{j+1-2i}{3}$, and \eqref{eq:ellip_comp3} follows.
\end{proof}

\begin{proof}[Proof of Theorem~\ref{thm:poly_comp}]

This is similar to the previous proof. We consider $P(a,1)$ with $a > 1$ irrational, and we take $\Omega$ to be the rectangle with vertices $(0,0),(a,0),(0,1),(a,1)$.
For $\vec{v} = (v_x,v_x) \in \R_{\geq 0}^2$ we then have
\begin{align*}
||\vec{v}||_{\Omega}^* = \langle \vec{v},(a,1)\rangle = a v_x + v_y.
\end{align*}

For $s \geq 2$ we have
\begin{align*}
||(1,0)||_{\Omega}^* + ||(0,s-1)||_{\Omega}^* = a + s-1 < a + s = ||(1,s)||_{\Omega}^*,
\end{align*}
and hence case (2) in Corollary~\ref{cor:main_thm_4_cases} cannot occur as a minimizer.
For $i \geq 0$, $j \geq 1$ we have
\begin{align*}
2||(0,1)||_{\Omega}^* + ||(1,0)||_{\Omega}^* = 2 + a < 2 + 2a = 2||(1,1)||_{\Omega}^*,
\end{align*}
so case (1) can only occur if $j = 1$. Therefore, there must be a minimizer from the following list:
\begin{itemize}
\item
$(0,1)^{\times i} \times (1,1)$ for $i \geq 0$
\item
$(0,1)^{\times i} \times (1,0)$ for $i \geq 1$
\item
$(0,s)$ for $s \geq 1$,
\end{itemize}
from which \eqref{eq:poly_formula} follows.
\end{proof}

\begin{proof}[Proof of Theorem~\ref{thm:quad}]

The polygon $\Om := Q(a,b,c) \subset \R_{\geq 0}^2$ has vertices $(0,0),(c,0),(a,b),(0,1)$.
For $\vec{v} = (v_x,v_y) \in \R_{\geq 0}^2$, we have
\begin{align*}
||\vec{v}||_{\Omega}^* = \max_{\vec{w} \in \Omega}\langle \vec{v},\vec{w}\rangle = \max(cv_x,av_x+bv_y,v_y).
\end{align*}
Recall that by assumption we have $c \geq 1$, $a \leq c$, $b \leq 1$, $a + bc \geq c$, and $M := \max(a+b,c) \leq 2$.

For $j \geq 1$, we have
\begin{align*}
||(1,j)||_{\Omega}^* = \max(c,a+jb,j),
\end{align*}
and in particular
\begin{align*}
||(1,1)||_{\Omega}^* = \max(c,a+b,1) = \max(c,a+b) = M.
\end{align*}

By the above, we have
$$
\gt_2(X) = \min(||(1,1)||_{\Omega}^*, ||(0,2)||_{\Omega}^*) =  \min(M,2) = M.
$$
Next, because  $c<2$, we have $$
||(0,3)||_{\Omega}^*> 3 > ||(1,0)||_{\Omega}^*
+ ||(0,1)||_{\Omega}^*= 1+c
$$
so that
\begin{align*}
\gt_3(X) &= \min\bigl(||(1,2)||_{\Omega}^*, ||(1,0)||_{\Omega}^*
+ ||(0,1)||_{\Omega}^*\bigr) \\
& =
 \min\bigl(\max(2, a+2b,c), 1+c\bigr) \\ &=
 \min\bigl(\max(2, a+2b), 1+c\bigr).
\end{align*}
Note that $\gt_3(X) =2$  if $a+2b<2$ and otherwise $= \min(a+2b, 1+c) < 3$.
In particular, if $a+2b>2$ the minimum could be represented by either orbit set. 

We next claim that 
$$
||(1,j)||_{\Omega}^* > ||(0,1)||_{\Omega}^* + ||(1,j-2)||_{\Omega}^*, \quad j\ge 3
$$
If $b>1/2$ we must check that
$$
\max(j, a + jb) > 1 + \max(j-2, a+(j-2)b) = \max(j-1, a+jb-2b+1),$$ which holds because $2b>1$.
If $b<1/2$ and $j\ge 3$ then $a+jb < 2 + \frac{j-1}2 \le j$ for $j\ge 3$, so that $$
||(1,j)||_{\Omega}^*= j > ||(1,j-2)||_{\Omega}^* + ||(0,1)||_{\Omega}^*, \quad j\ge 3.
$$
Thus in all cases $(1,j), j\ge 3,$ does not occur in a minimal orbit set.
Further $(0,k), k\ge 2,$ is never minimal since it can be replaced by $(1,1)\cup (0,1)^{\times k/2}$ 
for even $k$ or $(1,0)\cup (0,1)^{\times (k-1)/2}$  for odd $k$.

Therefore, taking into account the discussion of $\gt_3$,  we find that
 minimizers must take one of the following forms
\begin{itemize}\item
$(0,1)^{\times i}\times (1,1)^{\times j}$ where $j=0$ only if $i=1$;
\item
$(0,1)^{\times i}\times (1,2) $ or $(0,1)^{\times i}\times (1,0)$ (but not both).
\end{itemize}
In particular, 
\begin{align*}
\gt_4(X) &= ||(0,1)||_{\Omega}^* + ||(1,1)||_{\Omega}^* = 1 + M < 3,
\end{align*}
and 
\begin{align*}
\gt_6(X) &= ||(0,1)^{\times 2}||_{\Omega}^* + ||(1,1)||_{\Omega}^* = 2 + M < 4.
\end{align*}
On the other hand,  $\gt_5(X)$ might be represented by $(0,1)\times (1,2), (0,1)^{\times 2}\times (1,0)$ or $(1,1)\times (1,1)$ and so is given by
\begin{align*}
\gt_5(X) &=  \min\bigl( \max(3, 1 + a + 2b, c), 2 + c, 2 M\bigr).
\end{align*}
If $M<3/2$, then because the first two terms above are $\ge 3$, we find that
$ \gt_5(X)= 2 M$.   However, if $3/2< M<2$ then any of these three terms might be minimal.

For $k> 6$ it is again useful to consider the cases $M<3/2$ and $M>3/2$ separately.  In the former case, it is more efficient to increase the index by adding copies of $(1,1)$ so that minimal orbit sets always have $i\le 2$.
In particular,  orbit sets of the form $(0,1)^{\times i}\times (1,2) $ or $(0,1)^{\times i}\times (1,0)$ are not minimal when $i>2$, and so can only affect the capacities   $\gt_k$ for $k\le 7$.  Moreover when $M<3/2$
\begin{align*}
||(0,1)^{\times 2}\times (1,2)||_{\Omega}^* & = 2 + \max(2,a+2b,c) \\
& >\  4 \  > 1+2M
= ||(0,1)\times (1,1)^{\times 2}||_{\Omega}^*.
\end{align*}
Therefore the capacities for $k\ge 6$ are given by the orbit sets 
 $$
(0,1)^{\times 2}\times (1,1)^{\times j}, \ (0,1)\times  (1,1)^{\times j+1},\ (1,1)^{\times j+2},\quad  j\ge 1,\; M<3/2.
$$
The claims in  (i) follow readily.
\MS

If $M>3/2$,  minimal orbit sets always have $j \le 2$ since it is more efficient to use
$(0,1)^{\times 3}\times (1,1)^{\times j-2} $ instead of $(1,1)^{\times j}$.  
 Which of  $(0,1)^{\times i}\times (1,2) $ or $(0,1)^{\times i}\times (1,0)$
is more efficient  is determined by the value of $\gt_3$, while the value of $\gt_5$ 
determines whether it is in fact best to use $(1,1)^{\times 2}$ when representing elements of odd index $\ge 5$.   Thus the odd capacities for $k\ge 5$ are determined by $\gt_5$, while the even capacities
are more straightforward 
 since they are always calculated by orbit sets of the form $(0,1)^{\times i}\times (1,1)$.  This proves (ii).
\end{proof}

\begin{rmk}\label{rmk:abc}    
When $2\le n< c<n+1$ one can check that $\gt_k = k$ for $k\le n$, represented by the
orbit $e_{0,k}$.  In this case, the $\gt_k$ again limit on a period two cycle.  However, the precise values in this cycle depend on $b$.  To see this, note for example that if  $n= 2\ell$ is even, then  
$$
\gt_{k+1} = \min_{i\le \ell} \A(e_{0,1}^i\cup e_{1,2(\ell-i)}) = \min_{i\le \ell}\bigl(i + \max(2(\ell-i),\ a + 2(\ell-i)b,\ c)\bigr),
$$
and which orbit set gives the minimax depends on whether $b>1/2$ or $b< 1/2$.  For example, if we assume that $a<c$ are both very close to $n$ then the minimax is determined by the minimum  value of 
$i+ a + 2(\ell-i)b = a+2\ell b + i(1-2b)$.  Thus   if $ b<1/2$ one should take $i = \ell$, while if $b < 1/2$ one should take $i=0$.
\end{rmk}

\begin{proof}[Proof of Theorem~\ref{thm:Lp}]

For $\Omega := \Omega_p$, recall that we have $\calA_{\Omega}(e_{i,j}) = ||(i,j)||_{\Omega_p}^* = ||(i,j)||_q$. 
We can ignore case (4) in Corollary~\ref{cor:weak_perm} for $s \geq 2$, since we have
for $s = 2$ we have $$||(1,1)||_q = 2^{1/q} \leq 2 = ||(0,2)||_q$$ and for $s \geq 3$ we have
$$||(1,0)||_q + ||(0,s-2)||_q = s-1 \leq s = ||(0,s)||_q.$$
Similarly, we can ignore case (2), since we have 
$$||(1,0)||_q + ||(0,s-1)||_q = s < ||(1,s)||_q.$$
Noting that $||(0,1)||_q = ||(1,0)||q$, we can also effectively ignore case (3) by relaxing the condition $j \geq 1$ in case (1). In other words, we have that
$\gt_k(X_{\Omega_p})$ is the minimal quantity of the form $$i||(0,1)||_q + j ||(1,1)||_q = i + j2^{1/q},$$ subject to $2i + 3j -1 = k$ for $i,j \in \Z_{\geq 0}$.

We have $2||(1,1)||_q \leq 3||(0,1)||_q$ if and only if $2^{1/q} \leq 3/2$, i.e. $q \geq \tfrac{\ln(2)}{\ln(3/2)}$, or equivalently $p \leq \frac{\ln(2)}{\ln(4/3)}$.
In this case we can assume $i \in \{0,1,2\}$, and the value of $i$ is then determined by looking at the equation $2i + 3j - 1 = k$ modulo $3$, from which \eqref{eq:Lp_1} immediately follows.
Similarly, in the case $p > \frac{\ln(2)}{\ln(4/3)}$ we can assume $j \in \{0,1\}$, and the value of $j$ is then determined by looking at the equation $2i + 3j - 1 = k$ modulo $2$, which immediately gives \eqref{eq:Lp_2}.
\end{proof}

\appendix

\section{Regularity after stabilization}

\counterwithin{thm}{section}

In this appendix we give a self-contained proof that regularity persists after dimensional stabilization. We also refer the reader to \cite[\S7.4]{Pereira_thesis} for a related approach.

Let $X$ be a Liouville domain, and let $W := X \smx B^2(c)$ be a smoothing of $X \times B^2(c)$ for some $c > 0$ as in Lemma~\ref{lem:smoothing}.
Let $D$ be a local symplectic divisor in $X$ near a point $p \in X$, and let $\wt{D} = D \times B^2(\eps)$ for $\eps > 0$ small be a corresponding local symplectic divisor in $W$ near $\wt{p} := (p,p_0)$ for $p_0 := 0 \in B^2(c)$.
Let $J$ be an admissible almost complex structure on $\wh{X}$ which is integrable near $p$ and preserves $D$, and let $\wt{J}$ be an admissible almost complex structure on $\wh{W}$ which is integrable near $\wt{p}$, preserves $\wt{D}$, and restricts to $J$ along $\wt{X} \times \{0\} \approx \wt{X}$ (so that in particular $\wh{X} \times \{0\}$ is $\wt{J}$-holomorphic).

Our main goal is to prove:
\begin{prop}\label{prop:reg_after_stab}
Let $u$ be an asymptotically cylindrical $J$-holomorphic punctured sphere in $\wh{X}$ satisfying the constraint $\lll \T_D^{(m)}p\rrr$ for some $m \in \Z_{\geq 1}$, and such that each asymptotic Reeb orbit is nondegenerate with normal Conley--Zehnder index $1$.
Assume that $u$ is regular and has index zero (taking into account the constraint $\lll \T_D^{(m)}p\rrr$).
Let $\wt{u}$ denote the curve in $\wh{W}$ given by the composition of $u$ with the inclusion $\wh{X} \subset \wh{W}$.
Then $\wt{u}$ is also regular (taking into account the constraint $\lll \T_\tD^{(m)}p\rrr$).
\end{prop}
Note that in formulating the index and regularity of $u$ and $\tu$ we are as usual also allowing for arbitrary variations of the conformal structure of the domain.
Recall that the normal Conley--Zehnder index is defined for a Reeb orbit in $\bdy X$ by taking into account the Reeb flow in the direction normal to $\wh{X} \times \{0\}$ in $\wh{W}$, and we are implicitly using trivializations coming from the natural trivialization of the normal bundle of $\wh{X} \subset \wh{W}$ as in \S\ref{subsec:stab_lb}.

Let $\Sigma = S^2 \setminus \{z_1,\dots,z_\ell\}$ denote the domain of $u$, where $z_1,\dots,z_\ell$ are the punctures, and let $z_0 \in \Sigma$ denote the marked point which realizes the local tangency constraint.
Regularity of $u$ is equivalent to surjectivity of linearized Cauchy--Riemann operator
\begin{align*}
D\delbar_J(u,j): T_u\calB \oplus T_j\calT \rightarrow \calE_{(u,j)},
\end{align*}
where:
\begin{itemize}
  \item $T_u\calB = \calW_{\ssst \lll \T_D^{(m)}p\rrr}^{k,p,\delta}(u^*T\wh{X}) \oplus V$ 

  \item $\calW^{k,p,\delta}(u^*T\wh{X})$ denotes the Banach space of sections $\xi$ of $u^*T\wh{X}$ of weighted Sobolev class $\calW^{k,p,\delta}$ (c.f. \cite[\S7.2]{wendl_SFT_notes}), where we assume $k \geq m$ and $(k-m)p > 2$ (so that $\xi$ is $C^m$), and $\calW_{\ssst \lll \T_D^{(m)}p\rrr}^{k,p,\delta}(u^*T\wh{X}) \subset \calW^{k,p,\delta}(u^*T\wh{X})$
  denotes the subspace consisting of sections whose $m$-jet at $z_0$ lies in $D$ as in \cite[\S 6]{CM1} (in particular $\calW^{k,p,\delta}_{\lll p \rrr}(u^*T\wh{X})$ is the subspace such that $\xi$ vanishes at $z_0$)

\item $V \subset \calW^{k,p}_{\op{loc}}(u^*T\wh{X})$ is a $2\ell$-dimensional subspace as in \cite[\S3.1]{Wendl_aut}, consisting of smooth sections which are supported near the punctures and asymptotic to constant (in suitable trivializations) linear combinations of vector fields tangent to the trivial cylinders over the asymptotic Reeb orbits of $u$
(this is needed to the possibility of rotating and translating the asymptotic ends of $u$, as these deformations do not exponentially decay along the cylindrical ends)

  \item $\calT \subset \calJ(\Sigma)$ is a Teichm\"uller slice through $j$ as in \cite[\S3.1]{Wendl_aut},
which is in particular a smooth manifold containing $j$ and having (in the stable case) dimension $2(\ell+1)-6$, and 
 $T_j\calT \subset \Gamma(\ovl{\op{End}}_\C(T\Sigma))$ denotes its tangent space at $j$

  \item $\calE_{(u,j)} = \calW_{\ssst \lll \T_D^{(m-1)}p\rrr}^{k-1,p,\delta}(\ovl{\hom}_\C(T\Sigma,u^*T\wh{W}))$ consists of bundle homomorphisms from $T\Sigma$ to $u^*T\wh{W}$ over $\Sigma$ which are $(j,J)$-antilinear and whose $(m-1)$ jet at $z_0$ lies in $D$

\end{itemize}
Moreover, after choosing any symmetric connection $\nabla$ on $T\wh{X}$,
for $\xi \in T_u\calB$ and $y \in T_j\calT$, the linearized Cauchy--Riemann operator $D\delbar_J(u,j)$ takes the explicit form
\begin{align*}
D\delbar_J(u,j)(\xi,y) = D_u\xi + G_u y,
\end{align*}
where:
\begin{itemize}
  \item $D_u: T_u\calB \rightarrow \calE_{(u,j)}$ is given by \[D_u\xi = \nabla \xi + J \circ (\nabla \xi) \circ j + \nabla_{\xi} J \circ du \circ j\]
  \item $G_u: T_j\calT \rightarrow \calE_{(u,j)}$ is given by \[G_u y =  J \circ du \circ y.\]
\end{itemize}

Similarly, regularity of $\wt{u}$ is equivalent to surjectivity of the operator
\begin{align*}
D\delbar_J(\wt{u},j): T_{\wt{u}}\wt{\calB} \oplus T_j\calT \rightarrow \calE_{(\wt{u},j)},
\end{align*}
where:
\begin{itemize}
  \item $T_\tu\wt{\calB} = \calW_{\ssst \lll \T_{\tD}^{(m)}\wt{p}\rrr}^{k,p,\delta}(\wt{u}^*T\wh{W}) \oplus V$

  \item $\calE_{(\wt{u},j)} = \calW^{k-1,p,\delta}_{\ssst \lll \T_\tD^{(m-1)}p\rrr}(\ovl{\hom}_\C(T\Sigma,\tu^*T\wh{W}))$,
\end{itemize}
and for $\xi \in T_\tu\wt{\calB}$ and $y \in T_j\calT$ we have
\begin{align*}
D\delbar_J(\wt{u},j)(\xi,y) = D_{\wt{u}}\xi + G_{\wt{u}}y
\end{align*}
where:
\begin{itemize}
  \item $D_\tu: T_\tu\wt{\calB} \rightarrow \calE_{(\tu,j)}$ is given by \[D_{\wt{u}}\xi = \tnabla \xi + \wt{J} \circ (\tnabla \xi) \circ j + \tnabla_{\xi} 
  \wt{J} \circ d\tu \circ j\]
  \item $G_\tu: T_j\calT \rightarrow \calE_{(\tu,j)}$ is given by \[G_{\tu} y =  \wt{J} \circ d\tu \circ y,\]
\end{itemize}
where $\wt{\nabla}$ is any symmetric connection on $T\wh{W}$.

Note that the embedding $W \hookrightarrow X \times B^2(c)$ naturally extends to a diffeomorphism $\wh{W} \cong \wh{X} \times \wh{B^2(c)}$, and we get a corresponding splitting of the tangent bundle of $\wh{W}$:
\begin{align*}
T\wh{W} \cong T^\ver \wh{W} \oplus T^\hor \wh{W}.
\end{align*}
Under the identification $T^\ver\wh{W}|_{\wh{X} \times \{0\}} \approx T\wh{X}$,
this induces natural splittings

\begin{align*}
T_\tu\wt{\calB} \oplus T_j\calT \cong \underbrace{\left(\calW_{\ssst \lll \T_D^{(m)}p\rrr}^{k,p,\delta}(u^*T\wh{X}) \oplus V \oplus T_j\calT\right)}_{A_1} \oplus \underbrace{\left(\calW_{ \lll p_0 \rrr}^{k,p,\delta}(\tu^*T^\hor\wh{W})\right)}_{A_2}
\end{align*}
and 
\begin{align*}
\calE_{(\wt{u},j)} \cong \underbrace{\left(\calW^{k-1,p,\delta}_{\ssst \lll \T_D^{(m-1)}p\rrr}(\ovl{\hom}_\C(T\Sigma,u^*T\wh{X}))\right)}_{B_1} \oplus \underbrace{\left(\calW^{k-1,p,\delta}(\ovl{\hom}_\C(T\Sigma,\tu^*T^\hor\wh{W}))\right)}_{B_2}.
\end{align*}

From now on, we assume that the connection $\tnabla$ preserves this splitting and restricts to $\nabla$ under the identification $T\wh{X}$.
The above splitting induces a block matrix decomposition
\begin{align}\label{eq:block_decomp}
D\delbar_J(\tu,j) = 
\begin{pmatrix}
  M_{1,1} = D\delbar_J(u,j) & M_{1,2} \\ M_{2,1} & M_{2,2}
\end{pmatrix}.
\end{align}

\begin{lemma}\label{lem:M_2_1}
We have $M_{2,1} = 0$.  
\end{lemma}
\begin{proof}

We need to show that the image of $D\delbar_J(\tu,j)|_{A_1}$ lies in $B_1$.
Note that for $y \in T_j\calT$ we have
\begin{align*}
\wt{J} \circ d\tu \circ y \in \Gamma(\ovl{\hom}_\C (T\Sigma,\tu^* T\wh{X})),
\end{align*}
since $\wt{J}$ preserves $T^\ver \wh{W}|_{\wh{X}}$,
and hence $G_\tu y \in B_1$.
It therefore suffices to show that for any $\xi^\ver \in \Gamma(\tu^*T^\ver \wh{W})$, 
$D_\tu \xi^\ver$
 lands in $\Gamma(\ovl{\hom}_\C(T\Sigma,\tu^*T^\ver\wh{W}))$.
For $v \in \Gamma(T\Sigma)$, we have
\begin{align*}
(D_\tu\xi^\ver)(v) = \tnabla_v \xi^\ver + \wt{J} \tnabla_{jv}\xi^\ver + (\tnabla_{\xi^\ver}\wt{J})(q) 
\end{align*}
for $q := (d\tu)(jv) \in \Gamma(u^*T\wh{X})$.
Since $\tnabla$ and $\wt{J}$ respect the splitting $T\wh{W}|_{\wh{X}} = T^\ver \wh{W}|_{\wh{X}} \oplus T^\hor \wh{W}|_{\wh{X}}$, we have 
\begin{align*}
\tnabla_v \xi^\ver + \wt{J}\tnabla_{jv}\xi^\ver \in \Gamma(\tu^*T^\ver\wh{W}).
\end{align*}

Therefore it remains to show that $(\tnabla_{\xi^\ver}\wt{J})(q) \in \Gamma(\tu^*T^\ver \wh{W})$.
For this, it suffices to establish
\begin{align*}
(\tnabla_a \wt{J})(b) \in \Gamma(T^\ver\wh{W}|_{\wh{X}})
\end{align*}
for any $a,b \in \Gamma(T^\ver\wh{W}|_{\wh{X}})$.
Recall that the term $\tnabla_a\wt{J} \in \op{End}(T\wh{W}|_{\wh{X}})$ corresponds to applying the connection induced by $\tnabla$ (which we again denote by $\tnabla$) on the endomorphism bundle, and by its definition we have
\begin{align*}
(\tnabla_a \wt{J})(b) = \tnabla_a(\wt{J}b) - \wt{J}(\tnabla_a b).
\end{align*}
Similar to above, it is immediate that these last two terms lie in $\Gamma(T^\ver \wh{W}|_{\wh{X}})$.
 \end{proof}

\begin{lemma}\label{lem:M_2_2}
  The operator $M_{2,2}$ is surjective.
\end{lemma}
\begin{proof}
If we ignore the constraint $\lll p_0\rrr$,
the corresponding ($\R$-linear) Cauchy--Riemann type operator \[\calW^{k,p,\delta}(\tu^*T^\hor \wh{W}) \rightarrow \calW^{k-1,p,\delta}(\ovl{\hom}_\C(T\Sigma,\tu^*T^\hor\wh{W}))\]
is Fredholm, and by a version of Riemann--Roch with its index is easily computed to be $2$ (see e.g. \cite[\S2.1]{Wendl_aut}).
It follows that $M_{2,2}$ is also Fredholm, with index $0$, and hence to prove its surjectivity it suffices to establish $\ker M_{2,2} = \{0\}$.
Suppose by contradiction that $\eta$ is a nonzero element in $\ker M_{2,2}$. 
By elliptic regularity we can assume that $\eta$ is smooth, and its count $Z(\eta)$ of zeros is nonnegative (this follows by the similarity principle \cite[Thm. 2.32]{wendl_SFT_notes}), and in fact strictly positive since $\eta$ necessarily vanishes at the marked point $z_0$.
On the other hand, in the notation of \cite[\S2.1]{Wendl_aut}, each puncture $z_i$ of $\tu$ has normal Conley--Zehnder index $1$ and hence extremal winding number $\alpha_-(\mathbb{A}_{z_i}) = 0$,
 and therefore using \cite[Eq. 2.7]{Wendl_aut} we have
\[1 \leq Z(\eta) + Z_\infty(\eta) = c_1(\tu^*T^\hor\wh{W}) + \sum_{i=1}^\ell \alpha_-(\mathbb{A}_{z_i}) = 0, \]
a contradiction.
\end{proof}

\begin{proof}[Proof of Proposition~\ref{prop:reg_after_stab}]
  This follows immediately from the decomposition \eqref{eq:block_decomp} and Lemmas \ref{lem:M_2_1} and \ref{lem:M_2_2}.
\end{proof}

Now suppose that $J$ is an admissible almost complex structure on the symplectization of $\bdy X$,
and let $\wt{J}$ be an admissible almost complex structure on the symplectization of $\bdy W$ which restricts to $J$ on $\R \times (\bdy X \times \{0\})$.
An argument nearly identical to the above proves:
\begin{prop}\label{prop:reg_after_stab_symp}
Let $u$ be an asymptotically cylindrical $J$-holomorphic punctured sphere in $\R \times \bdy X$, such that each asymptotic Reeb orbit is nondegenerate with normal Conley--Zehnder index $1$.
Assume that $u$ is regular and has index zero.
Let $\wt{u}$ denote the curve given by the composition of $u$ with the inclusion $\R \times \bdy X \subset \R \times \bdy W$.
Then $\wt{u}$ is also regular.
\end{prop}

\bibliographystyle{math}
\bibliography{biblio}

\end{document}